\documentclass[sort & compress,times,3p]{elsarticle}
\usepackage[normalem]{ulem}
\usepackage{lineno}
\usepackage{xcolor}
\usepackage{soul}
\colorlet{siaminlinkcolor}{green!50!black}
\colorlet{siamexlinkcolor}{red!50!black}
\modulolinenumbers[5]

\journal{}

\usepackage{amsfonts,amsmath,amsthm}
\usepackage{graphicx,float}
\usepackage{epstopdf}
\usepackage{algorithmic}
\usepackage{enumerate,enumitem,amssymb,array,multirow,mathabx}
\usepackage[caption=false]{subfig}

\newcommand{\R}{\mathbb{R}}
\newcommand{\C}{\mathbb{C}}
\newcommand{\vx}{\mathbf{x}}
\newcommand{\dx}{\mathrm{d}\vx}
\newcommand{\rmd}{\mathrm{d}}

\newcommand{\vep}{\varepsilon}

\newcommand{\lrang}[2]{\langle #1, #2 \rangle}

\newcommand{\hly}[1]{\sethlcolor{yellow}\hl{#1}}

\usepackage{amsopn}

\DeclareMathOperator*{\argmin}{argmin}
\DeclareMathOperator{\sech}{sech}

\makeatletter
\newcommand{\thickhline}{%
	\noalign {\ifnum 0=`}\fi \hrule height 1.2pt
	\futurelet \reserved@a \@xhline
}
\newcolumntype{"}{!{\vrule width 1.2pt}}
\makeatother

\usepackage{hyperref}
\hypersetup{
	colorlinks = true
}

\usepackage[capitalize,nameinlink,noabbrev]{cleveref}

\theoremstyle{plain}
\newtheorem{theorem}{Theorem}[section]
\newtheorem{lemma}[theorem]{Lemma}
\newtheorem{proposition}[theorem]{Proposition}

\theoremstyle{definition}
\newtheorem{definition}[theorem]{Definition}

\theoremstyle{remark}
\newtheorem{remark}[theorem]{Remark}

\newtheorem{conj}[theorem]{Conjecture}
\newtheorem{exmp}[theorem]{Example}

\numberwithin{equation}{section}
\numberwithin{figure}{section}
\numberwithin{table}{section}

\crefname{exmp}{Example}{Examples}
\crefname{hypothesis}{Hypothesis}{Hypotheses}
\crefname{conj}{Conjecture}{Conjectures}
\crefformat{equation}{\textup{#2(#1)#3}}
\crefrangeformat{equation}{\textup{#3(#1)#4--#5(#2)#6}}
\crefmultiformat{equation}{\textup{#2(#1)#3}}{ and \textup{#2(#1)#3}}
{, \textup{#2(#1)#3}}{, and \textup{#2(#1)#3}}
\crefrangemultiformat{equation}{\textup{#3(#1)#4--#5(#2)#6}}%
{ and \textup{#3(#1)#4--#5(#2)#6}}{, \textup{#3(#1)#4--#5(#2)#6}}{, and \textup{#3(#1)#4--#5(#2)#6}}

\Crefformat{equation}{#2Equation~\textup{(#1)}#3}
\Crefrangeformat{equation}{Equations~\textup{#3(#1)#4--#5(#2)#6}}
\Crefmultiformat{equation}{Equations~\textup{#2(#1)#3}}{ and \textup{#2(#1)#3}}
{, \textup{#2(#1)#3}}{, and \textup{#2(#1)#3}}
\Crefrangemultiformat{equation}{Equations~\textup{#3(#1)#4--#5(#2)#6}}%
{ and \textup{#3(#1)#4--#5(#2)#6}}{, \textup{#3(#1)#4--#5(#2)#6}}{, and \textup{#3(#1)#4--#5(#2)#6}}
\begin{document}
	\hypersetup{
		allcolors = siaminlinkcolor,
		urlcolor = siamexlinkcolor,
	}
	\begin{frontmatter}
		
		\title{Computing the least action ground state of the nonlinear Schr\"odinger equation by a normalized gradient flow}
		
		\author[1]{Chushan Wang\corref{cor1}}
		\cortext[cor1]{Tel.: }
		\ead{E0546091@u.nus.edu}
		
		\address[1]{Department of Mathematics, National University of Singapore, Singapore 119076, Singapore}
		
		\begin{abstract}
			In this paper, we generalize the normalized gradient flow method which was first applied to computing the least energy ground state to compute the least action ground state. A continuous normalized gradient flow (CNGF) will be presented and the action diminishing property will be proved to provide a mathematical justification of the gradient flow with discrete normalization (GFDN). Then we use backward-forward Euler method to further discretize the GFDN in time which leads to the GFDN-BF scheme. It is shown that the GFDN-BF scheme preserves the positivity and diminishes the action unconditionally. We compare it with other three schemes which are modified from corresponding ones designed for the least energy ground state and the numerical results show that the GFDN-BF scheme performs much better than the others in accuracy, efficiency and robustness for large time steps. Extensive numerical results of least action ground states for several types of potentials are provided. We also use our numerical results to verify some existing results and lead to some conjectures. 
		\end{abstract}
		
		
	\end{frontmatter}
	
	
	\section{Introduction}
	We consider the nonlinear Schr\"odinger equation (NLSE) with focusing nonlinearity
	\begin{equation}\label{NLSE}
		i \partial_t \psi = - \Delta \psi + V(\vx)\psi - \beta |\psi|^{2\alpha} \psi, \quad \vx \in \R^d, \quad t > 0, 
	\end{equation}
	where $ \psi = \psi(\vx, t) $ is a complex-valued function, $ V(\vx) $ is a real-valued potential and
	\begin{equation*}
		\beta > 0, \quad 0<\alpha<\frac{2}{(d-2)_{+}} \text{ with } (d-2)_+ = \max\{0, d-2\}. 
	\end{equation*}
	Without loss of generality, we fix $ \beta = 1 $. The assumptions on the potential $ V $ will be made clear in \cref{sec:2} with some examples shown in \cref{exmp:poten}. For the nonlinearity, $ \alpha = 2/d $ is called $ L^2 $-critical since the scaling $ \psi(\vx, t) \rightarrow \lambda^\frac{1}{\alpha} \psi(\lambda \vx, \lambda^2 t) $ preserves the $ L^2 $ norm while $ 0<\alpha<2/d $ and $ \alpha>2/d $ are called $ L^2 $-subcritical and $ L^2 $-supercritical respectively. Although we only consider the homogeneous power-type nonlinearity, it is possible to generalize our method to more complex nonlinearities (see \cref{rem:generalization}).
	
	The NLSE \cref{NLSE} with our assumptions on the potential function includes many important models. When $ V(\vx) \equiv 0 $, it arises in various physical contexts such as nonlinear optics and plasma physics \cite{coz2009}. When $ V $ is chosen as Dirac's delta function, it can be used to model a Bose-Einstein Condensation (BEC) in the presence of point defects or impurities \cite{adami2021,fukuizumi2008}. Besides, recently, it has attracted a lot of attention when $ V $ is an inverse power potential (see \cite{dinh2020, fukaya2021, fukuizumi2003} and references therein). 
	
	Formally, the NLSE \cref{NLSE} conserves the mass
	\begin{equation}
		M(\psi(\cdot, t)) = \int_{\R^d} |\psi(\vx, t)|^2 \dx \equiv M(\psi(\cdot, 0)), \quad t \geq 0, 
	\end{equation}
	and the energy
	\begin{equation}
		\begin{aligned}
			E(\psi(\cdot, t)) 
			&= \int_{\R^d} \left[ |\nabla \psi(\vx, t)|^2 + V(\vx) |\psi(\vx, t)|^2 - \frac{1}{\alpha+1} |\psi(\vx, t)|^{2\alpha+2} \right] \dx \equiv E(\psi(\cdot, 0)), \quad t \geq 0.
		\end{aligned} 
	\end{equation}
	
	In some applications, people are interested in standing wave solutions of the NLSE \cref{NLSE}, which take the form
	\begin{equation}
		\psi(\vx, t) = e^{i \omega t} \phi(\vx), \quad \vx \in \R^d, 
	\end{equation}
	where $ \omega \in \R $ and $ \phi $ is a real/complex-valued function independent of time. Inserting it into \cref{NLSE}, we see that $ \phi $ satisfies the stationary nonlinear Schr\"odinger equation
	\begin{equation}\label{SNLS}
		-\Delta \phi + V \phi + \omega \phi - |\phi|^{2\alpha} \phi  = 0. 
	\end{equation}
	There are two classes of solutions to \cref{SNLS} that are of particular interest. One is the least energy ground state and the other is the least action ground state. We shall briefly review each of them. 
	
	When $ 0 < \alpha \leq 2/d $ (i.e. the $ L^2 $-critical and -subcritical regime), the \emph{least energy ground state} $ \phi^{E}_m $ with mass $ m \in \R^+ $ is defined as the $ L^2 $ constrained energy minimizer:
	\begin{equation}\label{prob:energy ground state}
		\phi_m^{E} := \argmin_{\substack{\phi \in H^1(\R^d) \\ M(\phi) = m}} E(\phi). 
	\end{equation}
	Then \cref{prob:energy ground state} is a nonconvex minimization problem, and the associated Euler-Lagrangian equation reads	
	\begin{equation}\label{eq:evp}
		-\Delta \phi + V \phi - |\phi|^{2\alpha} \phi  = \mu^g \phi, 
	\end{equation}
	which is a nonlinear eigenvalue problem for $ (\mu^g, \phi) $ under the constraint $ M(\phi) = m $. The eigenvalue $ \mu^g $ can be computed from its corresponding eigenfunction $ \phi $ by
	\begin{equation}\label{mu}
		\mu^g = \mu^g(\phi) = \frac{\int_{\R^d} \left[ |\nabla \phi(\vx)|^2 + V(\vx) |\phi(\vx)|^2 - |\phi(\vx)|^{2\alpha+2} \right] \dx}{m}. 
	\end{equation}
	Then we see that any solution to the least energy ground state \cref{prob:energy ground state} with corresponding eigenvalue $ \mu^g $ will solve \cref{SNLS} with $ \omega = -\mu^g $. 
	
	We recall some results about the existence of the least energy ground state. It is well-known that when $ V(\vx) \equiv 0 $, the least energy ground state exists for any $ m > 0 $ when $ 0<\alpha<2/d $ \cite{cazenave_lions}. More results under very general conditions on the nonlinear term can be found in \cite{jeanjean2021}. When $ V(\vx) \not \equiv 0 $, the existence of least energy ground state is discussed when $ 0 \not \equiv V(\vx) \leq 0 $ is continuous and satisfies $ \lim_{|\vx| \rightarrow \infty} V(\vx) = 0 $ \cite{ikoma2020}. When $ V = - \gamma/|x|^\sigma \ (\gamma>0, 0<\sigma<\min\{2, d\}) $ is an attractive inverse power potential, it is known that the least energy ground state exists when $ 0<\alpha<2/d $ for any $ m > 0 $ and when $ \alpha = 2/d $ for any $ 0<m<m_{\text{c}} $ with $ m_c $ being some critical value \cite{li2020}. In \cite{bao2013}, the existence and non-existence of the least energy ground state are shown for potentials satsifying $ V \geq 0 $ and $ \lim_{|\vx|\rightarrow \infty} V(\vx) = \infty $. 
	
	The gradient flow with discrete normalization (GFDN) introduced in \cite{bao2004} (also known as the imaginary time evolution method in physical literature) is one of the most popular techniques for computing the least energy ground state. We refer the readers to \cite{bao2004, bao2006, ionut2010, bao2013, xavier2014, bao2015, bao2018, bao2019, cai2021, liu2021} and references therein for more detailed discussion about the normalized gradient flow approach and to \cite{ionut2017, wen2017, xavier2017, xavier2018} for some optimization methods. Some recent results of convergence of the GFDN can be found in \cite{faou2018, henning2020}. 
	
	With the focusing pure power nonlinearity, there is no least energy ground state when $ \alpha > 2/d $, i.e. the $ L^2 $-supercritical regime and when $ \alpha = 2/d $, the least energy ground state usually does not exist for all $ m>0 $. However, the applications to plasma physics sometimes involve supercritical solitary waves and the least action ground state which will be introduced later is considered. In this paper, we will focus on computing the least action ground state. In this case, $ \omega $ is always given and \cref{SNLS} is a semilinear elliptic equation. 
	
	We define the Schr\"odinger operator $ H_\omega: H^1(\R^d) \rightarrow H^{-1}(\R^d) $ as
	\begin{equation}
		H_\omega \phi = -\Delta \phi + V \phi + \omega \phi. 
	\end{equation}
	We denote by $ H_0 $ the operator $ -\Delta + V $. Let $ \omega_0 $ be defined as
	\begin{equation}\label{def:omega_0}
		\omega_0 := -\inf_{\substack{\phi \in H^1(\R^d) \\ \| \phi \|_{L^2}= 1 }} \int_{\R^d} \left[|\nabla \phi(\vx)|^2 + V(\vx)|\phi(\vx)|^2\right] \dx.  
	\end{equation}
	In most of the cases especially those are physically meaningful, the infimum above will be non-positive. Hence, we will assume that $ \omega_0 \geq 0 $ in this paper. However, even if $ \omega_0 < 0 $, our methods are also valid with slight modification as shown in \cref{rem:modify}. If we further have $ \omega_0>0 $, then $ - \omega_0 $ is the smallest eigenvalue of $ H_0 $ and we denote by $ \phi_0^{\text{lin}} $ the corresponding nonnegative eigenfunction with $ L^2 $ norm equal to 1 (also called normalized linear ground state), which is strictly positive and unique up to a constant phase if $ V $ is locally integrable and bounded from above \cite{lieb2001, zhou2017}. Later, we will always assume that $ \omega > \omega_0 \geq 0 $. 
	
	Then we define the \emph{action functional} $ S_\omega: H^1(\R^d) \rightarrow \R $ as
	\begin{equation}
		S_\omega(\phi) = E(\phi) + \omega M(\phi) = \int_{\R^d} \left[|\nabla \phi(\vx)|^2 + V(\vx) |\phi(\vx)|^2 + \omega |\phi(\vx)|^2 - \frac{1}{\alpha+1} |\phi(\vx)|^{2\alpha+2} \right] \rmd x.  
	\end{equation}
	We note that $ \phi \in H^1(\R^d) $ solves \cref{SNLS} if and only if $ \phi $ is a critical point of $ S_\omega $. The \emph{least action ground state} is defined as the minimizer of $ S_\omega $ among all nontrivial solutions of \cref{SNLS} (or nonzero critical point of $ S_\omega $). Let $ \mathcal{F}_\omega $ be the set of all nontrivial solutions to \cref{SNLS}, that is
	\begin{equation}\label{F_omega}
		\mathcal{F}_\omega = \{0 \neq \phi \in H^1(\R^d) : -\Delta \phi + V \phi + \omega \phi - |\phi|^{2\alpha} \phi = 0 \}. 
	\end{equation}
	Then the {least action ground state} $ \phi_\omega^\text{S} $ is given by 
	\begin{equation}\label{action ground state}
		\phi_\omega^\text{S} = \argmin_{ \phi \in \mathcal{F}_\omega} S_\omega(\phi). 
	\end{equation}
	We denote the action of the least action ground state associated with $ \omega $ by $ S_g(\omega) $, mass by $ M_g(\omega) $ and energy by $ E_g(\omega) $, i.e., ($ M_g $ and $ E_g $ may not be well-defined if the uniqueness of the least action ground states is not known)
	\begin{equation}\label{S_g}
		S_g(\omega) = S_\omega(\phi^S_\omega), \quad M_g(\omega) = M(\phi^S_\omega), \quad E_g(\omega) = E(\phi^S_\omega). 
	\end{equation} 
	By the method of the Nehari manifold, one can establish a variational characterization of the least action ground state. We define the \emph{Nehari manifold} $ \mathcal{N}_\omega $ as
	\begin{equation}
		\mathcal{N}_\omega = \{0 \neq \phi \in H^1(\R^d): I_\omega(\phi) = 0\}, 
	\end{equation}
	where $ I_\omega $ is the \emph{Nehari functional} defined as
	\begin{equation}
		I_\omega(\phi) = \int_{\R^d} \left[ |\nabla \phi(\vx)|^2 + V(\vx) |\phi(\vx)|^2 + \omega |\phi(\vx)|^2 - |\phi(\vx)|^{2\alpha+2} \right] \dx. 
	\end{equation}
	By the method in \cite{fukaya2019}, one can show that under our assumptions \ref{rq1}, \ref{rq2} and \ref{rq3} or \ref{rq3'} made in \cref{sec:2}, the least action ground state $ \phi^\text{S}_\omega $ of \cref{SNLS} exists, which is also a minimizer of $ S_\omega $ on the Nehari manifold $ \mathcal{N}_\omega $, i.e, 
	\begin{equation}\label{minimization}
		\phi^\text{S}_\omega = \argmin_{\phi \in \mathcal{N}_\omega} S_\omega(\phi). 
	\end{equation}
	For simplicity of notation, we will use $ \phi_\omega $ instead of  $ \phi^\text{S}_\omega $ to denote the least action ground state. 
	
	Then we review some results about the least action ground state of \cref{SNLS}. When $ V(\vx) \equiv 0 $, \cref{SNLS} is the scalar field equation and it has been proved that when $ \omega>0 $ there exists a unique positive, radial and exponentially decreasing least action ground state and all the other least action ground states can be obtained by phase shift and/or translation of it \cite{coz2009}. Besides, the solution to this equation is closely related to the sharp constants in Gagliardo-Nirenberg inequality and the global well-posedness of some nonlinear Schr\"odinger equations \cite{weinstein1982}. Similar results (except translation invariance) hold when $ V = - \gamma/|x|^\sigma \ (\gamma>0, 0<\sigma<\min\{2, d\}) $ is the attractive inverse power potentials (see \cite{fukaya2021} for the proof and the sufficient conditions for uniqueness with more general potentials). When $ V = -Z \delta \  (Z>0) $ is the attractive delta potential in 1D, the unique positive ground state can be explicitly given by
	\begin{equation}\label{delta_exact}
		\phi_\omega(x) = \left\{ (\alpha+1) \omega \sech^2\left[ \alpha\sqrt{\omega} |x| + \tanh^{-1}\left(\frac{Z}{2\sqrt{\omega}}\right) \right] \right\}^{\frac{1}{2\alpha}}, \quad \omega>\frac{Z^2}{4}. 
	\end{equation}
	Moreover, when we choose $ Z=0 $ in \cref{delta_exact}, it gives the least action ground states for $ V(\vx) \equiv 0 $. The exponential decay of the least action ground state allows us to truncate the whole space problem onto a large enough bounded domain in computation, as what people usually do when computing the least energy ground state. 
	
	We remark that there is a long standing problem about the relationship between the least energy ground state and the least action ground state: whether a least energy ground state is a least action ground state and the converse. We refer the readers to \cref{sec:action_vs_energy} and \cite{carles2021,dovetta2021,jeanjean2021,coz2009} for more details. Roughly speaking, with the homogeneous pure power nonlinearity we are considering, the least energy ground state will also be a least action ground state with $ \omega = - \mu^g $ defined in \cref{mu}, and all the least action ground states of this $ \omega $ will have the same mass and are least energy ground states. Though the original proof of this result in \cite{dovetta2021} consider the non-potential case on general domain, we found that it is still valid in our setting with the variational characterization by means of Nehari manifold \cref{minimization}. 
	
	People are also concerned about the orbital stability and instability (see \cref{def:stability}) of some certain class of standing wave solutions $ e^{i \omega t} \phi_\omega(\vx) $ of \cref{NLSE}. In general, it is difficult to consider all the standing waves and particular efforts are paid to least energy ground states and least action ground states. To prove stability or instability, the following criterion due to Grillakis, Shatah and Strauss \cite{GSS1987} is very powerful: 
	\begin{enumerate}[label = (C\arabic*)]
		\item \label{cri1} if $ \left. \frac{d}{d\omega} M(\phi_\omega) \right \vert_{\omega = \omega_0} > 0 $, then the standing wave solution $ e^{i \omega_0 t} \phi_{\omega_0}(\vx) $ is stable; 
		\item \label{cri2} if $ \left. \frac{d}{d\omega} M(\phi_\omega) \right \vert_{\omega = \omega_0} < 0 $, then the standing wave solution $ e^{i \omega_0 t} \phi_{\omega_0}(\vx) $ is unstable. 
	\end{enumerate}
	In order to use such criterion, one has to prove some spectral properties of the linearized operator (second variation of $ S_\omega $), which is proved for some potentials with radial symmetry such as zero potential, delta potentials \cite{fukuizumi2008} and inverse power potentials \cite{fukaya2021}. For the scalar field equation (i.e. $ V(\vx) \equiv 0 $), it is well understood that the standing wave solution is stable for any $ \omega > 0 $ when $ \alpha<2/d $ and unstable for any $ \omega>0 $ when $ \alpha \geq 2/d $ (see also \cite{cazenave2003, cazenave_lions} for the proof without using the above criterion). In \cite{fukuizumi2003}, it is shown that under suitable assumptions on $ V $, the standing wave is stable for $ 0<\alpha<2/d $ and sufficiently large $ \omega $, or for $ 0<\alpha<2/(d-2)_+ $ and $ \omega $ close to $ \omega_0 $. Later, in \cite{fukuizumi_unstable}, the instability of standing waves when $ \alpha > 2/d $ and $ \omega $ is sufficiently large is shown. In \cite{fukuizumi2008}, the exact stability and instability regime for 1D attractive delta potentials is found. However, to our best knowledge, there is no complete result when $ V $ is the attractive inverse power potential. Some partial results can be found in \cite{fukaya2019,fukaya2021,li2020}. In \cref{sec:stability}, by the above criterion and our numerical results, we will make a conjecture about the stability and instability of standing waves with this kind of potentials. 
		
	After the completion of the present work, we became aware of the paper \cite{liu2022}, where the authors considered the computation of least action ground states for the NLSE with/without rotation under a confining potential and proposed a discrete normalized gradient flow with asymptotic Lagrange multiplier for a simplified variational characterization. 
	
	The aim of this paper is to propose a gradient flow with discrete normalization (GFDN) to compute the least action ground state. We also present a continuous normalized gradient flow (CNGF) and prove its action diminishing property, which gives a mathematical justification of the GFDN. We compare different semi-discretization schemes for GFDN and find that the best one in all the cases is the GFDN-BF scheme. As the application of the GFDN-BF scheme, extensive numerical examples of least action ground states with different potentials are shown. Also, we verify some existing results about the least action ground state numerically and lead to some conjectures. 
	
	The rest of the paper is organized as follows. In \cref{sec:2}, we state the assumptions on the potential and give some examples. Besides, we show some preliminary results about the the Nehari manifold and some asymptotic results about the least action ground states. In \cref{sec:3}, we introduce the GFDN and show that GFDN can be viewed as a discretization of CNGF which preserves $ I_\omega $ and diminishes the action. In \cref{sec:numerical methods}, we further discretize the GFDN in time and introduce the GFDN-BF scheme. Spatial discretization for different types of potentials are also discussed. In \cref{sec:numerical results}, we compare different semi-discretization schemes for the GFDN and show various numerical results. Also, we apply our scheme to verify some existing results and lead to some conjectures. 
	
	Throughout the paper, we use $ H^k(\Omega) $ and $ \| \cdot \|_{H^k(\Omega)} $ to denote the standard Sobolev spaces and their norms, respectively, where $ \Omega = \R^d $ or $ \Omega \subset \R^d $ is a bounded domain. In particular, the norm and inner product of $ L^2(\Omega) = H^0(\Omega) $ is denoted by $ \| \cdot \|_{L^2(\Omega)} $ and $ \lrang{\cdot}{\cdot} $, respectively. Also, we use $ L^p(\Omega) $ to denote the standard $ L^p $ spaces with norm $ \| \cdot \|_{L^p(\Omega)} $. Sometimes we shall omit $ \Omega $ when $ \Omega = \R^d $ for simplicity. 
	
	\section{Properties of the Nehari manifold and the least action ground state}\label{sec:2}
	In this section, we first state the assumptions on the potential $ V $ and give some examples. Then we show some basic properties of the Nehari manifold which will be frequently used in the paper. We also review some asymptotic results of the least action ground state, which are observed in the numerical experiments. 
	\subsection{Preliminaries}\label{sec:setting}
	We first list the assumptions on the potential function $ V $: 
	\begin{enumerate}[label = (A\arabic*)]
		\item \label{rq1} $ V \in L^q(\R^d) + L^\infty(\R^d) $, where $ q = 1 $ when $ d = 1 $, $ q > 1 $ when $ d=2 $ and $ q = d/2 $ when $ d \geq 3 $. When $ d=1 $, $ V $ can also be the sum of a bounded Borel measure and an $ L^\infty $ function that vanishes at infinity. 
		\item \label{rq2}$ V $ vanishes at infinity, i.e.
		\begin{equation}
			|\{\vx\in\R^d: |V(\vx)|>a\}| < \infty, \quad \forall a>0. 
		\end{equation}
		\item \label{rq3} $ V $ is non-positive in the sense that the potential energy defined in \cref{potential_energy} is always non-positive. (Actually, this assumption can be weakened to \ref{rq3'} as illustrated in \cref{nonpositive}. ) 
	\end{enumerate}
	
	Let $ G:H^1(\R^d) \rightarrow \R $ be the potential energy defined by
	\begin{equation}\label{potential_energy}
		G(\psi) = \int_{\R^d} V(\vx) |\psi(\vx)|^2 \dx. 
	\end{equation}
	From assumptions \ref{rq1} and \ref{rq2}, we can deduce that (see \cite{cazenave2003,lieb2001}): the potential energy functional $ G: H^1(\R^d) \rightarrow \R $ is weakly continuous and $ S_\omega, I_\omega \in C^1(H^1(\R^d); \R) $ have Fr\'eschet derivatives $ S_\omega', I'_\omega \in C(H^1(\R^d); H^{-1}(\R^d)) $. Moreover, $ I_\omega(\phi) $ satisfies
	\begin{equation*}
		I_\omega(\phi) = \lrang{S'_\omega(\phi)}{\phi}_{H^{-1}, H^1} \text{ for all } \phi \in H^1(\R^d), 
	\end{equation*}
	where $ \lrang{\cdot}{\cdot}_{H^{-1}, H^1} $ is the dual pair between $ H^{-1}(\R^d) $ and $ H^1(\R^d) $.
	
	\begin{remark}\label{nonpositive}
		The non-positive assumption on the potential $ V $ can be weakened. It suffices to assume that the potential energy is negative for at least one least action ground state in the non-potential case. To be more precise, according to the result in \cite{coz2009}, let $ \phi^0_\omega (\omega>0) $ be the unique, positive, radial and decreasing least action ground state of \cref{SNLS} with $ V(\vx) \equiv 0 $ (we don't specify the dimension $ d $ in the definition of $ \phi^0_\omega $ but it will always be clear from the context). Note that all the least action ground states in this case are given by phase shift and translation of $ \phi_\omega^0 $, i.e. 
		\begin{equation*}
			e^{i \theta} \phi_\omega^0(\cdot - \mathbf{y})
		\end{equation*}
		for some $ \theta \in \R $ and $ \mathbf{y} \in \R^d $. Then, instead of assuming that $ V $ is non-positive, we just need to assume that
		\begin{enumerate}[label = (A3')]
			\item \label{rq3'}  $ V(\vx) \equiv 0 $ or there exists $ \mathbf y \in \R^d $ such that
			\begin{equation*}
				\int_{\R^d} V(\vx) |\phi^0_\omega(\vx - \mathbf y)|^2 \rmd \vx < 0. 
			\end{equation*}
		\end{enumerate}
		As mentioned before, in 1D, $ \phi_\omega^0 $ can be given analytically by
		\begin{equation}\label{exact_non_potential}
			\phi_\omega^0(x) = \left[ (\alpha+1)\omega \sech^2\left( \alpha \sqrt{\omega} x \right) \right]^{\frac{1}{2\alpha}}, \quad x \in \R, \quad \omega>0. 
		\end{equation}
	\end{remark}
	
	In the following, the assumptions \ref{rq1}, \ref{rq2} and \ref{rq3} will always be assumed (most results are still valid if \ref{rq3} is replaced by \ref{rq3'} as discussed in \cref{rem:modify}). Here, we give some examples of potentials that satisfy these assumptions. 
	\begin{exmp}\label{exmp:poten}
		\begin{enumerate}[label = \roman*)]
			\item attractive delta potentials in 1D:
			\begin{equation}\label{delta_potential}
				V(x) = -Z \delta(x), \quad x \in \R \text{ and } Z > 0. 
			\end{equation} 
			\item finite well potentials:
			\begin{equation}
				V(\vx) = \left\{
				\begin{aligned}
					&-Z, &&\vx \in U \subset \R^d, \\
					&0,  &&\vx \notin U, 
				\end{aligned}
				\right. 
				\quad Z > 0. 
			\end{equation} 
			\item attractive inverse power potentials: 
			\begin{equation}\label{inverse_power}
				V(\vx) = -\frac{\gamma}{|\vx|^\sigma}, \quad \vx \in \R^d, \quad \gamma>0, \quad 0<\sigma < \min\{2, d\}. 
			\end{equation} 
			\item zero potential (non-potential): $ V(\vx) \equiv 0 $. 
			\item smooth potentials: 
			\begin{align}
				&V(\vx) = - Ze^{-|\vx|^2}, \quad \vx \in \R^d, \quad Z>0. 
			\end{align}
		\end{enumerate}
	\end{exmp}
	
	\subsection{Properties of Nehari manifold}
	We define the functional $ I_\omega^\ast:H^1(\R^d) \rightarrow \R $ associated with $ H_\omega $ as	
	\begin{equation*}
		I^\ast_\omega(\phi) = \lrang{H_\omega(\phi)}{\phi}_{(H^{-1}(\R^d), H^{1}(\R^d))} = \int_{\R^d} \left[|\nabla \phi(\vx)|^2 + V(\vx) |\phi(\vx)|^2 + \omega |\phi(\vx)|^2 \right] \dx, 
	\end{equation*}
	which is the quadratic part of $ I_\omega $. Under the assumption $ \omega > \omega_0 $, we have $ I_\omega^\ast(\phi) \geq 0 $ for all $ \phi \in H^1(\R^d) $. Furthermore, we have the equivalence of norms: there exist $ C_1, C_2 > 0 $ depending on $ \omega $ such that for any $ \phi \in H^1(\R^d) $, 
	\begin{equation}\label{lem:equivalence of norms}
		C_1 \| \phi \|_{H^1(\R^d)}^2 \leq I_\omega^\ast(\phi) \leq C_2 \| \phi \|_{H^1(\R^d)}^2, 
	\end{equation}
	when $ \omega > \omega_0 $. 
	
	An important consequence of \cref{lem:equivalence of norms} is that the Nehari manifold is bounded away from $ 0 $ in {$ L^{2\alpha+2} $} norm, which is stated in the following proposition. 
	\begin{proposition}\label{boundedaway}
		When $ \omega>\omega_0 $, there exists $ \delta>0 $ depending on $ \omega $ such that for any $ \phi \in \mathcal{N}_\omega $, one has {$ \| \phi \|_{L^{2\alpha+2}(\R^d)} > \delta $}. 
	\end{proposition}
	\begin{proof}
		By \cref{lem:equivalence of norms} and the Sobolev embedding theorem, one has
		\begin{equation*}
			I_\omega(\phi) = I_\omega^\ast(\phi) - \| \phi \|_{L^{2\alpha+2}(\R^d)}^{2\alpha+2} \geq C_1 \| \phi \|_{H^1(\R^d)}^2 - C_d \| \phi \|_{H^1(\R^d)}^{2\alpha+2}. 
		\end{equation*}
		Then there exists some $ \delta>0 $ such that when $ \| \phi \|_{H^1(\R^d)} \leq \delta $ and $ \phi \neq 0 $, one has $ I_\omega(\phi) > 0 $. It follows that if $ \phi \in \mathcal{N}_\omega $, then $ \| \phi \|_{H^1} > \delta $. Since $ \phi \in \mathcal{N}_\omega $ implies $ \| \phi \|_{L^{2\alpha+2}}^{2\alpha+2} = I_\omega^\ast(\phi) $, the conclusion follows immediately. 
	\end{proof}
	With the functional $ I_\omega^\ast $, we have
	\begin{equation}
		\begin{aligned}
			S_\omega(\phi) 
			&= I_\omega(\phi) + \frac{\alpha}{\alpha+1} \| \phi \|_{L^{2\alpha+2}}^{2\alpha+2} \\
			&= \frac{1}{\alpha+1} I_\omega(\phi) + \frac{\alpha}{\alpha+1} I^\ast_\omega(\phi)
		\end{aligned}\label{S1S2}
	\end{equation}
	and
	\begin{equation}\label{I}
		I_\omega(\phi) = I_\omega^\ast(\phi) - \| \phi \|_{L^{2\alpha+2}}^{2\alpha+2}. 
	\end{equation}
	
	We have the following lemma that is extremely helpful and is the motivation of our numerical scheme, which states that each ray $ \{\lambda \phi\}_{\lambda>0} $ with $ 0 \neq \phi \in H^1(\R^d) $ will intersect the Nehari manifold $ \mathcal{N}_\omega $ at one point. 
	\begin{lemma}\label{lambda}
		When $ \omega>\omega_0 $, for any $ \phi \in H^1(\R^d) \setminus \{0\} $, there exists a unique $ \lambda_\omega(\phi) > 0 $ such that $ I_\omega(\lambda_\omega(\phi) \phi) = 0 $ with
		\begin{equation}\label{eq:lambda}
			\lambda_\omega(\phi) = \left( \frac{I_\omega^\ast( \phi )}{ \| \phi \|^{2\alpha+2}_{L^{2\alpha+2}}} \right)^{\frac{1}{2\alpha}} =  \left( 1 + \frac{I_\omega( \phi )}{ \| \phi \|^{2\alpha+2}_{L^{2\alpha+2}}} \right)^{\frac{1}{2\alpha}}. 
		\end{equation} 
		Moreover, one has $ \lambda_\omega(\phi) < 1 $ if $ I_\omega(\phi)<0 $ and $ \lambda_\omega(\phi)>1 $ if $ I_\omega(\phi)>0 $. 
	\end{lemma}
	\begin{proof}
		The proof is very straightforward and we shall omit it for brevity. 
	\end{proof}
	
	As shown in \cref{minimization}, the least action ground state can also be characterized as the minimizer of $ S_\omega $ on the Nehari manifold $ \mathcal{N}_\omega $. Then we can show the existence of nonnegative least action ground states, which enable us to choose real initial data in our numerical scheme. 
	
	\begin{proposition}[existence of nonnegative ground state]\label{nonnegative}
		If $ \phi_\omega $ is a least action ground state for some $ \omega > \omega_0 $, then so is $ |\phi_\omega| $. Moreover, there exists a constant $ \theta \in \R $ such that $ \phi_\omega = e^{i \theta}|\phi_\omega| $. 
	\end{proposition}
	\begin{proof}
		It is well known that (see \cite{bao2013,lieb2001})
		\begin{equation*}
			\int_{\R^d} |\nabla|\phi(\vx)||^2 \rmd \vx \leq \int_{\R^d} |\nabla \phi(\vx)|^2 \dx
		\end{equation*}
		and the equality holds if and only if there exists a constant $ \theta \in \R $ such that $ \phi = e^{i \theta} |\phi| $. 
		Then one has
		\begin{equation*}
			S_\omega(|\phi_\omega|) \leq S_\omega(\phi_\omega) \quad \text{and} \quad I_\omega(|\phi_\omega|) \leq I_\omega(\phi_\omega) = 0. 
		\end{equation*}
		By \cref{lambda}, one can find $ 0 < \lambda \leq 1 $ such that $ I_\omega(\lambda |\phi_\omega|) = 0 $. Recall \cref{S1S2}, then it follows that
		\begin{equation*}
			S_\omega(\lambda |\phi_\omega|) = \frac{\alpha }{\alpha+1} \| \lambda |\phi_\omega| \|_{L^{2\alpha+2}}^{2\alpha+2} \leq \frac{\alpha }{\alpha+1} \| \phi_\omega \|_{L^{2\alpha+2}}^{2\alpha+2} = S_\omega(\phi_\omega), 
		\end{equation*}
		which implies that $ \lambda = 1 $ and thus $ |\phi_\omega| $ is also a least action ground state. The second part of the proposition follows immediately. 
	\end{proof}
	In the following, we shall only consider the nonnegative least action ground states, which, in particular, are real valued. 
	
	\subsection{Properties of least action ground states when $ \omega \rightarrow \infty $ and $ \omega \rightarrow \omega_0 $}\label{sec:asymtotics}
	In this subsection, we review some results about the limiting behavior of the least action ground states when $ \omega \rightarrow \infty $ and $ \omega \rightarrow \omega_0 $. Most of the results described here can be found in \cite{fukuizumi_unstable, fukuizumi2003}. 
	
	First, the following properties of $ S_g(\omega) $ (defined in \cref{S_g}) are well known: i) $ S_g(\omega) $ is strictly increasing in $ \omega $; ii) $ S_g(\omega) \rightarrow 0 $ as $ \omega \rightarrow \omega_0 $ and $ S_g(\omega) \rightarrow \infty $ as $ \omega \rightarrow \infty $; iii) When $ V(\vx) \not \equiv 0 $ and $ \omega $ close to $ \omega_0 > 0 $, 
	\begin{equation}\label{estimate of S_g}
		0 < S_g(\omega) \lesssim (\omega - \omega_0)^{1 + \frac{1}{\alpha}}, 
	\end{equation}
	and when $ V(\vx) \equiv 0 $, $ S_g(\omega) \sim \omega^{1 + \frac{1}{\alpha} - \frac{d}{2}} $. With more involved analysis, one can show that the estimate \cref{estimate of S_g} is sharp (see the proof of Lemma 4.3 in \cite{fukuizumi2003}). Besides, we can have a precise characterization of the nonnegative least action ground states when $ \omega \rightarrow \omega_0 $ and $ \omega \rightarrow \infty $. Let $ \{\phi_\omega\}_{\omega > \omega_0} $ be the family of non-negative least action ground states of \cref{SNLS}. We define two rescaled function $ \widehat\phi_\omega $ and $ \widecheck \phi_\omega $ as
	\begin{equation}\label{rescaling}
		\widehat \phi_\omega(\vx) = \frac{\phi_\omega(\vx)}{ \| \phi_\omega \|_{L^2} }, \quad \widecheck \phi_\omega(\vx) = \omega^{-\frac{1}{2\alpha}} \phi_\omega \left(\frac{\vx}{\omega^\frac{1}{2}} \right ), \quad \vx \in \R^d, \quad \omega > \omega_0 \geq 0. 
	\end{equation} 
	Then if $ \omega_0 > 0 $ and $ V \in L_\text{loc}^1(\R^d) $, one has
	\begin{equation}\label{lim_omega_0}
		\lim_{\omega \rightarrow \omega_0} \| \widehat \phi_\omega - \phi_0^{\text{lin}}\|_{H^1(\R^d)} = 0, 
	\end{equation}
	where $ \phi_0^{\text{lin}} $ is the unique positive normalized linear ground state. On the other hand, if we in addition assume that $ V \in L^q(\R^d) + L^\infty(\R^d) $ for some $ q > d/2 $ and $ q \geq 1 $ and $ \phi_\omega $ is $ G_R $-invariant as defined in \cite{fukuizumi2003}, then one has
	\begin{equation}\label{lim_omega_infty}
		\lim_{\omega\rightarrow\infty} \| \widecheck \phi_\omega - \phi^0_1 \|_{H^1(\R^d)} = 0, 
	\end{equation}
	where  $ \phi^0_1 $ is defined in \cref{nonpositive}. A important observation is that when $ \alpha = 2/d $ is $ L^2 $-critical, one has
	\begin{equation}\label{mass_thereshold}
		\| \widecheck \phi_\omega \|_{L^2(\R^d)} = \| \phi_\omega \|_{L^2(\R^d)}, 
	\end{equation}
	which implies that $ M_g(\omega) \rightarrow M(\phi_1^0) $ as $ \omega \rightarrow \infty $. This will be observed in our numerical results in \cref{sec:numerical results}. 
	\begin{remark}
		The additional assumption that $ V \in L^q(\R^d) + L^\infty(\R^d) $ for some $ q > d/2 $ and $ q \geq 1 $ only differs from our assumption at the end-point $ q=d/2 $ when $ d \geq 3 $, which is used to guarantee that the rescaled potential energy of $ \widecheck \phi_\omega $ will converge to zero as $ \omega \rightarrow \infty $. The assumption that $ \phi_\omega $ is $ G_R $-invariant is used to exclude the possibility of translation. 
	\end{remark}
	\begin{remark}
		The assumption (V3) about the compactness of the minimizing sequence proposed in \cite{fukuizumi2003} can be verified under assumptions \ref{rq1}, \ref{rq2} and $ \omega_0 > 0 $ (see \cite{lieb2001}). 
	\end{remark}
	
	\section{Normalized Gradient Flow}\label{sec:3}
	In this section, we present a CNGF and show its $ I_\omega $ preserving and action decreasing property. 
	\subsection{Gradient flow with discrete normalization (GFDN)}
	Motivated by the method in \cite{bao2004}, we consider a discrete normalized gradient flow to solve the least action ground state. We shall use a uniform mesh in time with $ \tau $ being the time step size and $ t_n = n \tau $. To adapt an algorithm for the solution of the usual gradient flow to the minimization problem under a constraint, it is natural to consider the following splitting (or projection) scheme: 
	\begin{align}
		&\partial_t \phi = -\frac{1}{2} \frac{\delta S_\omega(\phi)}{\delta \phi} = \Delta \phi - V(\vx)\phi - \omega \phi + |\phi|^{2\alpha} \phi, \quad \vx \in \R^d, \quad t_n < t < t_{n+1}, \quad n \geq 0, \label{GF}\\
		&\phi(\cdot, t_{n+1}) := \lambda_{n+1}\phi(\cdot, t_{n+1}^-), \quad n\geq 0, \quad \lambda_{n+1} = \left( \frac{I_\omega^\ast(\phi(\cdot, t_{n+1}^-))}{\|\phi(\cdot, t_{n+1}^-)\|^{2\alpha+2}_{L^{2\alpha+2}}} \right)^{\frac{1}{2\alpha}} \label{normalization},\\
		&\phi(\vx, 0) = \phi_0(\vx), \quad \vx \in \R^d \label{endGF},
	\end{align}
	where $ \phi_0 \in \mathcal{N}_\omega $ (we will always require this condition) and $ \phi: \R^d \times [0, \infty) \rightarrow \R $ is a time dependent function. It follows immediately from the construction that the gradient flow \cref{GF},\cref{endGF} has the action diminishing property, i.e. 
	\begin{equation*}
		S_\omega(\phi(\cdot, t')) \leq S_\omega(\phi(\cdot, t)), \quad t_n < t < t'< t_{n+1}, \quad n \geq 0. 
	\end{equation*}
	\begin{remark}
		The solution to \cref{GF},\cref{endGF} may not preserve the normalized action decreasing property
		\begin{equation*}
			S_\omega(\tilde \phi(\cdot, t')) \leq S_\omega(\tilde \phi(\cdot, t)), \quad 0< t < t'<T_\text{max}, 
		\end{equation*} where $ \phi(x, t) $ is the solution to \cref{GF} with $ \phi(\cdot, 0) = \phi_0 \in \mathcal{N}_\omega $, $\tilde \phi = \lambda_\omega(\phi) \phi \in \mathcal{N}_\omega $ and $ T_{\text{max}} $ is the maximal existence time. Actually, we solve \cref{GF}, \cref{endGF} in one dimension with $ \Omega = \R $ and $ V(\vx) \equiv 0 $ numerically by the time splitting spectral method (see \cite{bao2004}) with the initial condition $ \phi_0 = \tilde g_j \in \mathcal{N}_\omega, j \in \{1, 2\} $ where $ g_1(x) = e^{-x^2/2} $ and $ g_2(x) = \sech(x)e^{-x^2/2} $ (note that we have to use different initial data for different $ \omega $). We fix $ \alpha = 1 $. In computation, we choose $ h=1/64 $ and $ \Omega = [-16, 16] $. \Cref{fig:GF_without_normalization} shows, for different $ \omega $, the evolution of normalized action $ S_\omega(\tilde \phi(\cdot, t)) $. This implies the exact flow of \cref{GF},\cref{endGF} may not preserve the normalized action decreasing property for large time step. However, when we further discretize \cref{GF} in time, it is shown that the normalized action is always decreasing for any $ \tau > 0 $ for the GFDN-BF scheme in \cref{sec:time discretization}. 
		
		\begin{figure}[htbp]
			\centering\vspace{-0.5cm}
			\subfloat{\includegraphics[width=0.475\textwidth]{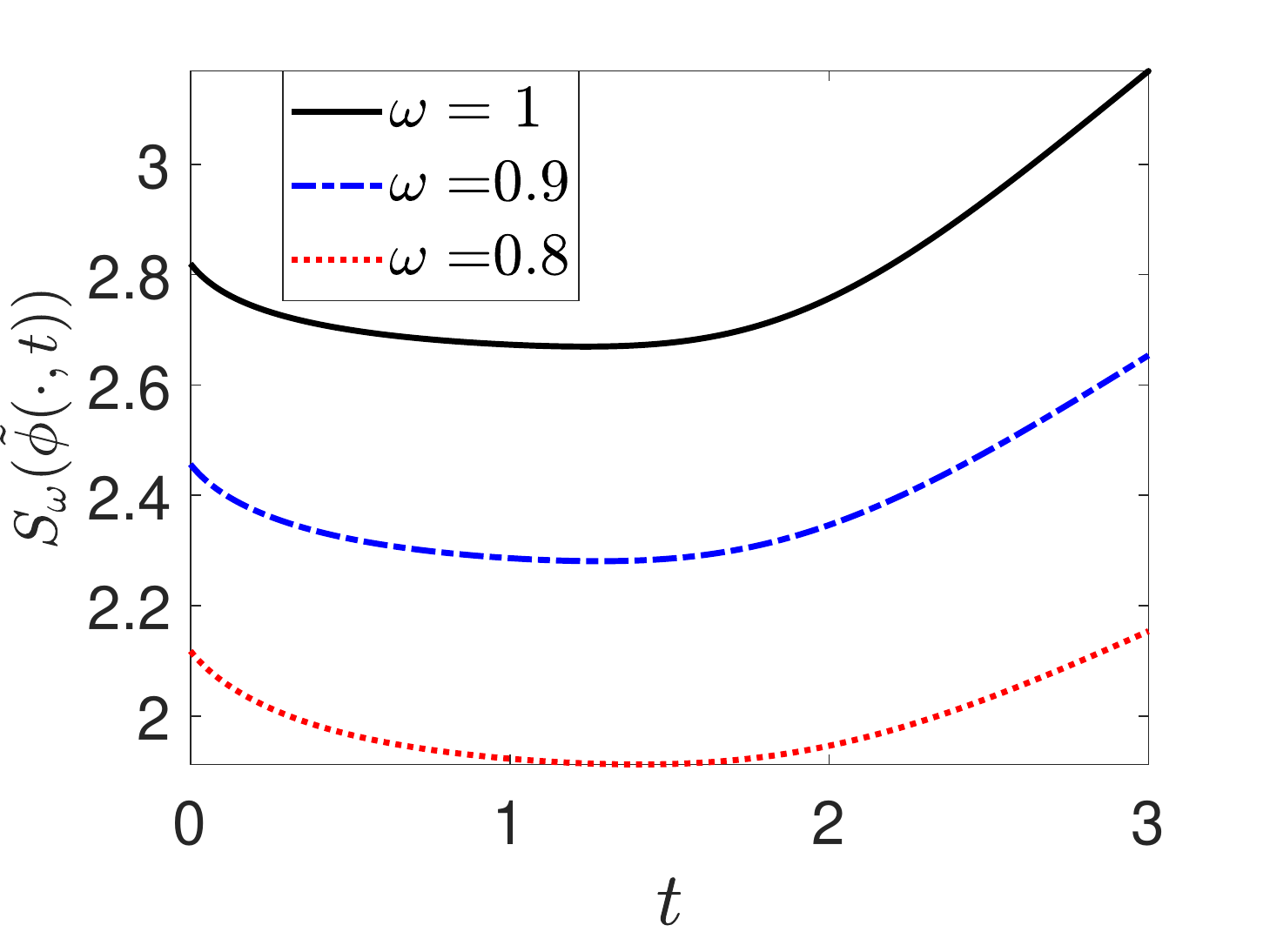}}
			\subfloat{\includegraphics[width=0.475\textwidth]{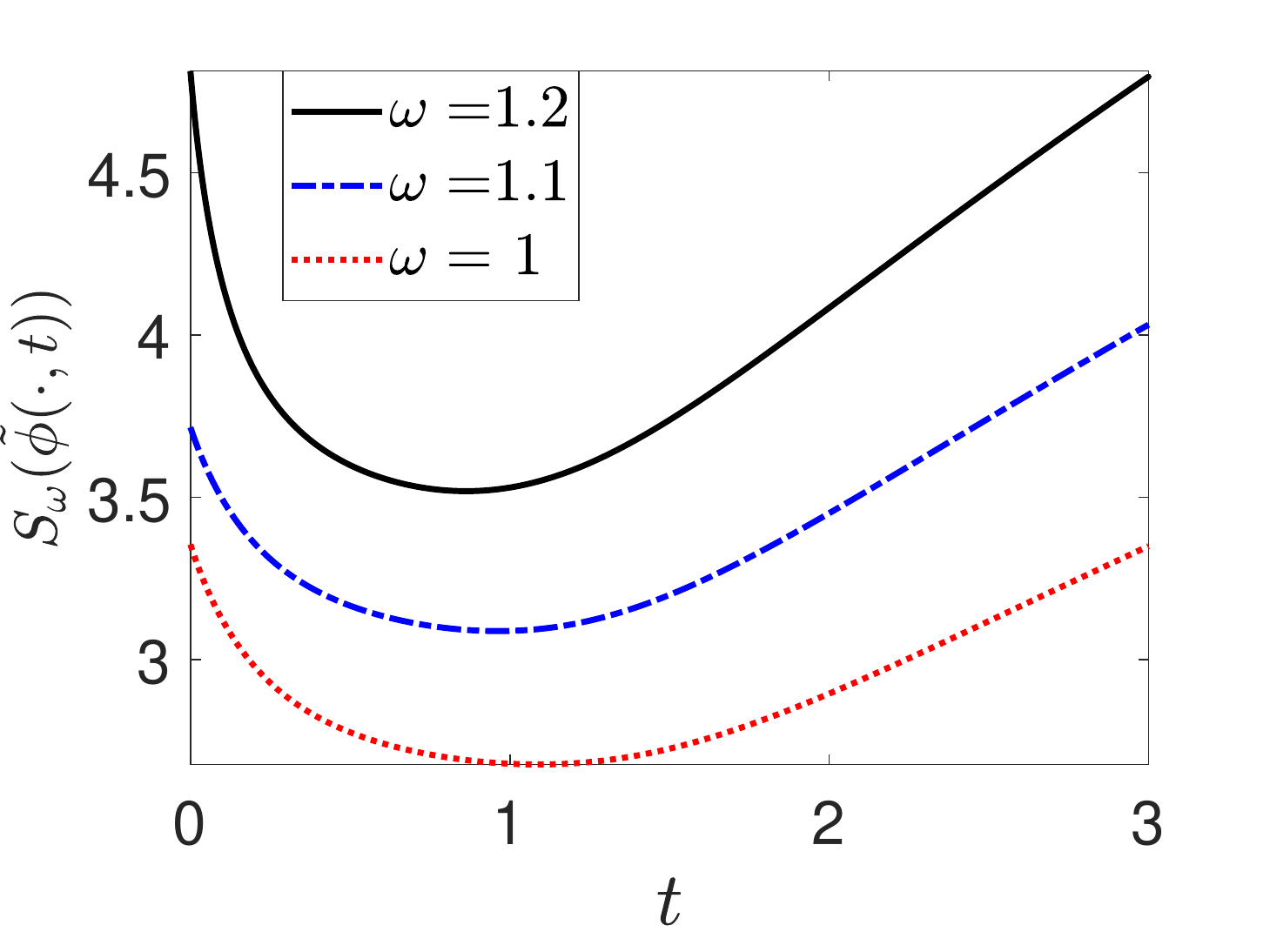}}\vspace{-0.35cm}
			\caption{$ S_\omega(\tilde \phi(\cdot, t)) $ as a function of time for with initial data $ \tilde g_1 $ (left) and $ \tilde g_2 $ (right)}\vspace{-0.5cm}
			\label{fig:GF_without_normalization}
		\end{figure}
	\end{remark}
	
	\subsection{Continuous normalized gradient flow (CNGF)}
	In fact, the normalization step \cref{normalization} is equivalent to solving the following ODE \emph{exactly}: 
	\begin{align}
		v_t(\vx, t) &= \mu_{\phi}(t, \tau) v(\vx,t), \quad \vx \in \R^d, \quad t_n < t < t_{n+1}, \quad n \geq 0, \\
		v(\vx, t_n) &= \phi(\vx, t_{n+1}^-), \quad \vx\in\R^d, 
	\end{align}
	where $ \mu_\phi $ (do not confuse it with $ \mu^g $ defined in \cref{mu}) is given by
	\begin{equation}
		\mu_{\phi}(t, \tau) \equiv \frac{1}{2\alpha \tau} \ln \lambda_{n+1}^{2\alpha} = \frac{1}{2\alpha\tau} \ln \frac{\int_{\R^d} \left[|\nabla \phi(t_{n+1}^-)|^2 + V|\phi(t_{n+1}^-)|^2 + \omega |\phi(t_{n+1}^-)|^2 \right] \dx}{ \int_{\R^d} |\phi(t_{n+1}^-)|^{2\alpha+2} \dx}, \quad t_n \leq t \leq t_{n+1}. 
	\end{equation}
	Thus the GFDN \cref{GF}-\cref{endGF} can be viewed as a first-order splitting method for the gradient flow with discontinuous coefficients: 
	\begin{align}
		&\partial_t \phi = \Delta \phi - V(\vx) \phi - \omega \phi + |\phi|^{2 \alpha} \phi + \mu_{\phi}(t, \tau) \phi, \quad \vx \in \R^d, \quad t \geq 0, \\
		&\phi(\vx, 0) = \phi_0(\vx), \quad \vx \in \R^d. 
	\end{align}
	
	We consider the GFDN \cref{GF}-\cref{endGF} as time step goes to zero. Letting $ \tau \rightarrow 0 $, we see that
	\begin{equation}
		\mu_\phi(t) := \lim_{\tau \rightarrow 0^+} \mu_{\phi}(t, \tau) = -\frac{\| H_\omega \phi(\cdot, t) \|^2_{L^2} - (\alpha+2) \lrang{H_\omega \phi(\cdot, t)}{|\phi(\cdot, t)|^{2\alpha} \phi(\cdot, t)} + (\alpha+1) \| \phi(\cdot, t) \|_{L^{4\alpha + 2}}^{4\alpha + 2}}{\alpha  \| \phi(\cdot, t) \|_{L^{2\alpha+2}}^{2\alpha+2}}. \label{mu_phi}
	\end{equation}
	The observation suggests that we should consider the following CNGF: 
	\begin{align}
		&\partial_t \phi = \Delta \phi - V(\vx)\phi - \omega \phi + |\phi|^{2\alpha} \phi + \mu_\phi(t) \phi, \quad \vx \in \R^d, \quad t \geq 0, \label{CNGF}\\
		&\phi(\vx, 0) = \phi_0(\vx), \quad \vx \in \R^d. \label{CNGFend} 
	\end{align}
	Then we shall show that if $ \phi_0 \in \mathcal{N}_\omega $ in \cref{CNGFend}, the CNGF \cref{CNGF}-\cref{CNGFend} will preserve $ I_\omega = 0 $ and diminish the action. Let $ \phi(t) = \phi(\cdot, t) (0 \leq t \leq T)$ and $ L(\phi) = H_\omega \phi - |\phi|^{2\alpha} \phi $. We first note that
	\begin{equation}\label{eq:mu_phi}
		\mu_\phi(t) = -\frac{\lrang{I_\omega'(\phi(t))}{L(\phi(t))}}{2\alpha  \| \phi(t) \|_{L^{2\alpha+2}}^{2\alpha+2}}. 
	\end{equation}
	Then one has
	\begin{equation*}
		\frac{d}{d t} I_\omega(\phi(t)) = \langle I'_\omega(\phi(t)), \phi_t(t) \rangle = \langle I'_\omega(\phi(t)), -L(\phi(t)) + \mu_\phi(t) \phi(t) \rangle = 2\mu_\phi(t) I_\omega(\phi(t)). 
	\end{equation*}
	Since $ I_\omega(\phi(0)) = 0 $, one can integrate it exactly to find that $ I_\omega(\phi(t)) = 0 $ for $ t \in [0, T] $. The $ S_\omega $ decreasing property follows from
	\begin{align*}
		\frac{d}{d t} S_\omega(\phi(t)) &= \langle S_\omega'(\phi(t)), \phi_t(t) \rangle = \langle S_\omega'(\phi(t)), -L(\phi(t)) \rangle + \mu_\phi(t) \langle S_\omega'(\phi(t)), \phi(t) \rangle \\
		& = -2\|L(\phi(t))\|_{L^2}^2 + 2\mu_\phi(t) I_\omega(\phi(t)) = -2\|L(\phi(t))\|_{L^2}^2 \leq 0. 
	\end{align*}
	\begin{remark}
		The RHS of \cref{CNGF} can also be viewed as a projected gradient and thus we shall call it projected gradient flow (PGF) in \cref{sec:numerical methods}. 
	\end{remark}

	\section{Further discretization of GFDN}\label{sec:numerical methods}
	In this section, we further discretize the GFDN in time by backward-forward Euler method to obtain the GFDN-BF scheme. Some good properties including positivity preserving and unconditionally action diminishing are proved. Then we show three standard spatial discretization methods to deal with different potentials. 
	\subsection{Semi-discretization in time}\label{sec:time discretization}
	We shall show the semi-discretization schemes for the GFDN \cref{GF}-\cref{endGF}, which is modified from corresponding schemes used to compute the least energy ground state. For comparison, we show another three semi-discretization schemes, which are also well-known in the literature of least energy ground state and we shall compare the performance of them in \cref{sec:comparison}. 
	
	We choose a time step size $ \tau > 0 $ and let $ t_n = n \tau $. We denote by $ \phi^n $ the approximation solution of $ \phi(\cdot, t) $ at $ t = t_n $. We use backward-forward Euler method for time discretization and then the GFDN-BF scheme reads: 
	\begin{align}
		&\frac{\tilde \phi^{n+1}-\phi^n}{\tau} = \Delta \tilde \phi^{n+1} - \omega \tilde \phi^{n+1} - V \phi^{n} +  |\phi^n|^{2\alpha} \phi^{n} \text{ on } \R^d,  \quad n \geq 0 \label{GFBF} \\ 
		&\phi^{n+1} = \lambda_{n+1} \tilde \phi^{n+1}, \quad \lambda_{n+1} = \left( \frac{I_\omega^\ast(\tilde \phi^{n+1}) }{  \| \tilde \phi^{n+1} \|^{2\alpha+2}_{L^{2\alpha+2}}} \right)^{\frac{1}{2\alpha}}, \quad \phi^0 = \phi_0 \in \mathcal{N}_\omega. \label{endGFBF}
	\end{align}
	
	To show that the GFDN-BF is well-defined, we have the following proposition. 
	\begin{proposition}\label{well-definedness of GFBF}
		Suppose that $ \omega>\omega_0 \geq 0 $. Then for any $ \tau > 0 $, \cref{GFBF} has a unique solution $ \tilde \phi^{n+1} $. Moreover, if $ \phi^n \neq 0 $, then $ \tilde \phi^{n+1} \neq 0 $. 
	\end{proposition}
	\begin{proof}
		The existence and uniqueness of the solution follows from standard elliptic theory. We shall only show that $ \tilde \phi^{n+1} \neq 0 $. Suppose that $ \tilde \phi^{n+1} = 0 $. It follows from \cref{GFBF} that
		\begin{equation}\label{eq:123}
			\left( \frac{1}{\tau} - V + |\phi^n|^{2 \alpha} \right) \phi^n = 0. 
		\end{equation}
		Multiplying both sides of \cref{eq:123} by $ \phi^n $ and integrating over $ \R^d $ to obtain
		\begin{equation*}
			\int_{\R^d} \left[ \frac{|\phi^n|^2}{\tau} - V |\phi^n|^2 + |\phi^n|^{2\alpha+2} \right] \rmd \vx = 0. 
		\end{equation*}
		Since $ I_\omega(\phi^n) = 0 $, one has
		\begin{equation*}
			\int_{\R^d} \left[ \frac{|\phi^n|^2}{\tau} + |\nabla \phi^n|^2 + \omega |\phi^n|^2 \right] \rmd \vx = 0, 
		\end{equation*}
		which implies $ \phi^n = 0 $, which contradicts the assumption. 
	\end{proof}
	We stop the iteration when 
	\begin{equation}\label{stopping}
		\frac{\| \phi^{n+1} - \phi^n \|}{\tau} \leq \vep, 
	\end{equation}
	where $ \vep>0 $ is a small constant depending on the error tolerance and $ \| \cdot \| $ can be chosen as any reasonable norm in practical computation. 
	
	Note that if the initial data $ \phi_0 \in \mathcal{F}_\omega $ in \cref{endGFBF} (i.e. $ \phi_0 $ is a nontrivial solution to \cref{SNLS}), then $ \phi^1 = \tilde \phi^{1} = \phi^0 $ and the GFDN-BF will converge after one step. Moreover, we have the following result about the accuracy of the GFDN-BF scheme, which will imply that the error is independent of the time step size $ \tau $ we choose. 
	\begin{proposition}\label{independence of time step size}
		Suppose that $ \omega>\omega_0 \geq 0 $. If the GFDN-BF scheme converges, i.e. there exits $ \phi^\infty \in H^1(\R^d) $ such that $ \phi^{n} \rightarrow \phi^\infty $ in $ H^1(\R^d) $ as $ n \rightarrow \infty $, then $ \phi^\infty \in \mathcal{F}_\omega $.  
	\end{proposition}
	\begin{proof}
		First note that by \cref{boundedaway} and $ \phi^n \in \mathcal{N}_\omega $, one has $ \| \phi^\infty \|_{H^1(\R^d)} = \lim_{n} \| \phi^n \|_{H^1(\R^d)} > 0 $. Let $ \phi^{n+1} = \phi^n + \delta_n $ with $ \delta_n \rightarrow 0 $ in $ H^1(\R^d) $. Note that $ \phi^{n+1} = \lambda_{n+1} \tilde \phi^{n+1} $ and multiply $ -\lambda_{n+1} $ on both sides of \cref{GFBF} to obtain
		\begin{equation}\label{eq:eq}
			\frac{\lambda_{n+1}-1}{\tau} \phi^n - \frac{\delta_n}{\tau} 
			= (-\Delta + \omega) \delta_n + (1-\lambda_{n+1}) (-\Delta + \omega) \phi^n + \lambda_{n+1} ( - \Delta \phi^n + \omega \phi^n + V \phi^n - |\phi^n|^{2\alpha} \phi^n). 
		\end{equation}
		Multiplying both sides of \cref{eq:eq} by $ \phi^n $, integrating over $ \R^d $ and noting that $ I_\omega(\phi^n) = 0 $ to obtain
		\begin{equation*}
			\frac{\lambda_{n+1}-1}{\tau} \| \phi^n \|_{L^2(\R^d)}^2 - \frac{1}{\tau} \lrang{\delta_n}{\phi^n}
			= (1-\lambda_{n+1})\left(\| \nabla \phi^n \|_{L^2(\R^d)}^2 + \omega \| \phi^n \|_{L^2(\R^d)}^2 \right) + \lrang{\nabla \delta_n}{\nabla \phi^n} + \omega \lrang{\delta_n}{\phi^n}. 
		\end{equation*}
		It follows that
		\begin{equation}\label{eq:eq2}
			\lim_{n \rightarrow \infty} \frac{\lambda_{n+1}-1}{\tau}  \left[ \| \nabla \phi^n \|_{L^2(\R^d)}^2 + \left(\omega+\frac{1}{\tau} \right) \| \phi^n \|_{L^2(\R^d)}^2 \right] = 0.  
		\end{equation}
		Since $ \phi^n \rightarrow \phi^\infty \neq 0 $ in $ H^1(\R^d) $, one has $ \lambda_{n} \rightarrow 1 $. Let $ n \rightarrow \infty $ in \cref{eq:eq} and then one has
		\begin{equation*}
			0 = -\Delta \phi^\infty + V \phi^\infty + \omega \phi^\infty - |\phi^\infty|^{2 \alpha} \phi^\infty \text{ in } H^{-1}(\R^d), 
		\end{equation*}
		which concludes the proof. 
	\end{proof}
	
	\begin{remark}
		As a consequence of \cref{independence of time step size}, if $ \phi^{n+1} = \phi^n $ for some $ n $ in the GFDN-BF scheme \cref{GFBF}-\cref{endGFBF}, then $ \phi^n \in \mathcal{F}_\omega $, which implies $ \phi^n $ will solve \cref{SNLS} exactly without any error term depending on $ \tau $. Note that this is different from what is observed for least energy ground state, where the Lagrangian multiplier must be involved for such property to hold \cite{liu2021}. 
	\end{remark}
	
	Then we show that the GFDN-BF is positivity preserving. 
	
	\begin{proposition}\label{prop:positivity_preserving}
		Suppose that $ \omega > \omega_0 \geq 0 $. If $ \phi_0 \geq 0 $ in \cref{endGFBF}, then $ \phi^n \geq 0 $ for all $ n \geq 0 $. 
	\end{proposition}
	
	\begin{proof}
		One notes that \cref{GFBF} can be rewritten as 
		\begin{equation*}
			-\Delta \tilde \phi^{n+1} + \left( \frac{1}{\tau} + \omega \right) \tilde \phi^{n+1} = \left( \frac{1}{\tau} - V + |\phi^n|^{2\alpha} \right) \phi^n. 
		\end{equation*}
		If $ \phi^n \geq 0 $, then it follows from the weak maximum principle (see, e.g., Theorem 8.1 of \cite{elliptic}) that $ \tilde \phi^{n+1} \geq 0 $, which implies that $ \phi^{n+1} \geq 0 $. 
	\end{proof}
	
	Most importantly, motivated by the very recent paper \cite{liu2022}, we are able to show that the GFDN-BF diminishes the action unconditionally. 
	\begin{proposition}\label{prop:energy_stable}
		Suppose that $ \omega > \omega_0 \geq 0 $. For any $ \tau > 0 $, let $ \phi^n (n \geq 0)$ be obtained from the GFDN-BF \cref{GFBF}-\cref{endGFBF}. Then one has
		\begin{equation*}
			S_\omega(\phi^{n+1}) \leq S_\omega(\phi^n), \quad  n \geq 0. 
		\end{equation*}
	\end{proposition}
	
	\begin{proof}
		Note that $ \phi^{n+1} = \lambda_{n+1} \tilde \phi^{n+1} $. Multiplying both sides of \cref{GFBF} by $ \lambda_{n+1} $, one has for $ n \geq 0 $, 
		\begin{equation}\label{eq:add}
			\frac{\phi^{n+1}-\lambda_{n+1} \phi^n}{\tau} = \Delta \phi^{n+1} - \omega \phi^{n+1} - \lambda_{n+1} V  \phi^{n} +  \lambda_{n+1} |\phi^n|^{2\alpha} \phi^{n}. 
		\end{equation}
		Multiplying both sides of \cref{eq:add} by $ 2\overline{(\phi^{n+1} - \lambda_{n+1} \phi^n)} $, taking the real part and integrating over $ \R^d $, one has
		\begin{equation*}
			\begin{aligned}
				\frac{2}{\tau} \| \phi^{n+1} - \lambda_{n+1} \phi^n \|_{L^2}^2 
				&\leq - \| \nabla \phi^{n+1} \|_{L^2}^2 - \omega \| \phi^{n+1} \|_{L^2}^2 - \int_{\R^d} V |\phi^{n+1}|^2 \rmd \vx \\
				&\quad + \lambda_{n+1}^2 \| \nabla \phi^n \|_{L^2}^2  + \omega \lambda_{n+1}^2 \| \phi^{n} \|_{L^2}^2 + \lambda_{n+1}^2 \int_{\R^d} V |\phi^{n}|^2 \rmd \vx \\
				&\quad + \int_{\R^d} |\phi^n|^{2\alpha}|\phi^{n+1}|^2 \rmd \vx - \lambda_{n+1}^2 \| \phi^n \|_{L^{2\alpha+2}}^{2\alpha+2}, 
			\end{aligned}
		\end{equation*}
		where we use the fact that $ \omega > 0 $, $ V \leq 0 $ and the identity
		\begin{equation*}
			\text{Re} \left[ 2z_1\overline{(z_1-z_2)} \right] = |z_1|^2 - |z_2|^2 + |z_1-z_2|^2, \quad z_1, z_2 \in \C. 
		\end{equation*}
		By Young's inequality, 
		\begin{equation*}
			\int_{\R^d} |\phi^n|^{2\alpha}|\phi^{n+1}|^2 \rmd \vx \leq \int_{\R^d} \left(\frac{\alpha}{\alpha+1}|\phi^n|^{2\alpha+2} + \frac{1}{\alpha+1}|\phi^{n+1}|^{2\alpha+2} \right) \rmd \vx. 
		\end{equation*}
		It follows that
		\begin{equation*}
			\begin{aligned}
				0 \leq \frac{2}{\tau} \| \phi^{n+1} - \lambda_{n+1} \phi^n \|_{L^2}^2  
				&\leq -S_\omega(\phi^{n+1}) + \lambda_{n+1}^2 I_\omega(\phi^n) + \frac{\alpha}{\alpha+1} \| \phi^n \|_{L^{2\alpha+2}}^{2\alpha+2}. 
			\end{aligned}
		\end{equation*}
		Since $ \phi^n \in \mathcal{N}_\omega $, one has $ I_\omega(\phi^n) = 0 $ and by \cref{S1S2}, 
		\begin{equation*}
			S_\omega(\phi^n) = \frac{\alpha}{\alpha+1} \| \phi^n \|_{L^{2\alpha+2}}^{2\alpha+2}. 
		\end{equation*}
		Hence, 
		\begin{equation*}
			\begin{aligned}
				S_\omega(\phi^{n+1}) \leq S_\omega(\phi^n). 
			\end{aligned}
		\end{equation*}
	\end{proof}
	
	\begin{remark}\label{rem:modify}
		If $ \omega_0 < 0 $ or assumption \ref{rq3'} instead of \ref{rq3} is assumed, we can modify \cref{GFBF} as
		\begin{equation}
			\frac{\tilde \phi^{n+1}-\phi^n}{\tau} = \Delta \tilde \phi^{n+1} - (\omega + \theta) \tilde \phi^{n+1} - (V - \theta) \phi^{n} + |\phi^n|^{2\alpha} \phi^{n} \text{ on } \R^d,  \quad n \geq 0,  \label{GFBF_modified}
		\end{equation}
		where $ \theta $ is a parameter satisfying $ \theta + \omega_0 > 0 $ and $ V - \theta \leq 0 $. The same properties as \cref{well-definedness of GFBF,independence of time step size,prop:positivity_preserving,prop:energy_stable} hold for this scheme. 
	\end{remark}
	
	\begin{remark}\label{rem:generalization}
		Although, in this paper, we only consider the single power nonlinearity, the GFDN-BF scheme is also applicable to some combined nonlinearities such as the combined power-type nonlinearity: $ \beta_1 |\phi|^{2\alpha_1}\phi - |\phi|^{2\alpha_2} \phi $ with $ 0 < \alpha_1 < \alpha_2 $ and $ \beta_1 \in \R $ (note that we require the leading nonlinear term to be focusing). The main difference is that at the normalization step \cref{endGFBF}, the explicit form of $ \lambda_{n+1} $ is no longer valid. To obtain $ \lambda_{n+1} $, one has to solve an algebraic equation of the form: $ \lambda^{2\alpha_2} - a\beta_1 \lambda^{2\alpha_1} = b $, where $ a, b > 0 $. It is easy to see that there always exists a positive solution to this algebraic equation since $ \alpha_2 > \alpha_1 $. If $ \beta_1 < 0 $, this solution is also unique and the GFDN-BF is still unconditionally action diminishing. If $ \beta_1>0 $, one may choose the one that minimizes the action, and in order for the GFDN-BF to diminish the action at each step, the time step size $ \tau $ need to be suitably small. 
	\end{remark}
	
	For the comparison purpose, we also introduce the following three schemes modified from corresponding schemes used for computing the least energy ground state. Since \cref{endGFBF} is shared by all of these schemes, we shall only show the counterpart of \cref{GFBF} for brevity. 
	\begin{enumerate}[label = \roman*)]
		\item GFDN with linearized backward Euler discretization (GFDN-BE) \cite{bao2004}: 
		\begin{equation*}
			\frac{\tilde \phi^{n+1}-\phi^n}{\tau} = \Delta \tilde \phi^{n+1} - V \tilde \phi^{n+1} - \omega \tilde \phi^{n+1} +  |\phi^n|^{2\alpha} \tilde \phi^{n+1} \text{ on } \R^d,  \quad n \geq 0. \label{GFBE}
		\end{equation*}
		
		\item Projected GFDN with backward-forward Euler discretization (PGFDN-BF) \cite{cai2021, liu2021}: 
		\begin{equation*}
			\frac{\tilde \phi^{n+1}-\phi^n}{\tau} = \Delta \tilde \phi^{n+1} - \omega \tilde \phi^{n+1} - V \phi^{n} +  |\phi^n|^{2\alpha} \phi^{n} + \mu_{\phi^n} \phi^n \text{ on } \R^d,  \quad n \geq 0, \label{GFLM}
		\end{equation*}
		where
		\begin{equation}\label{eq:mu_phi_PGFDN}
			\mu_{\phi^n}= -\frac{\| H_\omega \phi^n \|^2_{L^2} - (\alpha+2) \lrang{H_\omega \phi^n}{|\phi^n|^{2\alpha} \phi^n} + (\alpha+1) \| \phi^n \|_{L^{4\alpha + 2}}^{4\alpha + 2}}{\alpha  \| \phi^n \|_{L^{2\alpha+2}}^{2\alpha+2}}. 
		\end{equation}
		
		\item GFDN with time splitting (GFDN-TS) \cite{bao2004}: 
		\begin{equation*}
			\tilde \phi^{n+1} = S_1^{\tau/2} S_2^\tau S_1^{\tau/2} \phi^n, \quad n \geq 0, 
		\end{equation*}
		where $ S_1^t = e^{t \Delta} $ and $ S_2^t $ is the flow associated with
		\begin{equation}\label{sp2}
			\partial_t \phi = - V \phi - \omega \phi + |\phi|^{2\alpha} \phi \text{ on } \R^d, \quad 0<t<\tau.  
		\end{equation}
		Actually, we are able to solve \cref{sp2} exactly. Then $ S_2 $ is explicitly given by 
		\begin{equation*}
			[S_2^\tau \phi](\vx) = 
			\left\{
			\begin{aligned}
				&\left[1 - \alpha \tau |\phi(\vx)|^{2\alpha}\right]^{-\frac{1}{2\alpha}} \phi(\vx), & V(\vx) + \omega = 0, \\
				&\left[\frac{(V(\vx) + \omega)e^{-2\alpha\tau(V(\vx)+\omega)}}{V(\vx)+\omega - (1 - e^{-2\alpha\tau(V(\vx)+\omega)})|\phi(\vx)|^{2\alpha}}\right]^{\frac{1}{2\alpha}} \phi(\vx), & V(\vx) + \omega \neq 0. 
			\end{aligned}
			\right.
		\end{equation*}
	\end{enumerate}
	We remark that for the PGFDN-BF scheme, an immediate drawback is that for $ \mu_{\phi^n} $ to be well-defined, $ \phi^n $ is supposed to be in $ H^2 $ according to \cref{eq:mu_phi_PGFDN}, which is not satisfied when $ V $ is the delta potential or the inverse power potential. 
	
	\subsection{Full-discretization}\label{sec:full_discretization}
	We shall introduce in this subsection the full discretization of the GFDN-BF scheme. 
	
	Since the original problem is posed on the whole space $ \R^d $, we shall first truncate the whole space problem onto a bounded domain $ \Omega $ (usually chosen as $ \Pi_{j=1}^d [a_j, b_j] $) with homogeneous Dirichlet boundary condition or periodic boundary condition. For simplicity, we only show the truncated GFDN-BF scheme with homogeneous Dirichlet boundary condition. All the functionals and the Nehari manifold previously defined can be defined similarly by replacing $ \R^d $ with $ \Omega $ and we shall use the same notation for them. Then the GFDN-BF scheme  reads: 
	\begin{align}
		&\frac{\tilde \phi^{n+1}-\phi^n}{\tau} = \Delta \tilde \phi^{n+1} - \omega \tilde \phi^{n+1} - V \phi^{n} +  |\phi^n|^{2\alpha} \phi^{n} \text{ on } \Omega,  \quad n \geq 0 \label{GFBF_t}\\ 
		&\tilde \phi^{n+1}(\vx) = 0, \quad \vx \in \partial \Omega, \label{GFBF_t_end}\\
		&\phi^{n+1} = \lambda_{n+1} \tilde \phi^{n+1}, \quad \lambda_{n+1} = \left( \frac{I_\omega^\ast(\tilde \phi^{n+1}) }{  \| \tilde \phi^{n+1} \|^{2\alpha+2}_{L^{2\alpha+2}(\Omega)}} \right)^{\frac{1}{2\alpha}}, \quad \phi^0(\vx) = \phi_0(\vx), \quad \vx \in \Omega, \label{GFBF_t_end2}
	\end{align}
	with $ \phi_0 \in \mathcal{N}_\omega $. 
	
	Note that in \cref{GFBF_t}-\cref{GFBF_t_end}, one only need to solve a linear elliptic equation with constant coefficient (remaining the same at each step) and potentially singular source. 
	
	In practical computation, one shall choose the spatial discretization properly according to the potential function under consideration. Here, we suggest three standard spatial discretization methods: finite difference method (FD), finite element method (FE) and sine pseudospectral method (SP). Although our model includes many singular potentials as shown before, we will not go into detail to talk about specific methods for different types of singularities but leave it as our future work. In this paper, we only use FE method to deal with them. Besides, for the FE and FD method, instead of using a uniform grid, one may use a fine mesh around the singularity of the potential and use a coarse mesh where the solution is exponentially decaying. Considering that FD, FE and SP are all quite standard, we will omit the detail of the numerical schemes for brevity. 
	
	\textbf{Finite difference method (FD) for $ C^2 $ potentials.}
	When the potential function $ V \in C^2(\R^d) $, one can use finite difference method for spatial discretization. At each step, a sparse linear system with the same coefficient matrix needs to be solved. 
	
	\textbf{Sine pseudospectral method (SP) for sufficiently smooth potentials.} 
	When the potential function is sufficiently smooth (e.g. the zero potential and the smooth potential in \cref{exmp:poten}), one can use SP method for spatial discretization. At each step, a constant coefficient linear system needs to be solved, which can be done efficiently by FFT. 
	
	\textbf{Finite element method (FE) for some singular potentials.} 
	Finite element method has the weakest regularity requirement and thus is suitable for the singular potentials such as the delta potential, well potential and inverse power potential in \cref{exmp:poten}. Similar to the FD method, a sparse linear system with the same coefficient matrix needs to be solved at each step. 
	
	It follows immediately that the linear system obtained by discretizing \cref{GFBF_t}-\cref{GFBF_t_end} with FE, FD or SP (if applicable) has a unique solution for any $ \tau > 0 $ when $ \omega>\omega_0\geq0 $.
	
	\section{Numerical results}\label{sec:numerical results}
	In this section, we will first compare the four semi-discretization schemes introduced in \cref{sec:time discretization}. Then we report some numerical results for the least action ground states with different potentials and in different spatial dimensions. Besides, we use our algorithm to verify some existing results and make some conjectures. Throughout this section, we equip $ \Omega $ with homogeneous Dirichlet boundary condition. In most of the cases, the initial data is chosen as 
	\begin{equation*}
		\phi_0 = \lambda_\omega(u_0) u_0 \in \mathcal{N}_\omega \quad \text{with} \quad u_0(\vx) = e^{-\frac{|\vx|^2}{2}}, \quad \vx \in \Omega, 
	\end{equation*}
	where $ \lambda_\omega $ is defined in \cref{eq:lambda}. However, to observe the symmetry breaking bifurcation, some shifted Gaussian type initial data are used (see \mbox{\cref{exmp:smooth}}). 
	\subsection{Comparison of different semi-discretization schemes}\label{sec:comparison}
	We shall compare the four semi-discretization schemes introduced in \cref{sec:time discretization} comprehensively including the computation time, the robustness for large time step size, the accuracy and the action decreasing property. 
	\begin{exmp}\label{exmp:comparison}
		We take $ d = 1 $ and $ V(\vx) \equiv 0 $, i.e, the non potential case in 1D. Note that, for this potential, the exact solution is known and the unique positive, radial, decreasing solution is given by \cref{exact_non_potential}. In this case, $ \omega_0 = 0 $. We solve the least action ground state of \cref{SNLS} with $ \alpha = 1 $ for different $ \omega = 0.1, 1, 10 $. We consider the following three cases: 
		\begin{description}
			\item[Case I.] $ \omega = 1 $, $ \Omega = [-32, 32] $, $ h = 2^{-4} $, $ \vep = 10^{-9} $, $ \tau = 1 $. 
			\item[Case II.] $ \omega = 0.1 $, $ \Omega = [-64, 64] $, $ h = 2^{-4} $, $ \vep = 10^{-9} $, $ \tau = 1 $.  
			\item[Case III.] $ \omega = 10 $, $ \Omega = [-16, 16] $, $ h = 2^{-4} $, $ \vep = 10^{-9} $, $ \tau = 1 $. 
		\end{description}
	\end{exmp}
	Since the potential is smooth, we use SP method for spatial discretization. In computation, we use discrete sine transform to solve the linear system in GFDN-BF and PGFDN-BF scheme, and use some iterative methods to solve the linear system in GFDN-BE. Since GFDN-TS is explicit, we do not need to solve any linear system. 
	The numerical results of the computation time, the iteration steps and the relative error of the GFDN-BF, GFDN-BE, PGFDN-BF and GFDN-TS are shown in Table \ref{table:comparison}, where the number of iteration steps is $ n $ if the stopping criteria is satisfied by $ \phi^n $ and the relative error of the numerical solution $ \phi^n $ is computed by
	\begin{equation*}
		\frac{\max |\phi^n-  \phi_{\text{exact}}|}{\max |\phi_{\text{exact}}|}. 
	\end{equation*}
	We also show the evolution of $ S_\omega $ in Case I for all the four schemes with different time step size $ \tau $ in \cref{comparison_diff_dt}. Then we can draw the following conclusion: 
	\begin{enumerate}
		\item GFDN-BF performs well in all the cases. The accuracy of the GFDN-BF scheme is independent of the time step size we choose, which is illustrated in \cref{independence of time step size} and observed in Table \ref{table:comparison}. It is very robust for extremely large time step size: the action is decreasing and the algorithm converges even for very large $ \tau $ as shown in \cref{comparison_diff_dt}. It is also very efficient since the linear system is obtained by discretizing a linear elliptic equation with constant coefficient and the coefficient remains the same at each step. This linear system will be sparse when choosing FE or FD method for spatial discretization. When using SP method for spatial discretization, it can also be solved efficiently by FFT. Note that the computational cost at each step is the same and independent of the time step size $ \tau $. Since one can observe from Table \ref{table:comparison} that larger time step size results in less iteration steps, this implies that we can choose sufficiently large $ \tau $ to make the GFDN-BF scheme converge fast. 
		
		\item GFDN-BE is good when time step size is small: $ S_\omega $ is decreasing and the accuracy is also independent of the time step size. However, it is not suitable for large time step and is much slower than GFDN-BF. It can not use large time step since the linear system may be singular when time step size is large. Small time step size will require much more iteration steps for the algorithm to converge. Since the linear system in the GFDN-BE scheme will be different at each step, it is usually solved by some iterative methods and thus the computation time of each step will also depend on the time step size. It seems difficult to compute the computational cost exactly but we can observe from the numerical results that it usually needs much more time than GFDN-BF scheme. 
		
		\item PGFDN-BF scheme show very similar behavior to GFDN-BF in all the cases in \cref{exmp:comparison}. However, as we have already mentioned when introducing this scheme, the major drawback of the PGFDN-BF scheme is the requirement of higher regularity of the solution and the potential function. To be precise, it will require the exact solution to be in $ H^2 $ and the potential function $ V $ to be in $ L^2 $, which is not satisfied when $ V $ is chosen as the delta potential or the inverse power potential. 
		
		\item The GFDN-TS differs from all the other three schemes in at least two aspects. Firstly, the accuracy of the solution depends on the time step size we choose and is of order $ \tau^2 $ as observed in Table \ref{table:comparison}. Secondly, it is a explicit scheme and the computation cost of each iteration mainly comes from the cost of FFT. As shown in Table \ref{table:comparison}, to obtain the same accuracy, it usually takes much more time than other schemes. 
	\end{enumerate}
	In summary, the GFDN-BF scheme is the best one in all the cases and thus we will only use GFDN-BF in our numerical experiments. Though not shown here, similar results can be observed when $ V $ is chosen as other potentials in \mbox{\cref{exmp:poten}}. 
	
	\begin{table}[H]
		\begin{center}
			\subfloat[Case I.]{
				\begin{tabular}{c"cccc}
					\thickhline
					Method              & $ \tau $ & CPU(s) & iteration steps & relative error of $\phi^n$ \\
					\thickhline
					\multirow{3}{*}{GFDN-BF} 
					& 0.01     & 0.32   & 1707            & 6.70E-10                         \\
					& 0.1      & 0.03   & 185             & 6.51E-10                              \\
					& 4        & 0.01   & 28              & 2.35E-09                         \\
					\thickhline
					\multirow{3}{*}{GFDN-BE} 
					& 0.01     & 1.58   & 1622            & 3.33E-09                             \\
					& 0.1      & 0.91   & 183             & 6.18E-10                               \\
					& 1        & -  & -              & -                             \\
					\thickhline
					\multirow{3}{*}{PGFDN-BF} 
					& 0.01     & 0.48   & 1706            & 6.70E-10                               \\
					& 0.1      & 0.05   & 183             & 6.56E-10                                \\
					& 4        & 0.01   & 23              & 1.71E-09                                \\
					\thickhline
					\multirow{3}{*}{GFDN-TS} 
					& 0.001    & 7.49   & 16926           & 1.55E-07                               \\
					& 0.01     & 0.71   & 1693            & 1.56E-05                                \\
					& 0.1      & 0.07   & 170             & 1.63E-03                            \\
					\thickhline
			\end{tabular}}
			
			\subfloat[Case II. ]{
				\begin{tabular}{c"cccc}
					\thickhline
					Method              & $ \tau $ & CPU(s) & iteration steps & relative error of $\phi^n$  \\
					\thickhline
					\multirow{3}{*}{GFDN-BF} 
					& 0.1      & 0.51   & 1455            & 2.10E-08                           \\
					& 1        & 0.05   & 160             & 1.96E-08                            \\
					& 8        & 0.01   & 38              & 2.90E-08                             \\
					\thickhline
					\multirow{3}{*}{GFDN-BE} 
					& 0.1      & 5.69   & 1454            & 2.08E-08                                \\
					& 1        & 4.79   & 156             & 2.09E-08                                \\
					& 8        & -   & -               & -                          \\
					\thickhline
					\multirow{3}{*}{PGFDN-BF} 
					& 0.1      & 0.74   & 1454            & 2.10E-08                           \\
					& 1        & 0.07   & 158             & 1.99E-08                              \\
					& 8        & 0.02   & 35              & 1.71E-08                              \\
					\thickhline
					\multirow{3}{*}{GFDN-TS} 
					& 0.01     & 10.37  & 14408           & 1.61E-07                              \\
					& 0.1      & 0.97   & 1442            & 1.56E-05                              \\
					& 1        & 0.1    & 145             & 1.63E-03                      	   \\
					\thickhline   
			\end{tabular}}
			
			\subfloat[Case III. ]{
				\begin{tabular}{c"cccc}
					\thickhline
					Method              & $ \tau $ & CPU(s) & iteration steps & relative error of $\phi^n$ \\
					\thickhline
					\multirow{3}{*}{GFDN-BF} 
					& 0.01     & 0.04   & 239             & 2.20E-11                                  \\
					& 0.1      & 0.01   & 52              & 2.46E-11                                    \\
					& 4        & 0.01   & 32              & 4.91E-10                              \\
					\thickhline
					\multirow{3}{*}{GFDN-BE} 
					& 0.01     & 0.35   & 215             & 3.27E-10                              \\
					& 0.1      & 0.4    & 43              & 1.87E-11                             \\
					& 1        & -   & -               & -                            \\
					\thickhline
					\multirow{3}{*}{PGFDN-BF} 
					& 0.01     & 0.05   & 237             & 2.06E-11                              \\
					& 0.1      & 0.01   & 45              & 2.17E-11                              \\
					& 4        & 0.01   & 25              & 3.28E-10                             \\
					\thickhline
					\multirow{3}{*}{GFDN-TS} 
					& 0.0001   & 7.78   & 22007           & 1.55E-07                            \\
					& 0.001    & 0.75   & 2201            & 1.56E-05                                \\
					& 0.01     & 0.08   & 220             & 1.61E-03                               \\
					\thickhline
			\end{tabular}}
		\end{center}
		\label{table:comparison}
		\caption{Comparison of different semi-discretization schemes for computing the least action ground states in \cref{exmp:comparison}. }
	\end{table}
	
	\begin{figure}[htbp]
		\centering\vspace{-0.5cm}
		\subfloat{\includegraphics[width=0.475\textwidth]{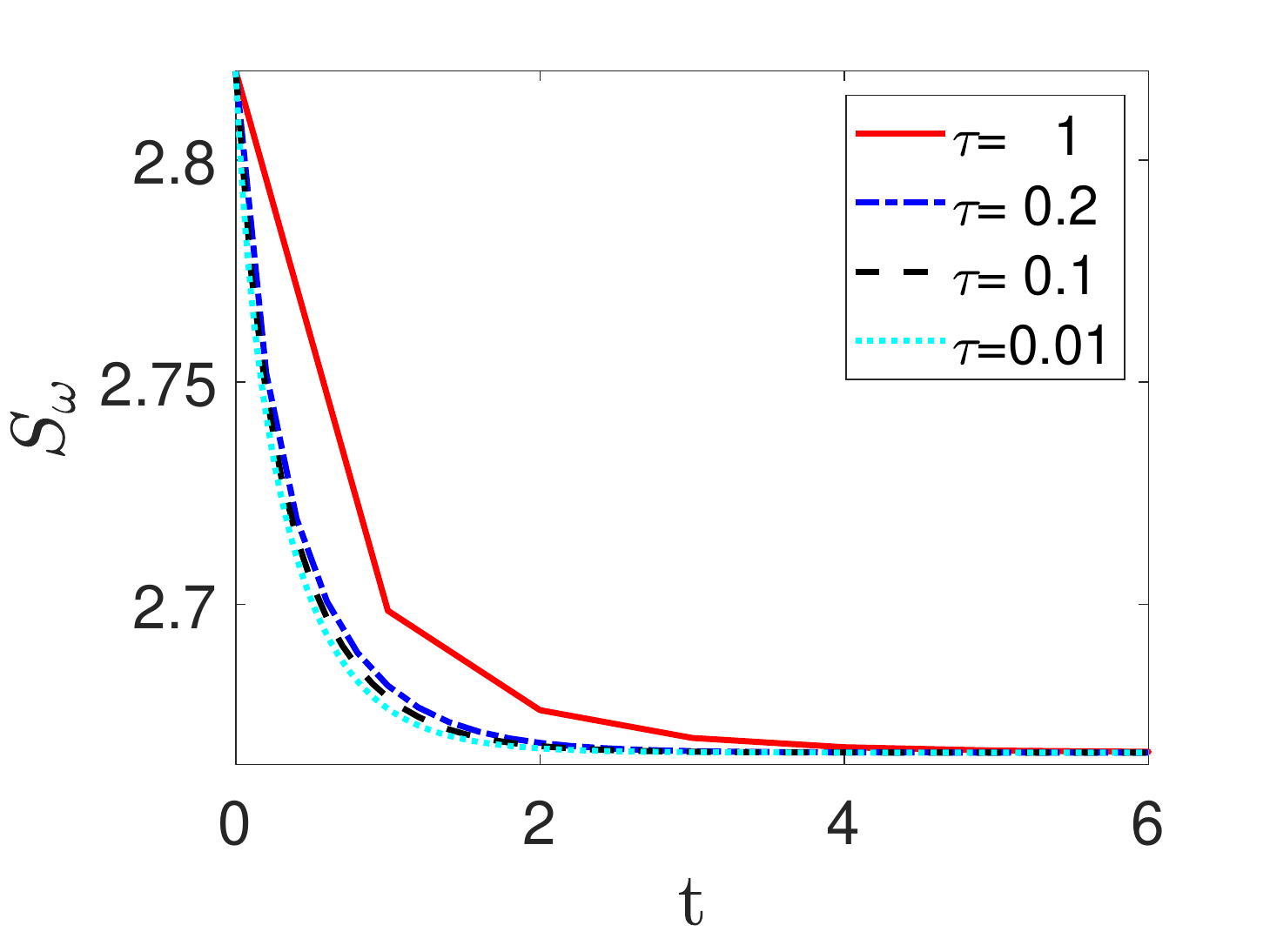}}
		\subfloat{\includegraphics[width=0.475\textwidth]{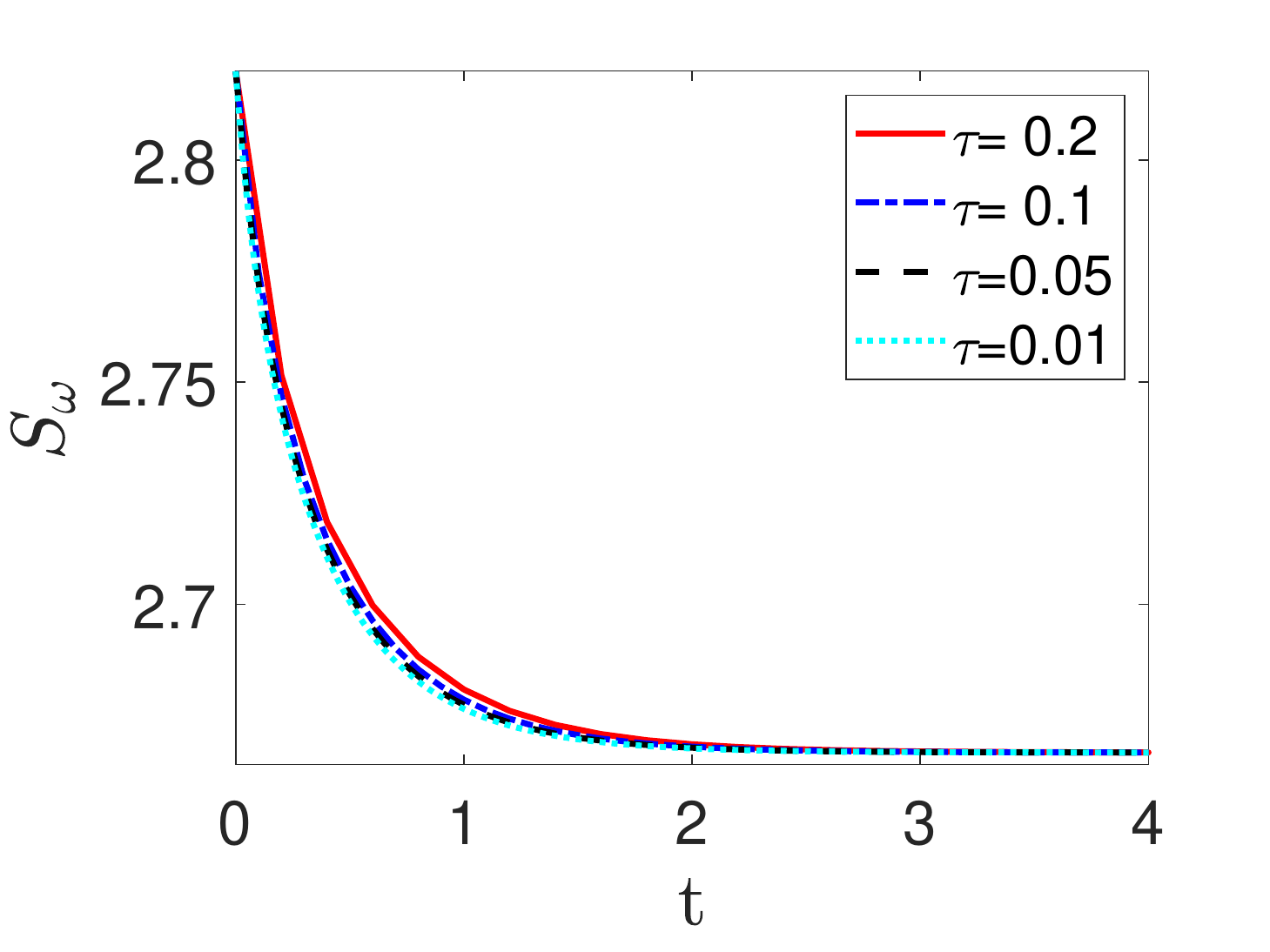}}\\ 
		\subfloat{\includegraphics[width=0.475\textwidth]{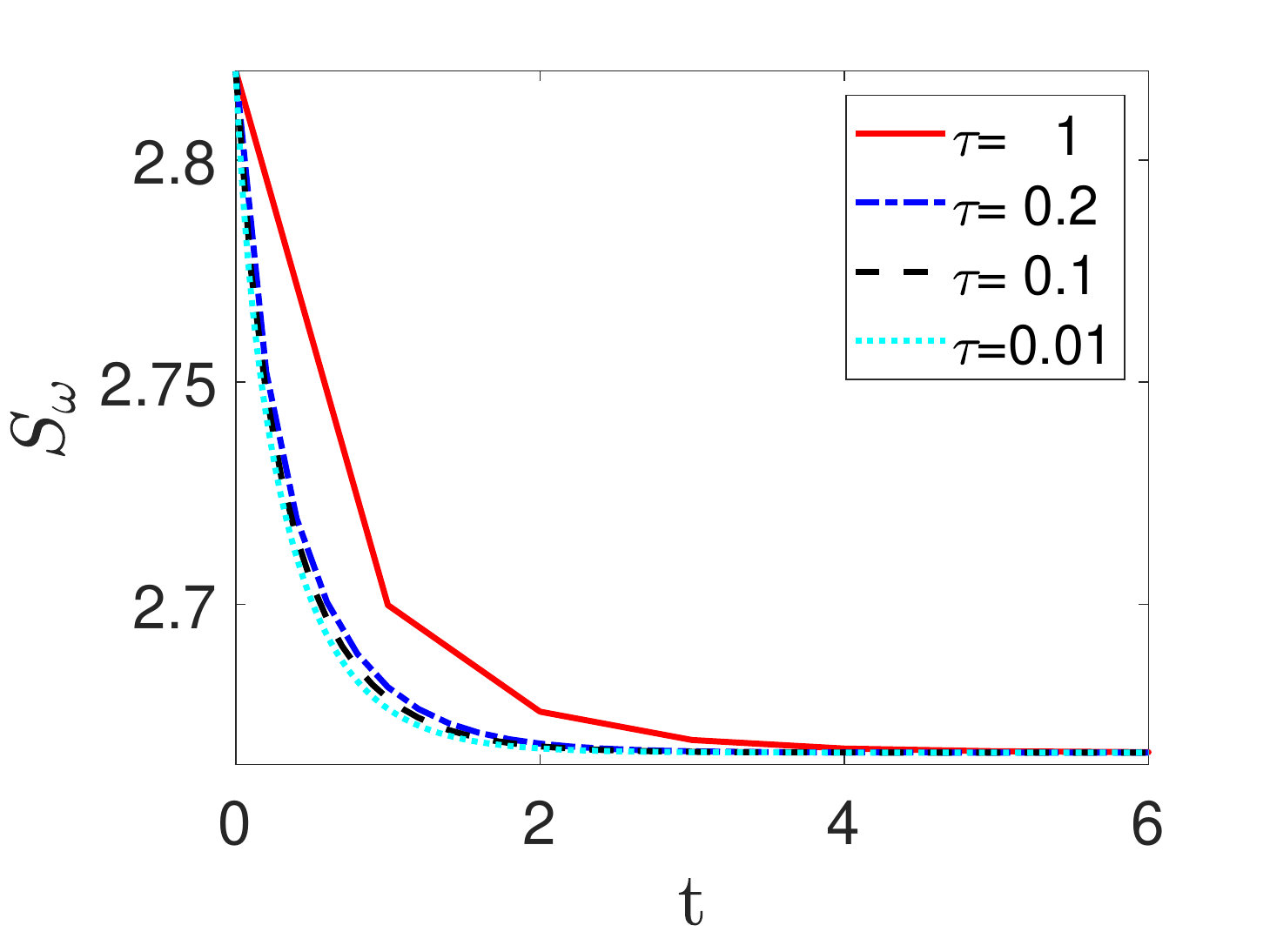}}
		\subfloat{\includegraphics[width=0.475\textwidth]{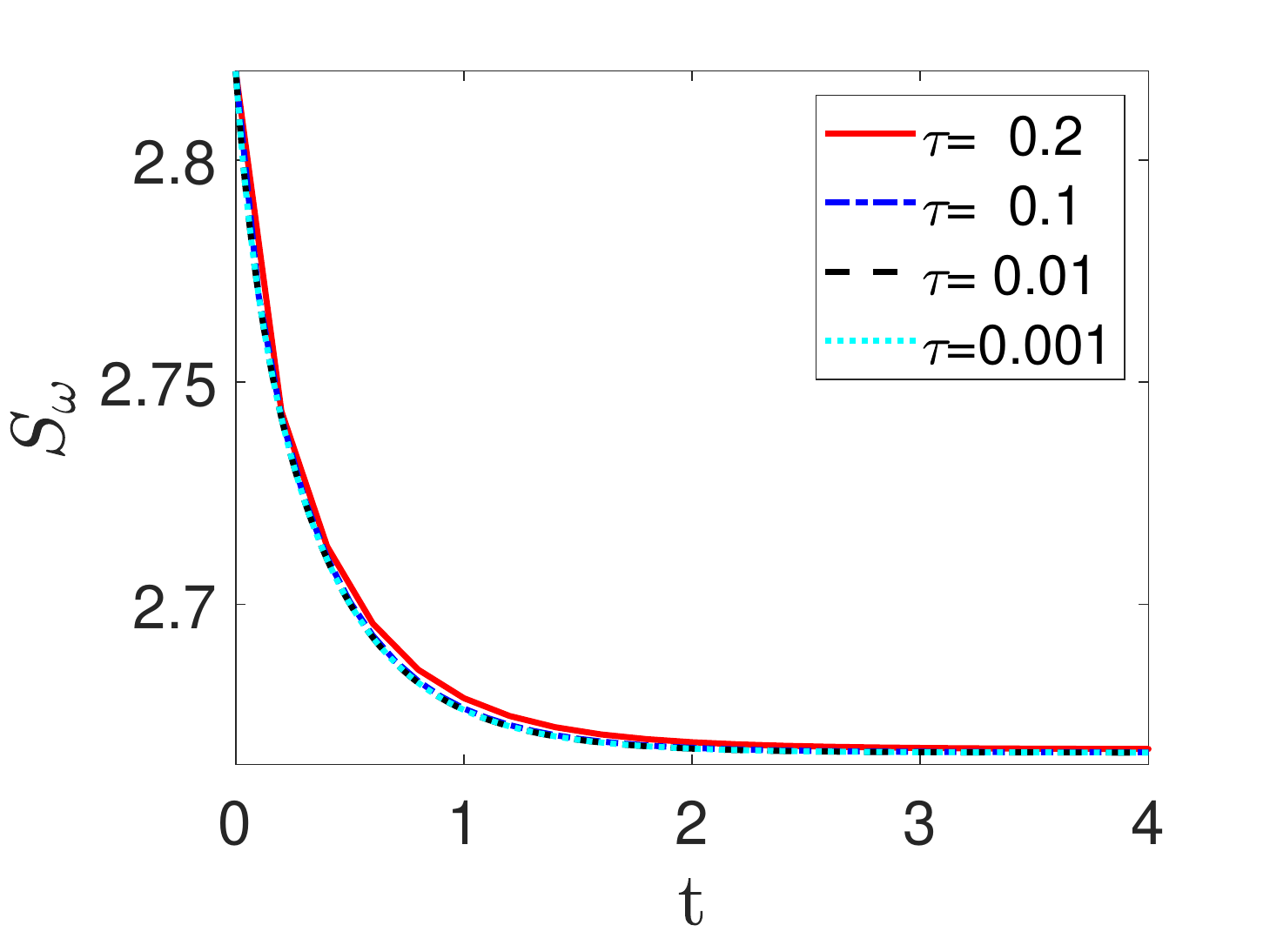}}
		\caption{comparison of the evolution of $ S_\omega $ computed by GFDN-BF (top left), GFDN-BE (top right), PGFDN-BF (bottom left) and GFDN-TS (bottom right) with different time step size $ \tau $ for Case I. in \cref{exmp:comparison}}
		\label{comparison_diff_dt}
	\end{figure}
	
	\subsection{Applications to least action ground state solutions}\label{sec:numerical examples}
	In this subsection, we will show the numerical results for different types of potentials including delta potentials (\cref{exmp:delta}), inverse power potentials (\cref{exmp:inverse_power}), well potentials (\cref{exmp:well}) and some smooth potentials (\cref{exmp:smooth}). These potentials are popular in physical literature and have attracted a lot of attention in mathematics. They are also quite different in terms of the singularity. In each example, we shall show the least action ground state solutions with different $ \alpha $ and $ \omega $ in 1D, 2D or 3D, and the change of $ S_g $, $ M_g $, $ E_g $ with respect to $ \omega $. Although the uniqueness (up to phase shift or translation) of the least action ground states is not known for general potentials (e.g. the potentials considered in \cref{exmp:well,exmp:smooth}), we want to mention that in our numerical experiments, we cannot find two different nonnegative least action ground states associated with the same $ \omega $. 
	
	Before proceeding to the specific examples, we shall first show the convergence rates of the GFDN-BF scheme in space numerically for all types of potentials to be discussed in this subsection. Since we currently only focus on the temporal discretization, we want to use these numerical results to justify our choice of spatial discretization method and to show the reliability of the results to be presented. As we have already mentioned, for the delta potential, the well potential and the inverse power potential, we just use FE method to deal with all of them while for smooth potentials, we use SP method and FD method. The uniform grid is used throughout this subsection. We consider the following potentials: 
	\begin{description}
		\item [I. Delta potential]
		\begin{equation}\label{eq:fulldiscretization_1}
			V(x) = - \delta(x) , \quad \alpha = 1 , \quad \omega = 1 , \quad \tau = 4 , \quad \vep = 10^{-9}, \quad \Omega = [-32, 32].  
		\end{equation}
		\item [II. Inverse power potential]
		\begin{equation}\label{eq:fulldiscretization_2}
			V(x) = -\frac{1}{|x|^{\frac{1}{2}}} , \quad \alpha = 1 , \quad \omega = 2, \quad \tau=4 , \quad \vep = 10^{-9}, \quad \Omega = [-16, 16].  
		\end{equation}
		\item [III. Well potential]
		\begin{equation}\label{eq:fulldiscretization_3}
			V(x) = \left\{
			\begin{aligned}
				&-1, &&|x| \leq 2, \\
				&0, &&|x|>2, 
			\end{aligned}
			\right., \  \alpha = 1, \  \omega = 2, \  \tau=1 , \  \vep = 10^{-9}, \  \Omega = [-16, 16].  
		\end{equation}
		\item [IV. Smooth potential]
		\begin{equation}\label{eq:fulldiscretization_4}
			V(x) = -e^{-|x|^2}, \quad \alpha = 1, \quad \omega = 1 , \quad \tau = 1 , \quad \vep = 10^{-12}, \quad \Omega = [-32, 32]. 
		\end{equation}
	\end{description}
	The numerical results are shown in \cref{fig:delta_convergence,fig:inverse_power_convergence,fig:well_convergence,fig:smooth_convergence} for the four cases above respectively. For the delta potential \cref{eq:fulldiscretization_1} and the well potential \cref{eq:fulldiscretization_3}, standard convergence order can be observed in \cref{fig:delta_convergence,fig:well_convergence}. However, we remark that such convergence can only hold when we include the points where the singularity of the potential function occurs into grid points (e.g. $ x=0 $ for the delta potential and $ x=\pm2 $ for the well potential). For the inverse power potential \cref{eq:fulldiscretization_2}, we can observe only first order convergence in $ H^1 $ norm and second order convergence in $ L^2 $ norm for both linear and quadratic element (see \cref{fig:inverse_power_convergence}). For the smooth potential \cref{eq:fulldiscretization_4}, second order convergence in both $ L^2 $ and $ H^1 $ norm can be observed when using finite difference method and spectral convergence can be observed when using sine pseudospectral method (see \cref{fig:smooth_convergence}). 
	
	
	\begin{figure}[htbp]
		\centering
		\subfloat{\includegraphics[width=0.475\textwidth]{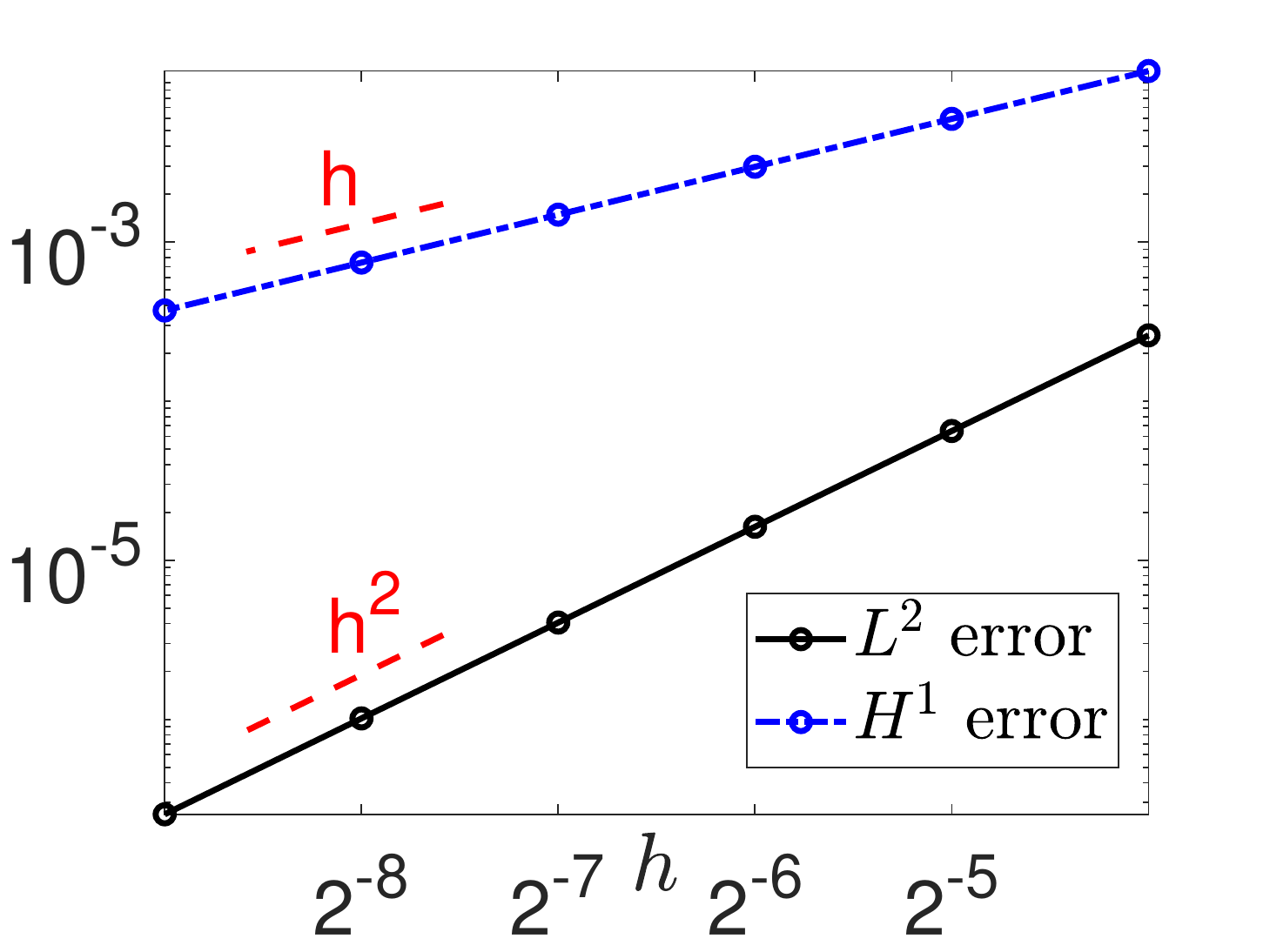}}
		\subfloat{\includegraphics[width=0.475\textwidth]{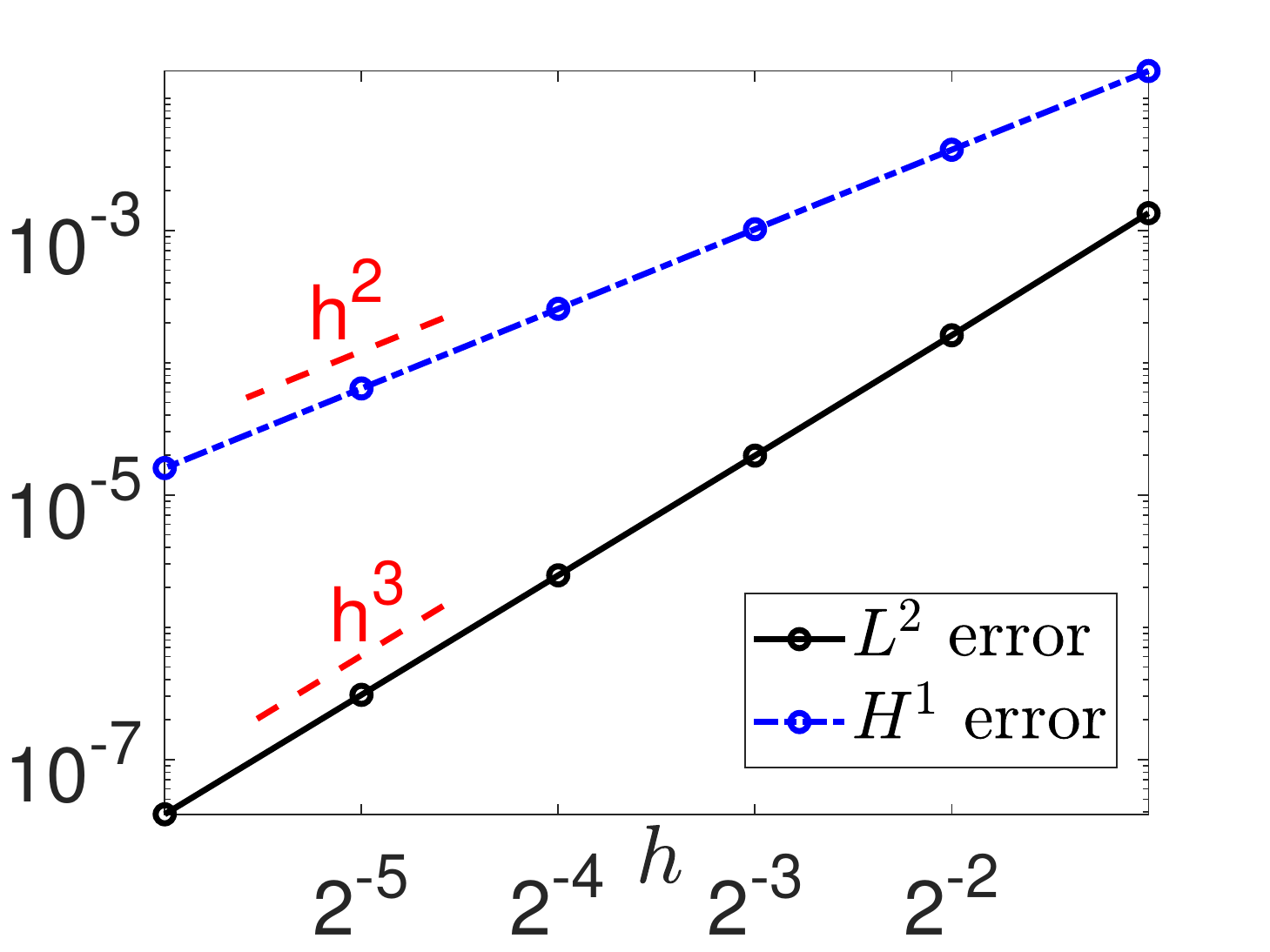}}
		\caption{convergence rate of finite element discretization (linear element (left) and quadratic element (right)) for delta potential \cref{eq:fulldiscretization_1} }
		\label{fig:delta_convergence}
	\end{figure}
	
	\begin{figure}[htbp]
		\centering
		\subfloat{\includegraphics[width=0.475\textwidth]{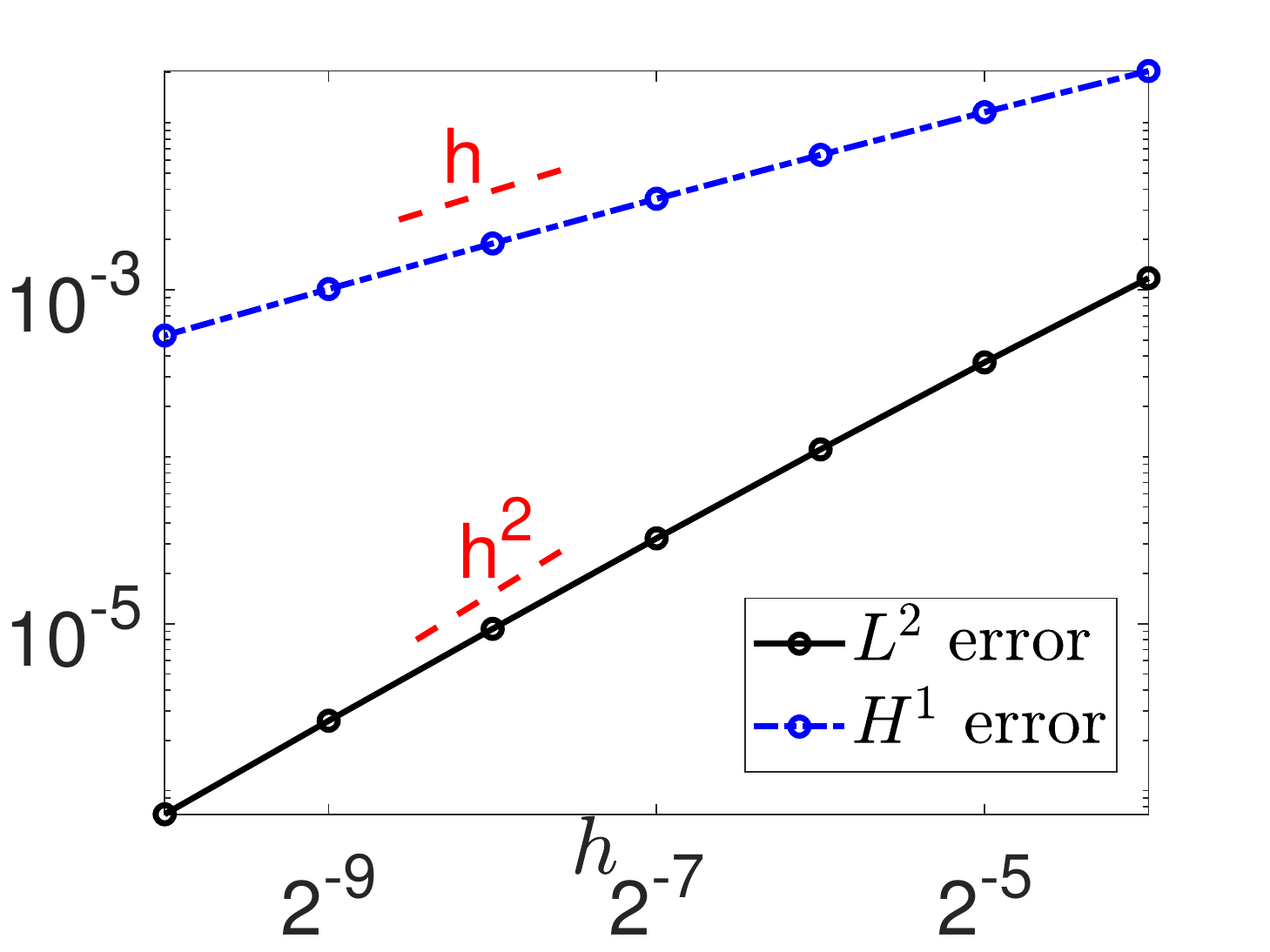}}
		\subfloat{\includegraphics[width=0.475\textwidth]{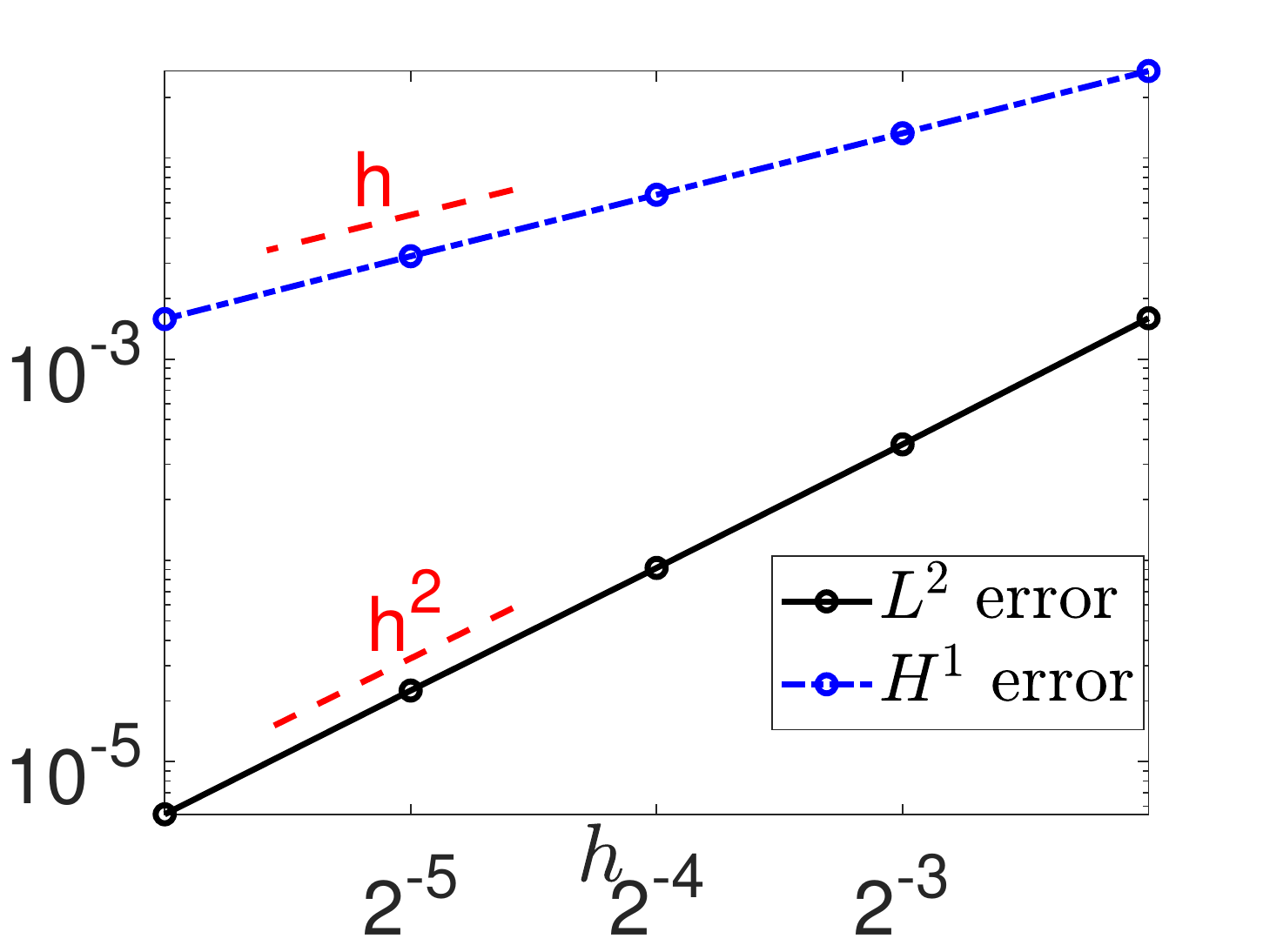}}
		\caption{convergence rate of finite element discretization (linear element (left) and quadratic element (right)) for inverse power potential \cref{eq:fulldiscretization_2} }
		\label{fig:inverse_power_convergence}
	\end{figure}
	
	\begin{figure}[htbp]
		\centering
		\subfloat{\includegraphics[width=0.475\textwidth]{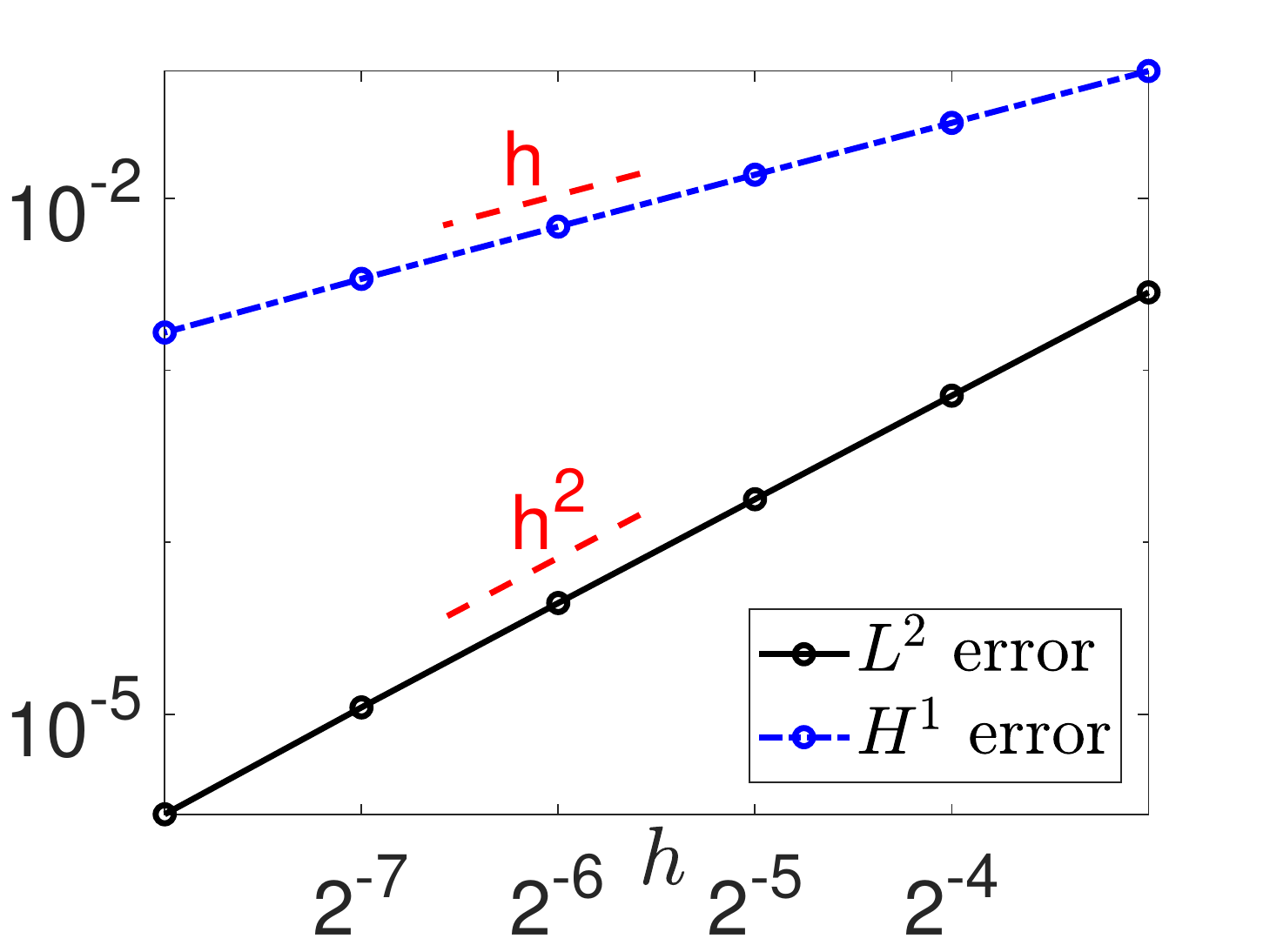}}
		\subfloat{\includegraphics[width=0.475\textwidth]{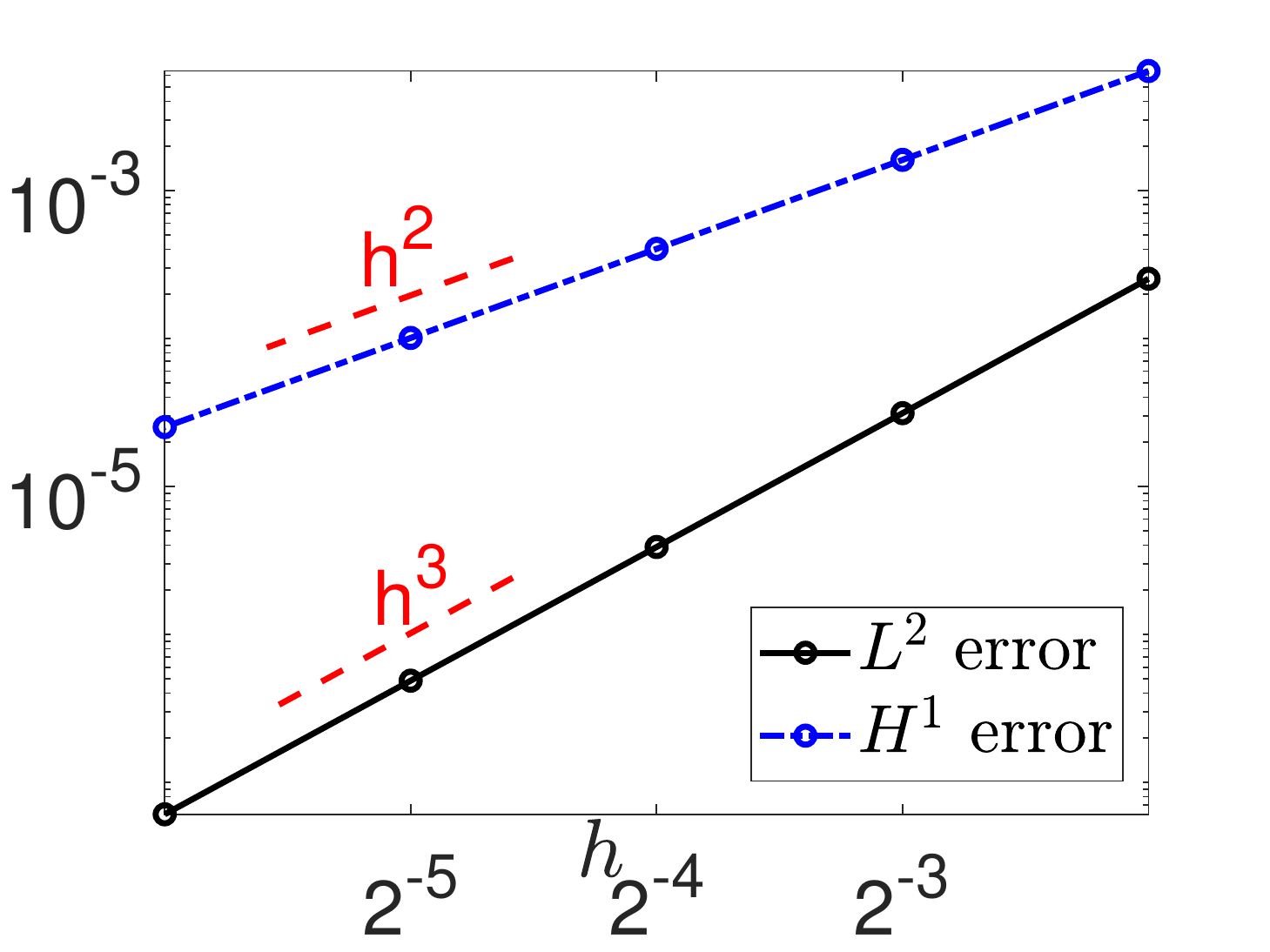}}
		\caption{convergence rate of finite element discretization (linear element (left) and quadratic element (right)) for well potential \cref{eq:fulldiscretization_3} }
		\label{fig:well_convergence}
	\end{figure}
	
	\begin{figure}[htbp]
		\centering
		\subfloat{\includegraphics[width=0.475\textwidth]{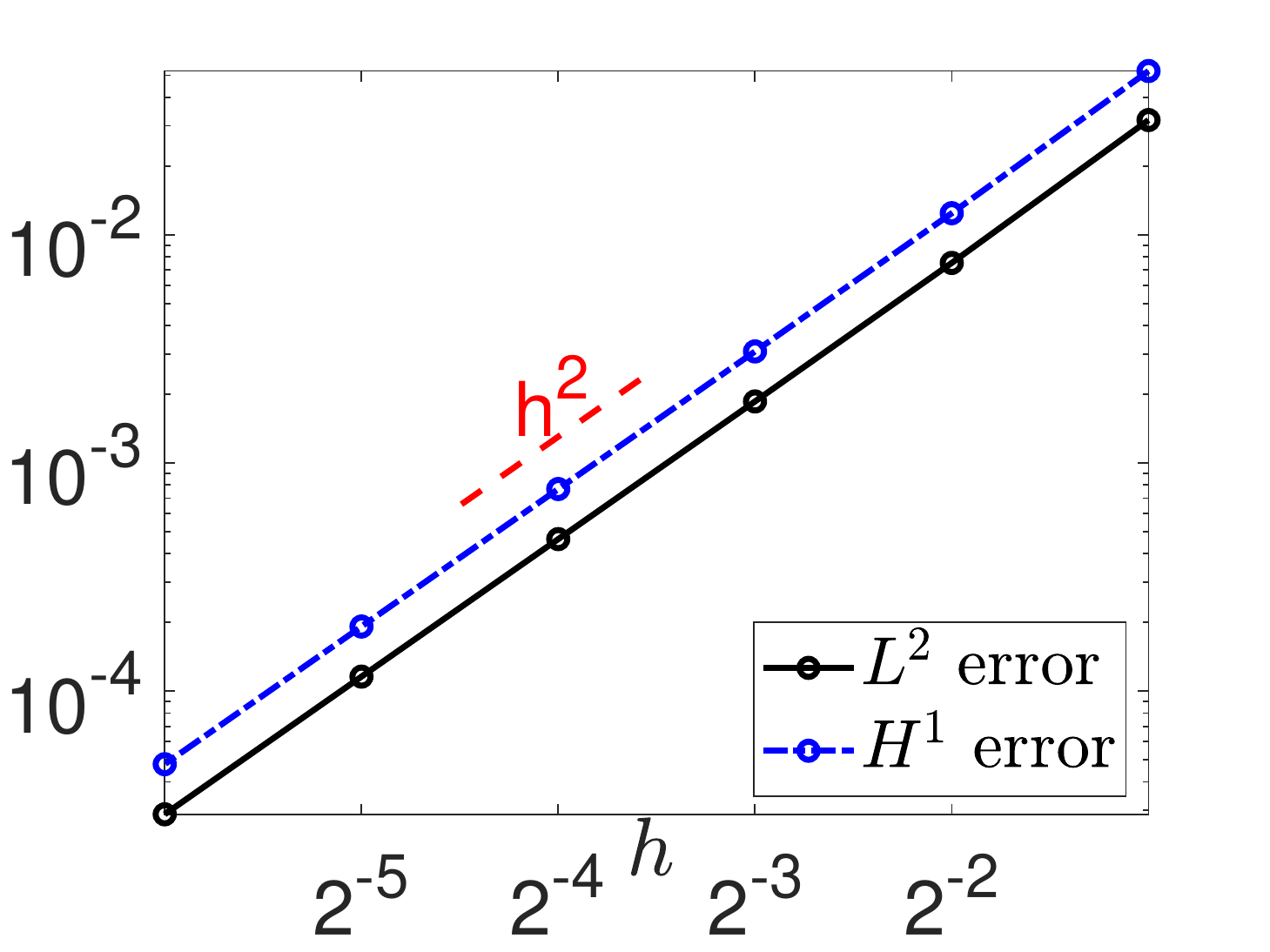}}
		\subfloat{\includegraphics[width=0.475\textwidth]{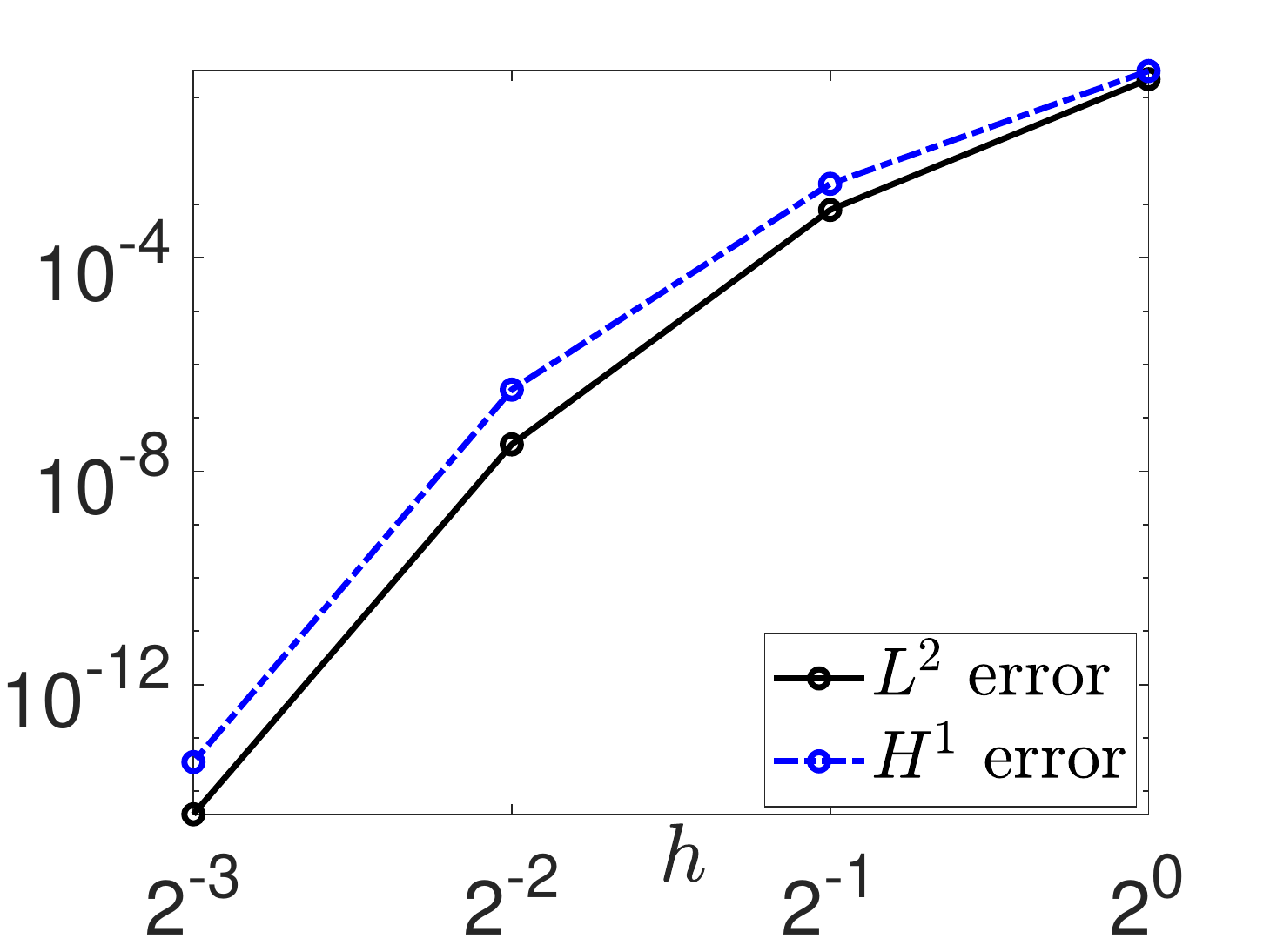}}
		\caption{convergence rate of finite difference method (left) and pseudospectral method (right) for smooth potential \cref{eq:fulldiscretization_4} }
		\label{fig:smooth_convergence}
	\end{figure}
	
	\begin{exmp}\label{exmp:delta}
		Least action ground states with delta potential
		\begin{equation}\label{V_1}
			V(x) = -Z \delta(x), \quad x \in \R. 
		\end{equation}
		It is known that $ \omega_0 = Z^2/4 $ and the unique positive ground state solution is given by \cref{delta_exact}. We consider the following three cases:
		\begin{description}
			\item[Case I. ] $ V = -\delta $, $ \alpha = 1 $, $ \omega_0 = 0.25 $. 
			\item[Case II. ] $ V = -2 \delta $, $ \alpha = 2 $, $ \omega_0 = 1 $. 
			\item[Case III. ] $ V =  - 2[\delta(x-1) + \delta(x) + \delta(x+1)] $, $ \alpha = 3 $, $ \omega_0 \approx 1.9216 $. 
		\end{description}
		
		Before showing the numerical results for all the cases above, we want to take Case III as an example to show the asymptotic results in \cref{lim_omega_0} and \cref{lim_omega_infty}. Let the rescaled function $ \widehat \phi_\omega $ and $ \widecheck \phi_\omega $ be defined as in \cref{rescaling}. The density of $ \widehat \phi_\omega $ and $ \widecheck \phi_\omega $ for different $ \omega > \omega_0 $ are shown in \cref{fig:delta_ground_states_normalized}. It can be clearly seen that $ \widehat \phi_\omega \rightarrow \phi^\text{lin}_0 $ as $ \omega \rightarrow \omega_0 $ and $ \widecheck \phi_\omega \rightarrow \phi_1^0 $ as $ \omega \rightarrow \infty $. 
		
		Then we shall show the least action ground states and change of $ S_g(\omega) $, $ M_g(\omega) $ and $ E_g(\omega) $ for Case I to III. For all the cases, we choose $ \Omega = [-32, -32] $, $ \tau = 1 $, $ \vep = 10^{-9} $, $ h = 2^{-5} $ and use quadratic finite element on a uniform mesh for spatial discretization. The results are shown in \cref{fig:delta_ground_states}. 
		
		We can observe from the left column in \mbox{\cref{fig:delta_ground_states}} that there is always a jump of the first derivative of the least action ground state around the support of the delta function and thus the solution is only in $ H^1 $. 
		
		We can observe from the right column in \cref{fig:delta_ground_states} that $ S_g(\omega) $ is always increasing with $ \omega $, which corresponds with the theoretical result shown in \cref{sec:asymtotics}. Moreover, in the $ L^2 $-subcritical and -critical case (Case I and II), the ground state mass $ M_g(\omega) $ is monotonically increasing, which indicates that the standing wave solutions are stable by the criterion \ref{cri1}-\ref{cri2} (see \cite{fukuizumi2008}). Besides, in the $ L^2 $-critical case (Case II), there is a threshold value of mass when $ \omega \rightarrow \infty $ given by $ M(\phi_1^0) $ which is a result of \cref{mass_thereshold}. 
		
		We want to mention that since we choose three delta functions located at $ x=0, \pm 1 $ in Case III, some different phenomena can be observed in the numerical experiments. That is if the initial data $ \phi_0 $ is shifted away from the origin, the GFDN-BF may give some asymmetric solutions different from what is shown in \cref{fig:delta_ground_states}. When checking the value of $ S_\omega $, one may find that such asymmetric solutions have larger action than the symmetric ones, which implies that they are not least action ground states. However, as shown in \cite{sym_breaking}, the symmetry breaking phenomenon can indeed happen and some examples are shown in \cref{exmp:smooth}. 
		
		\begin{figure}[htbp]
			\centering
			\subfloat{\includegraphics[width=0.475\textwidth]{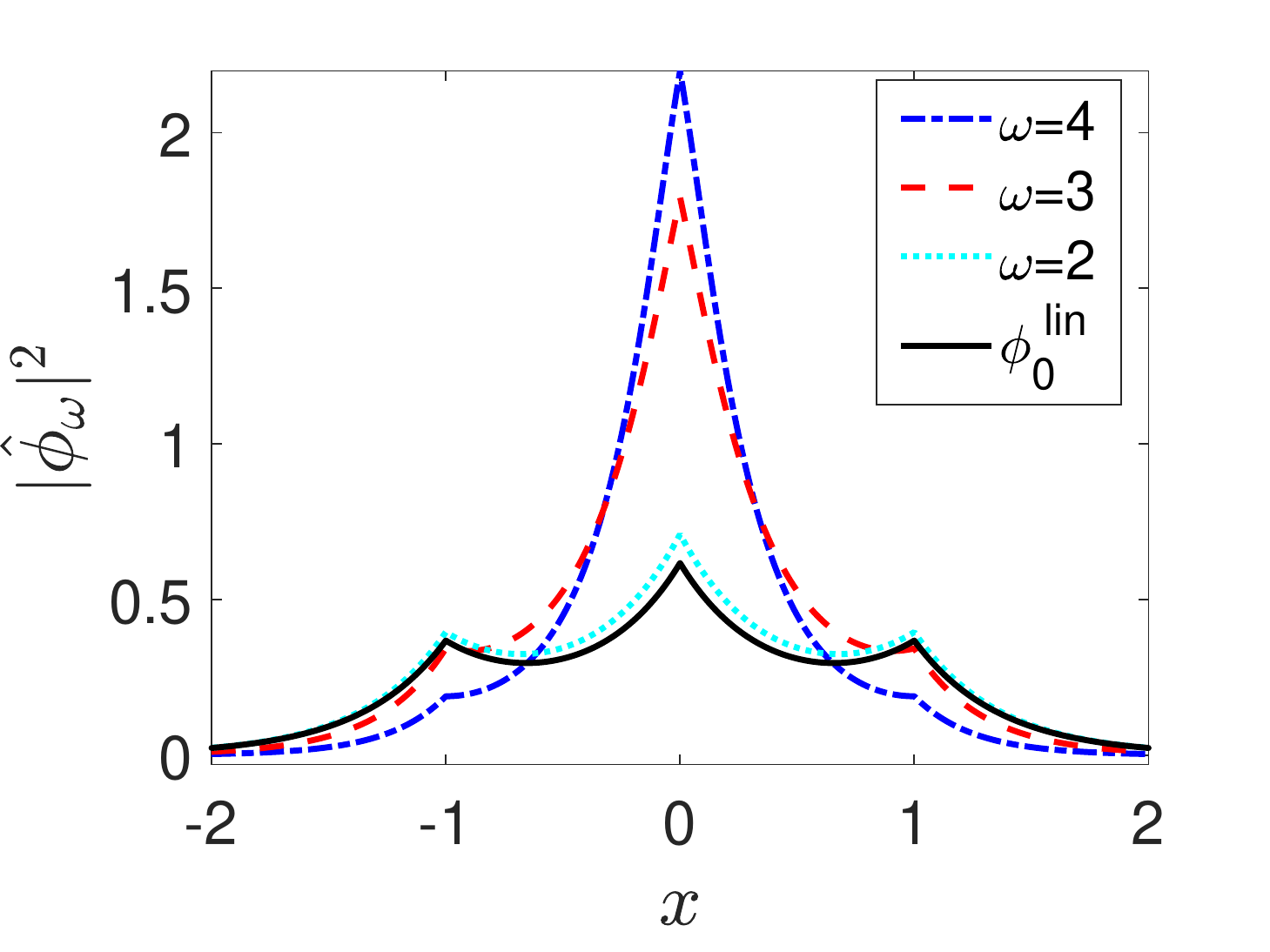}}
			\subfloat{\includegraphics[width=0.475\textwidth]{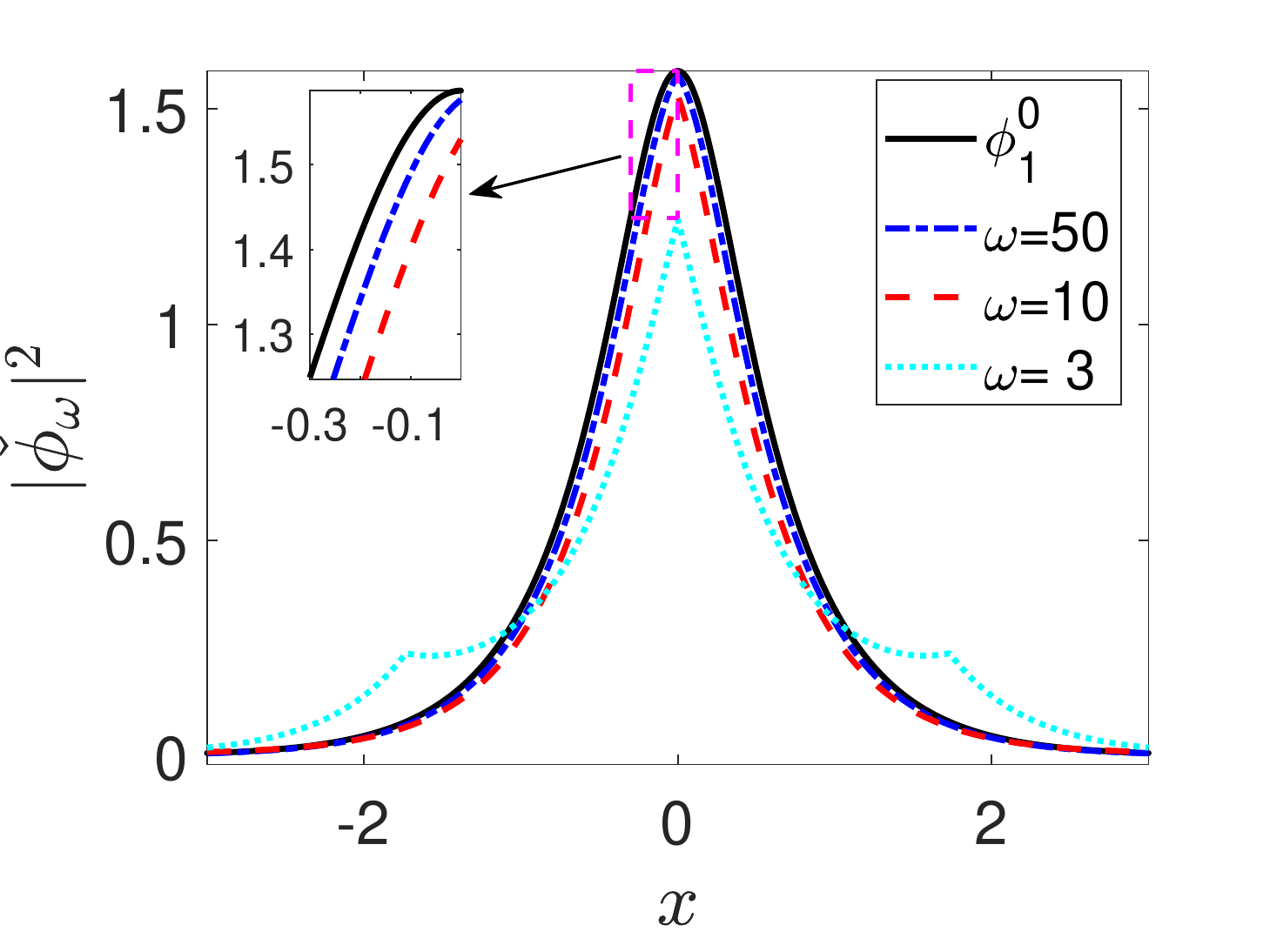}}
			\caption{density of the normalized least action ground states $ \widehat \phi_\omega $ (left) and $ \widecheck \phi_\omega $ (right) for Case III in \cref{exmp:delta}}
			\label{fig:delta_ground_states_normalized}
		\end{figure}
		
		\begin{figure}[htbp]
			\centering
			\subfloat{\includegraphics[width=0.475\textwidth]{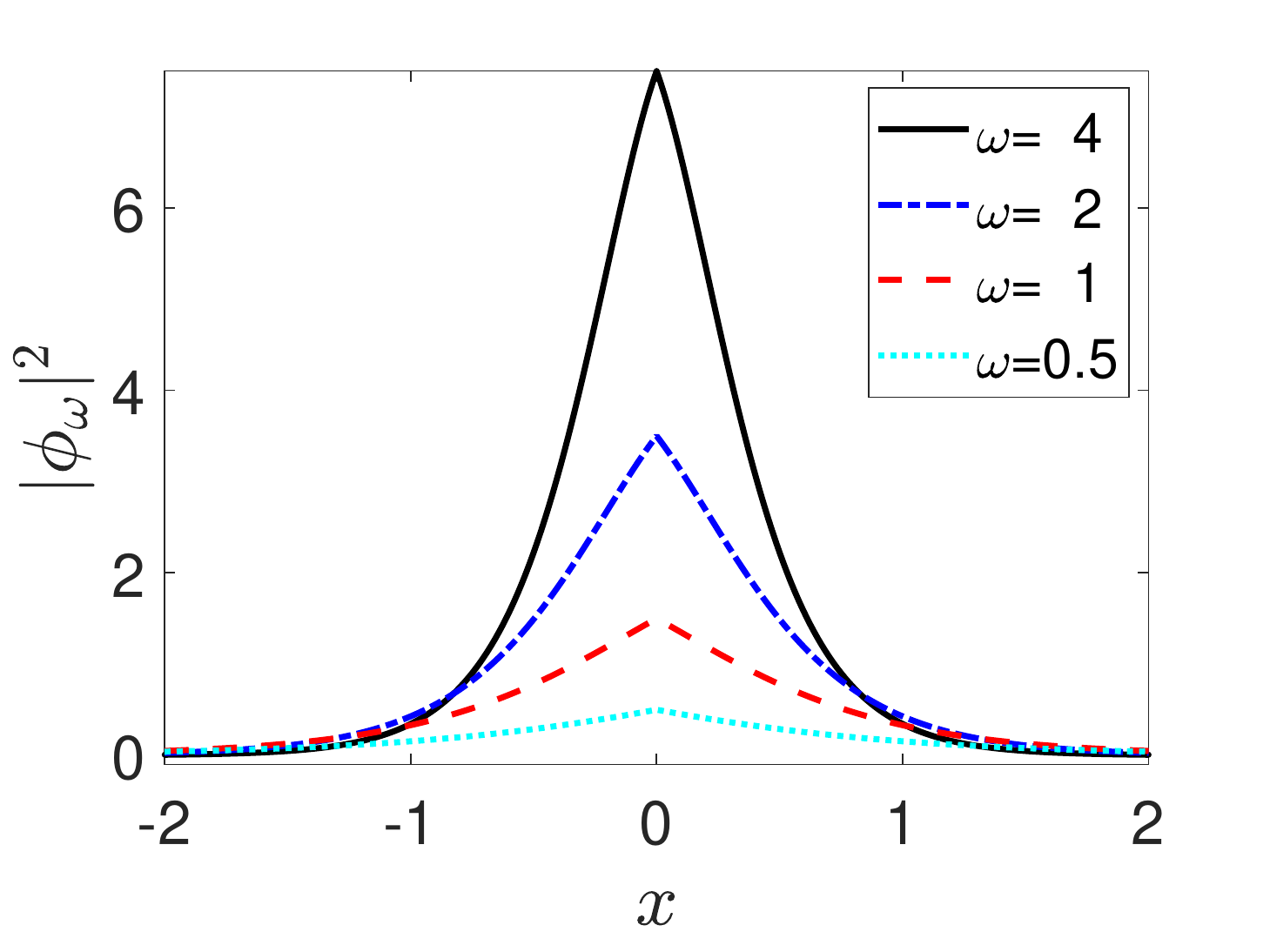}}
			\subfloat{\includegraphics[width=0.475\textwidth]{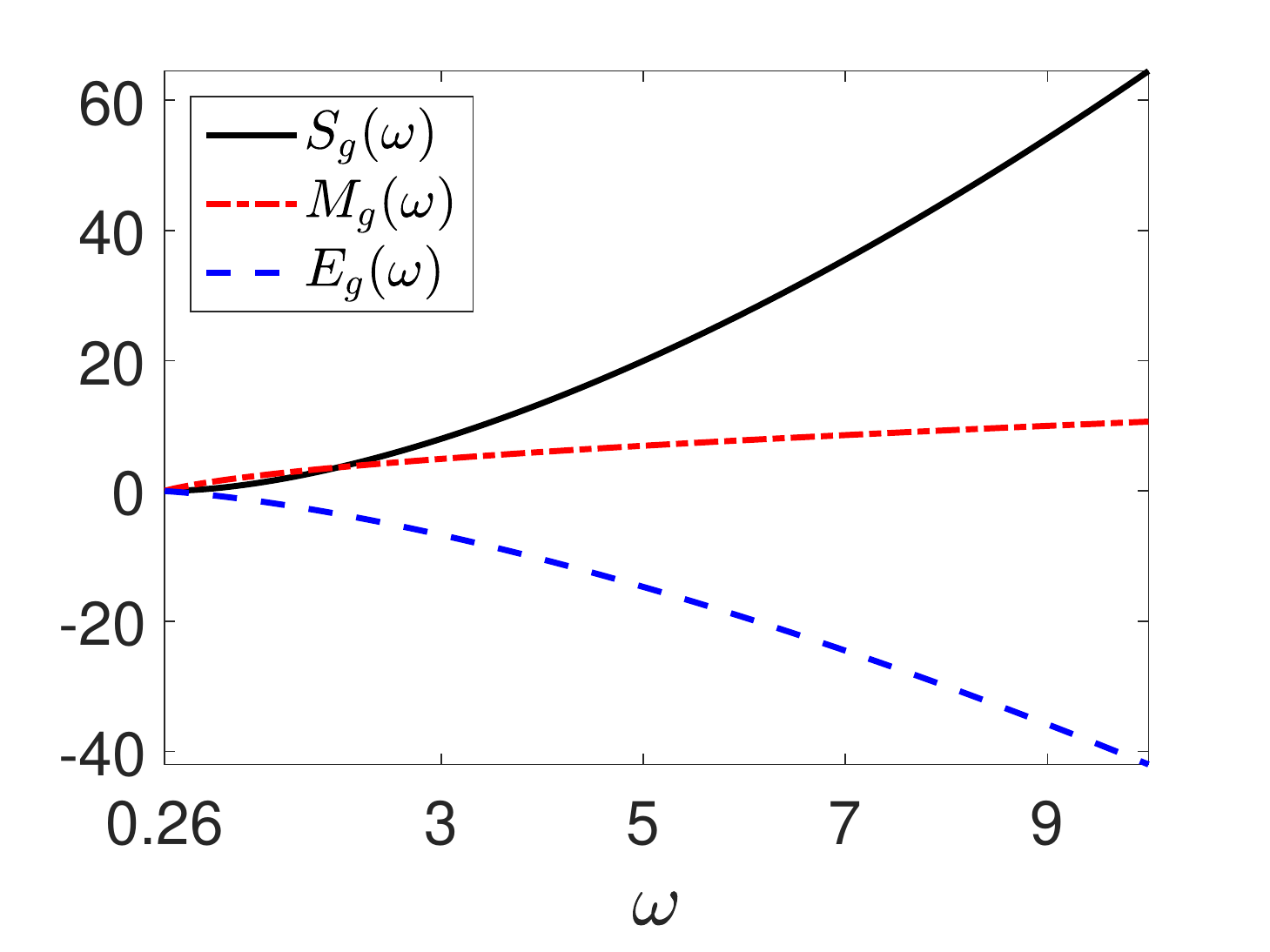}}\\
			\subfloat{\includegraphics[width=0.475\textwidth]{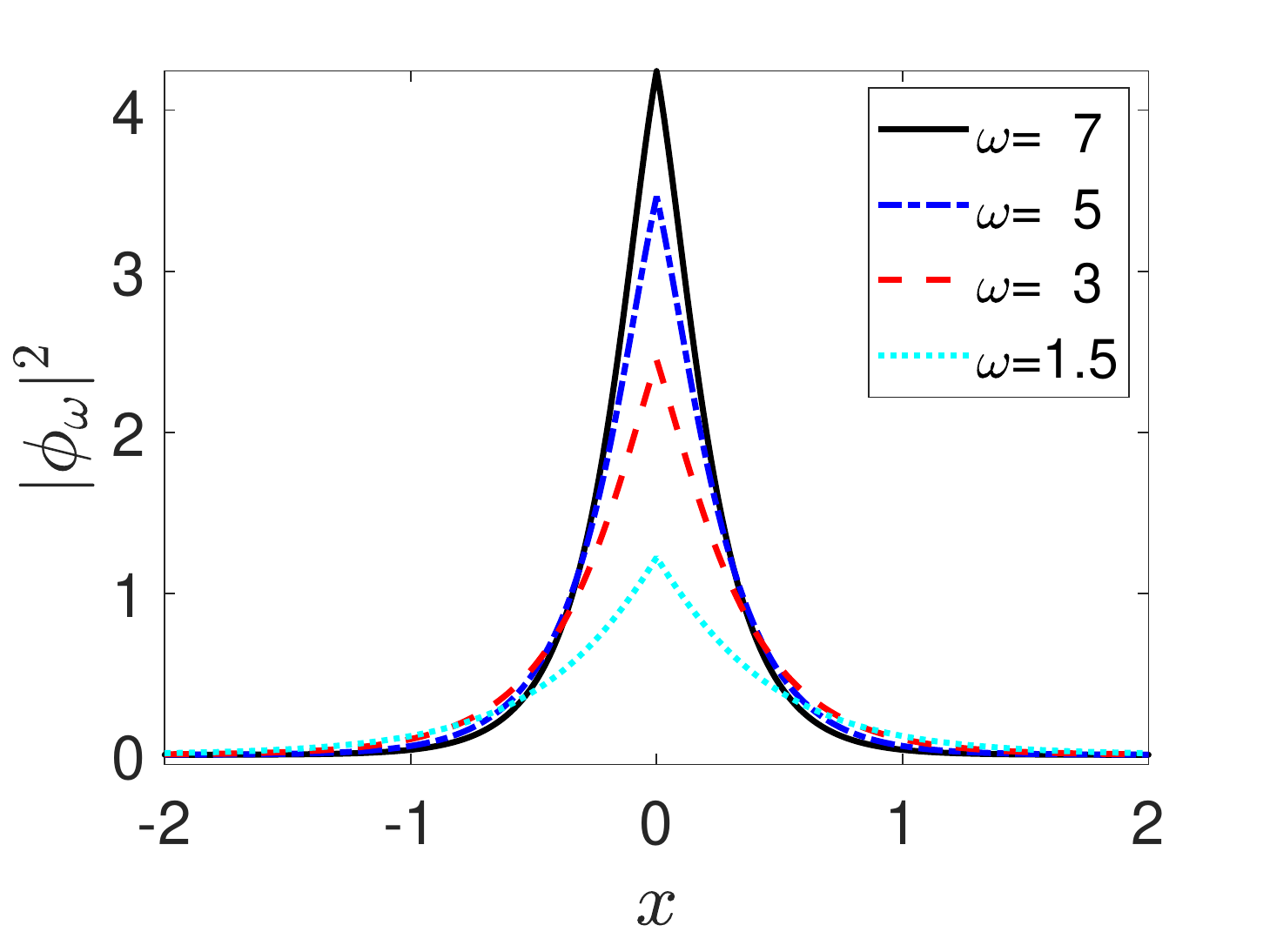}}
			\subfloat{\includegraphics[width=0.475\textwidth]{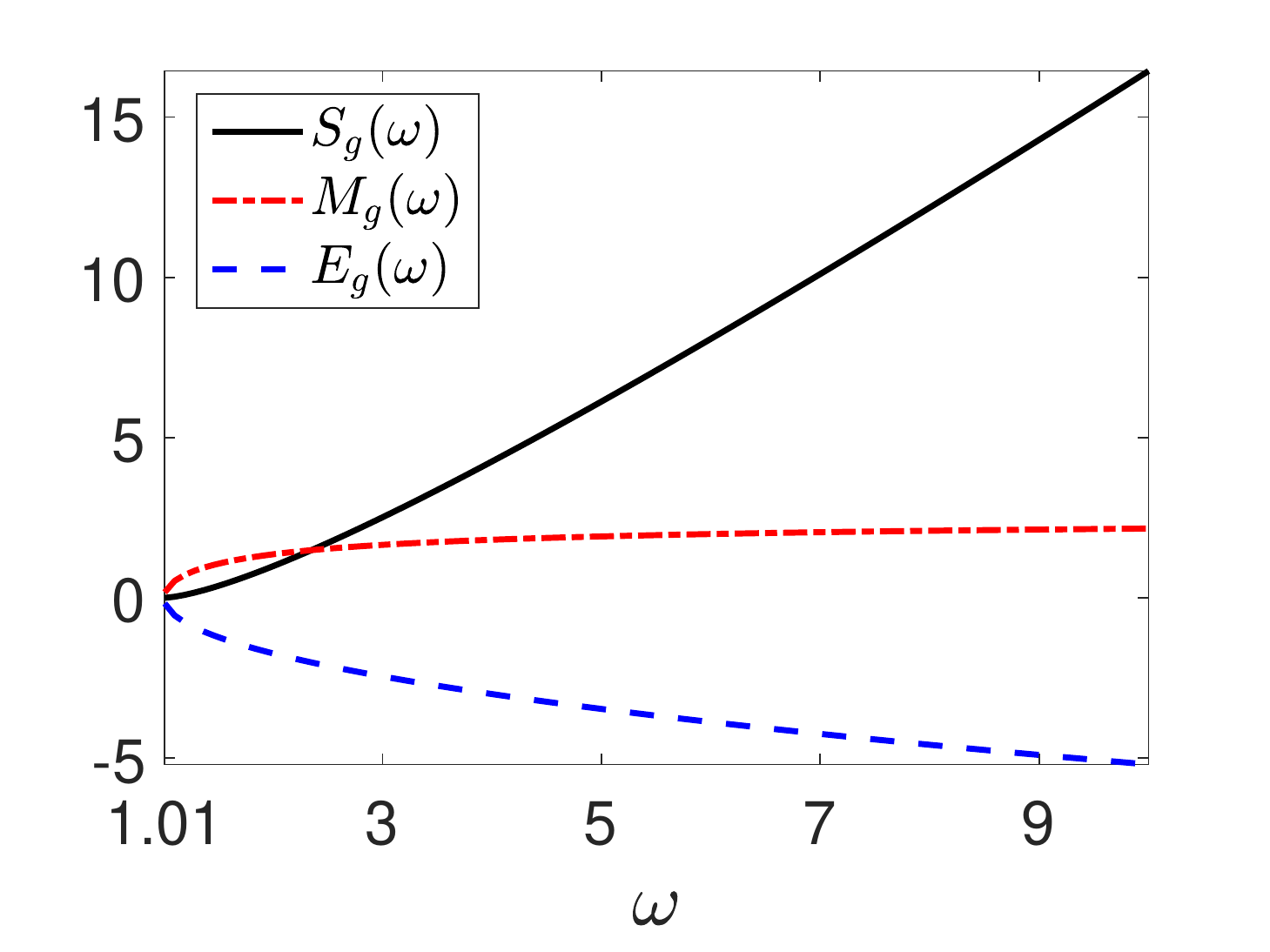}}\\
			\subfloat{\includegraphics[width=0.475\textwidth]{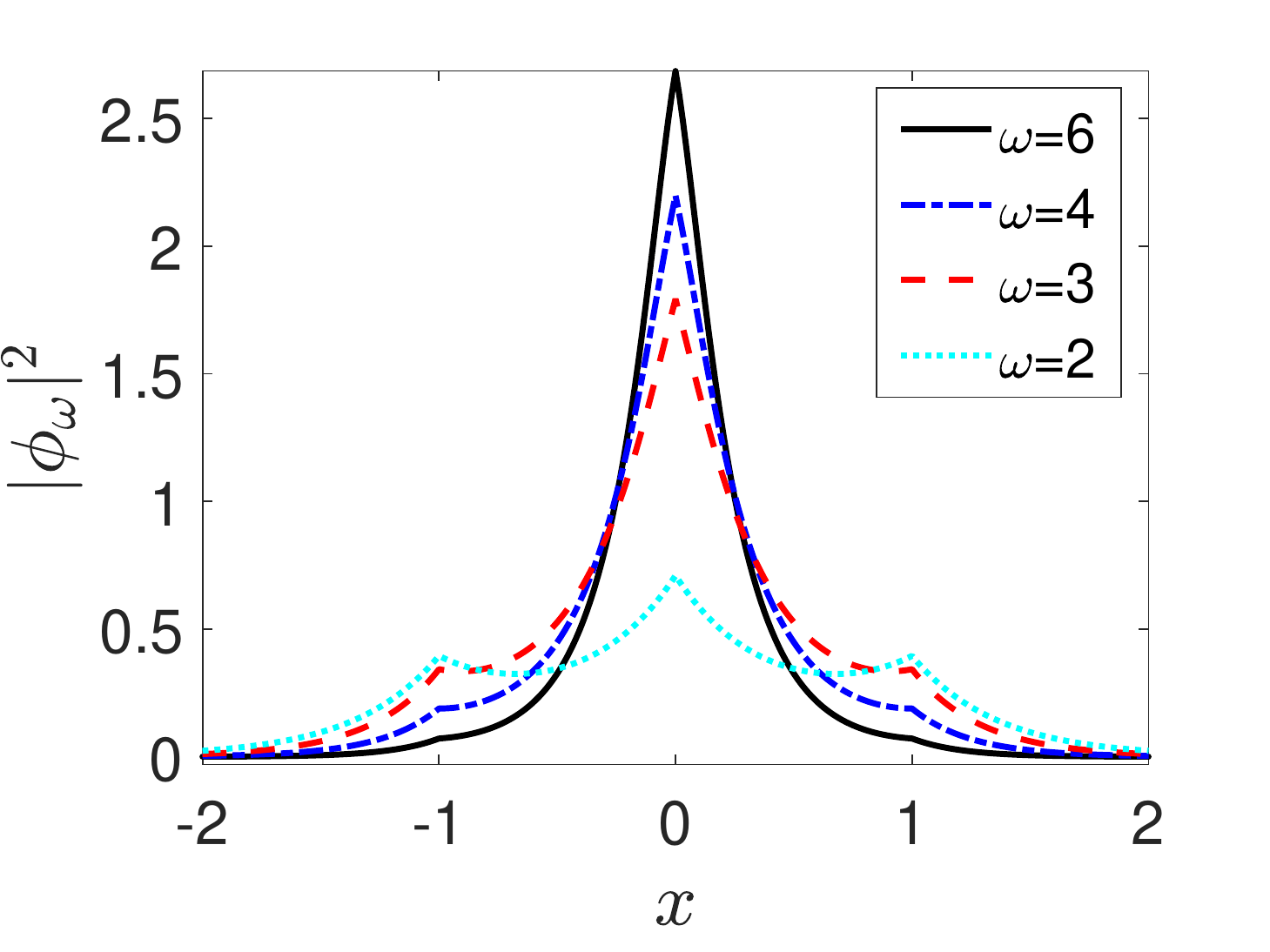}}
			\subfloat{\includegraphics[width=0.475\textwidth]{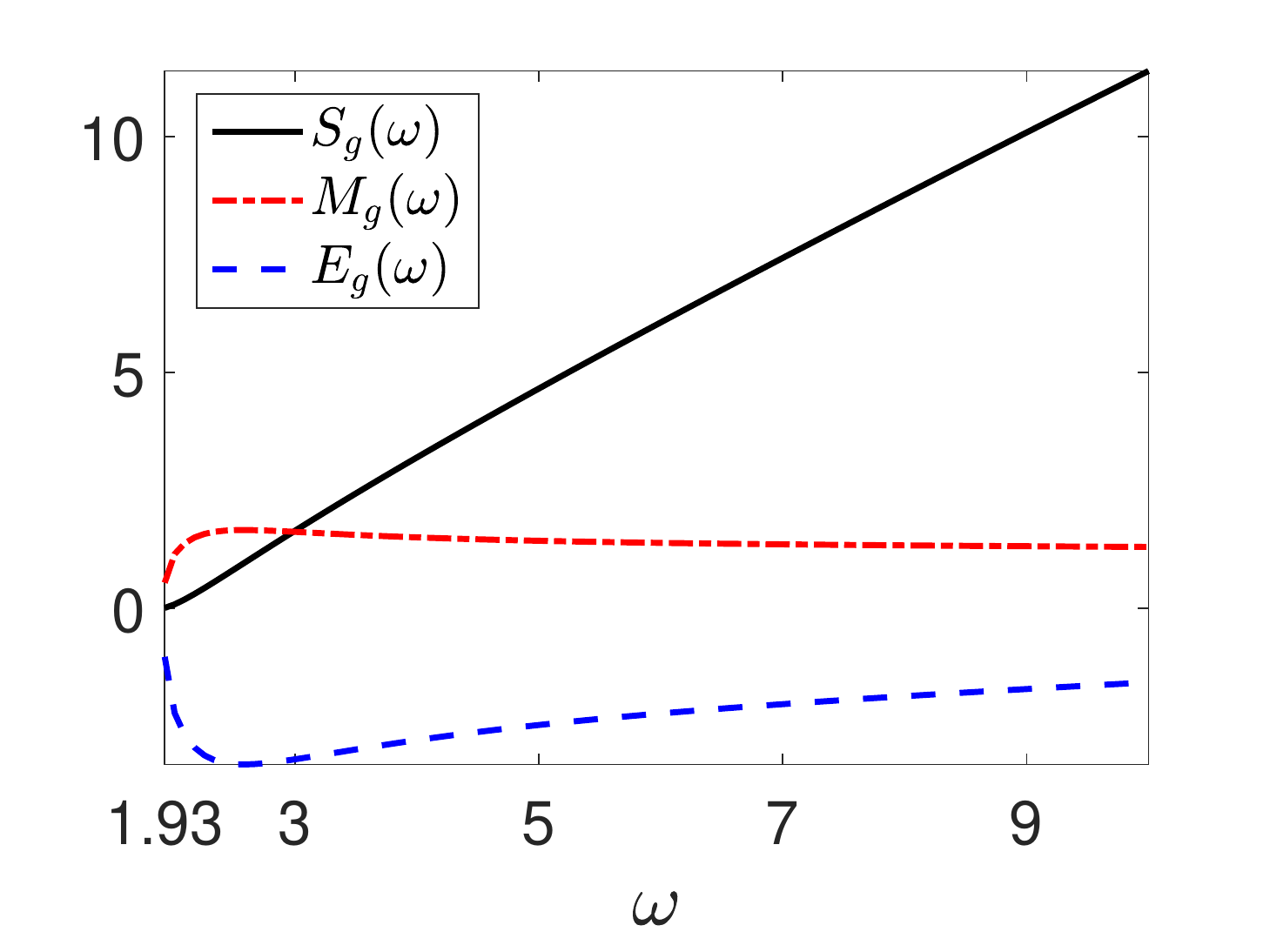}}
			\caption{density of the least action ground states (left) and change of $ S_g $, $ M_g $ and $ E_g $ (right) for Case I (top), Case II (middle) and Case III (bottom) in \cref{exmp:delta}}
			\label{fig:delta_ground_states}
		\end{figure}
	\end{exmp}
	
	\begin{exmp}\label{exmp:inverse_power}
		Least action ground states with attractive inverse power potential
		\begin{equation}\label{inver_power_potential_in_example}
			V(\vx) = -\frac{\gamma}{|\vx|^\sigma}, \quad \vx \in \R^d, \quad \gamma>0, \quad 0<\sigma < \min\{2, d\}. 
		\end{equation}
		We consider the following three cases in 1D, 2D and 3D: 
		\begin{description}
			\item [Case I.] $ d=1 $, $ \sigma = \frac{1}{2} $, $ \gamma = 1 $, $ \alpha = 1 $, $ \omega_0 \approx 1.6535 $; 
			\item [Case II. ] $ d=2 $, $ \sigma = 1 $, $ \gamma = 1 $, $ \alpha = 1 $, $ \omega_0 \approx 1.0000 $ ; 
			\item [Case III. ] $ d=3 $, $ \sigma = 1.5 $, $ \gamma = 1$, $ \alpha = 1 $, $ \omega_0 \approx 0.2986 $. 
		\end{description}
		
		Again, we shall show the least action ground state and change of $ S_g(\omega) $, $ M_g(\omega) $ and $ E_g(\omega) $ for Case I to III. Due to the radial symmetry of the potential in 2D and spherical symmetry in 3D, all the three cases can be reduced to a 1D problem. In computation, we choose $ \Omega= \{\vx \in \R^d : |\vx| < 16 \} $, $ \tau = 1 $, $ \vep = 10^{-9} $, $ h=2^{-8} $ and use linear finite element method with a uniform mesh for spatial discretization for all the cases. The reason we fix $ \alpha=1 $ is that $ \alpha=1 $ is $ L^2 $-subcritical when $ d=1 $, $ L^2 $-critical when $ d=2 $ and $ L^2 $-supercritical when $ d=3 $. The numerical results are shown in \cref{fig:inverse_power_ground_states}.
		
		From the left column in \cref{fig:inverse_power_ground_states}, we can observe very different behavior of the least action ground states around the origin for three cases. In Case I with $ \alpha = 1/2 $ (top), it seems the first derivative $ \phi_\omega'(x) \rightarrow 0 $ as $ x \rightarrow 0 $. In Case II with $ \alpha = 1 $, the first derivative of the least action ground state seems to tend to some nonzero constant. While in Case III with $ \alpha = 3/2 $, the first derivative seems to be unbounded around the origin. Actually, it is proved in \mbox{\cite{fukaya2021}} that when $ 0 < \alpha < 1 $, the least action ground state is in $ C^1 $ while when $ 1 \leq \alpha < 2 $, $ \phi_\omega'(|\vx|) \sim |\vx|^{1-\alpha} $ as $ \vx \rightarrow 0 $, which correspond with our numerical results. 
		
		For the change of $ S_g(\omega) $, $ M_g(\omega) $ and $ E_g(\omega) $ in the right column of \cref{fig:inverse_power_ground_states}, similar phenomena can be observed as in the previous example: $ S_g(\omega) $ is monotonically increasing with $ \omega $ and $ M_g(\omega) $ tends to a threshold value as $ \omega \rightarrow \infty $ when $ \alpha $ is $ L^2 $-critical (Case II) as well as the asymptotic results \cref{lim_omega_0} and \cref{lim_omega_infty} which is not shown here for brevity. Concerning $ M_g(\omega) $, we can observe that in the $ L^2 $-subcritical and $ L^2 $-critical case (Case I and II), the ground state mass $ M_g(\omega) $ is monotonically increasing while for the $ L^2 $-supercritical case (Case III), $ M_g(\omega) $ is increasing when $ \omega<\omega_c $ and decreasing when $ \omega>\omega_c $ for some $ \omega_c > \omega_0 $, which indicates \cref{conjecture} about the stability and instability of the standing wave solutions by the result in \cite{fukaya2021} and the criterion \ref{cri1}-\ref{cri2} (see more detailed discussion in \cref{sec:stability}).  
		\begin{figure}[htbp]
			\centering
			\subfloat{\includegraphics[width=0.475\textwidth]{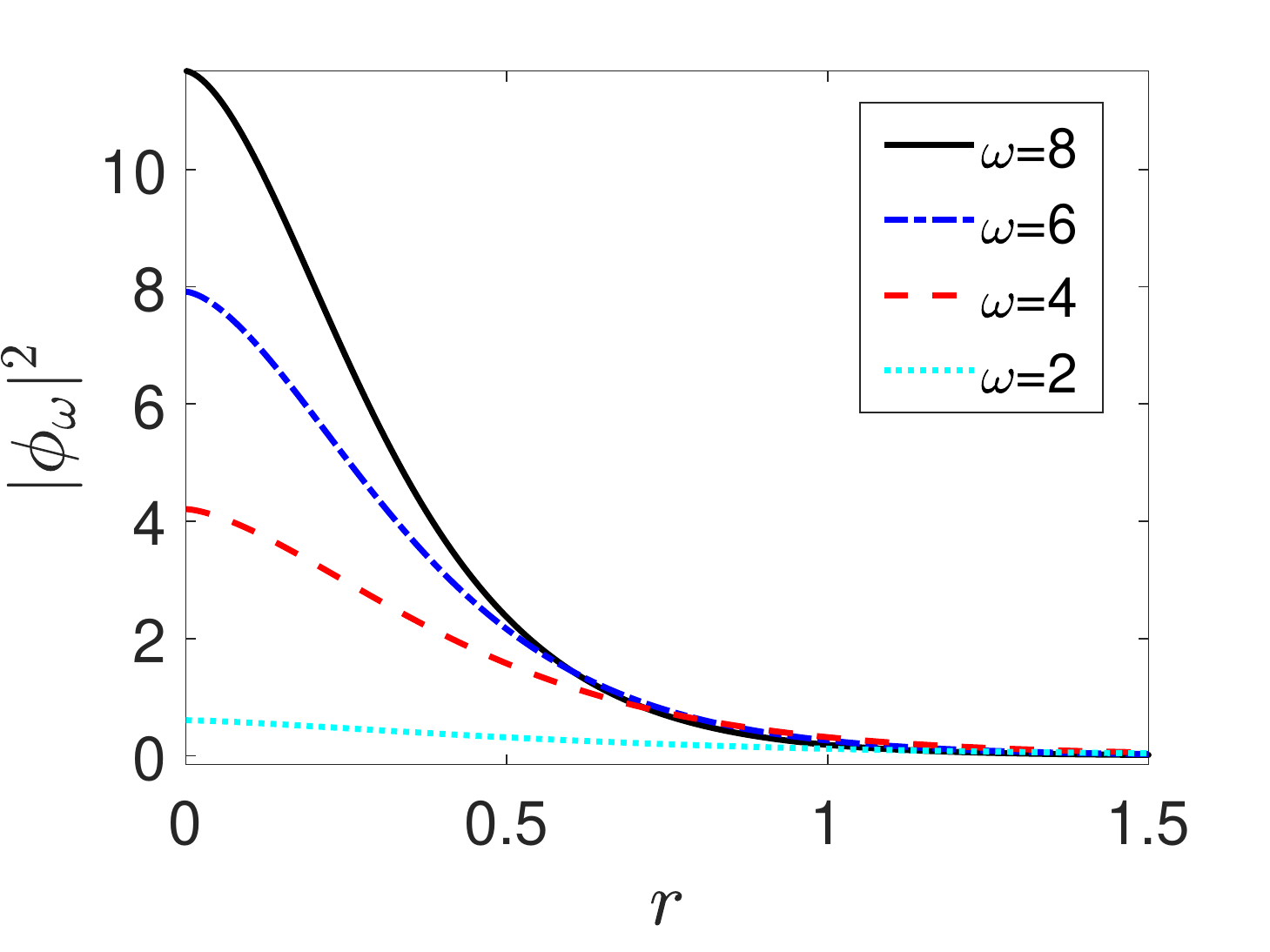}}
			\subfloat{\includegraphics[width=0.475\textwidth]{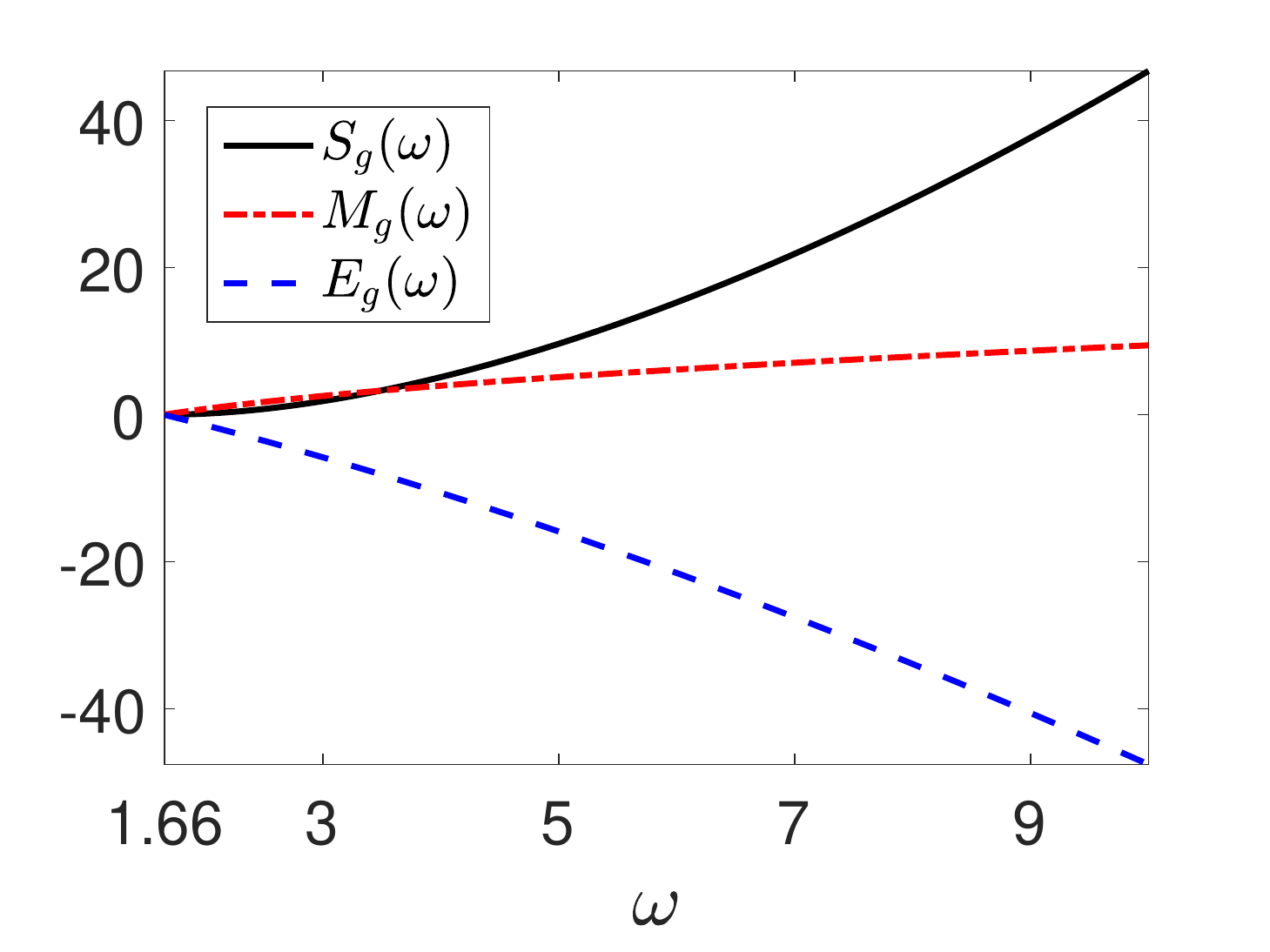}}\\
			\subfloat{\includegraphics[width=0.475\textwidth]{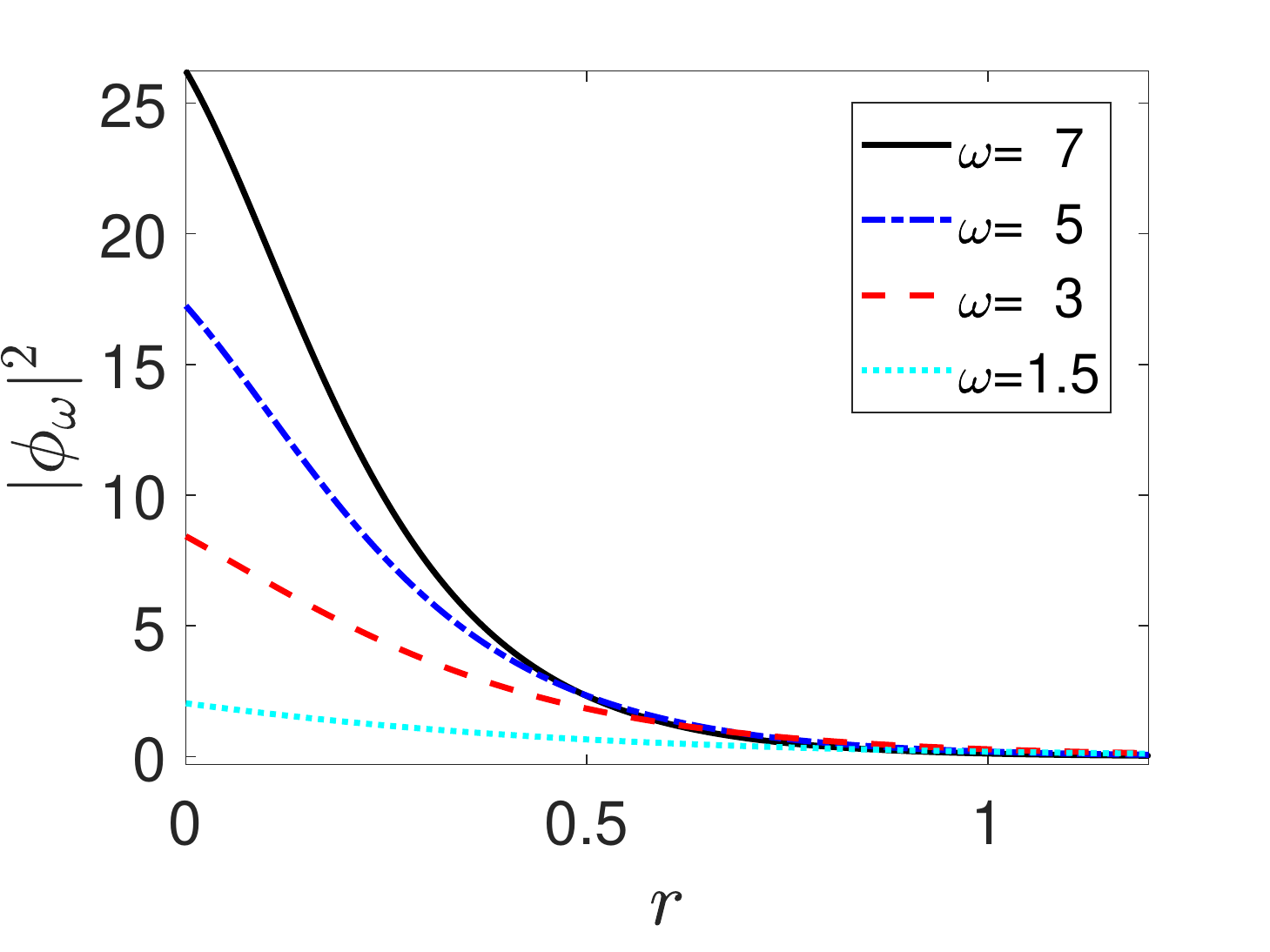}}
			\subfloat{\includegraphics[width=0.475\textwidth]{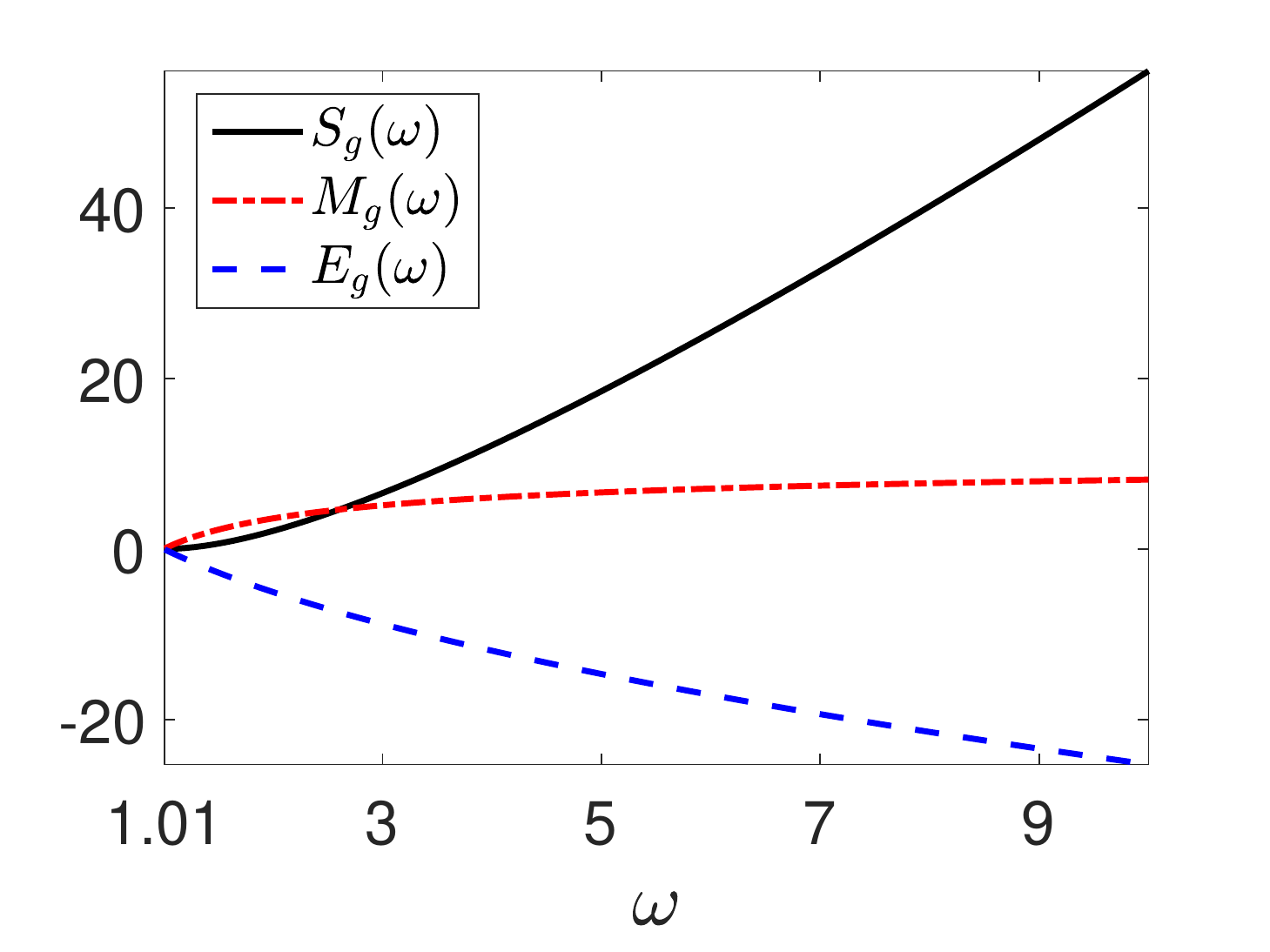}}\\
			\subfloat{\includegraphics[width=0.475\textwidth]{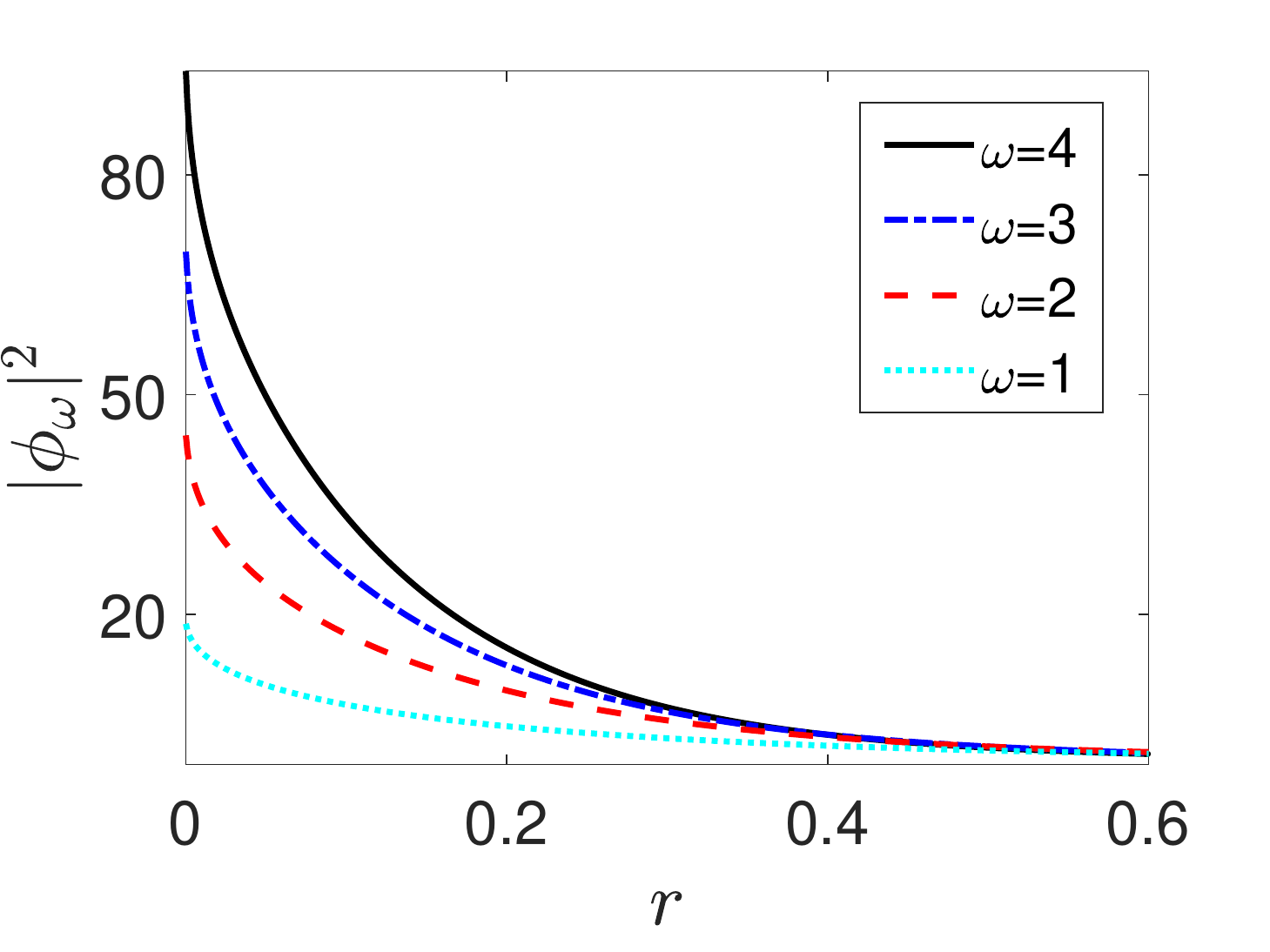}}
			\subfloat{\includegraphics[width=0.475\textwidth]{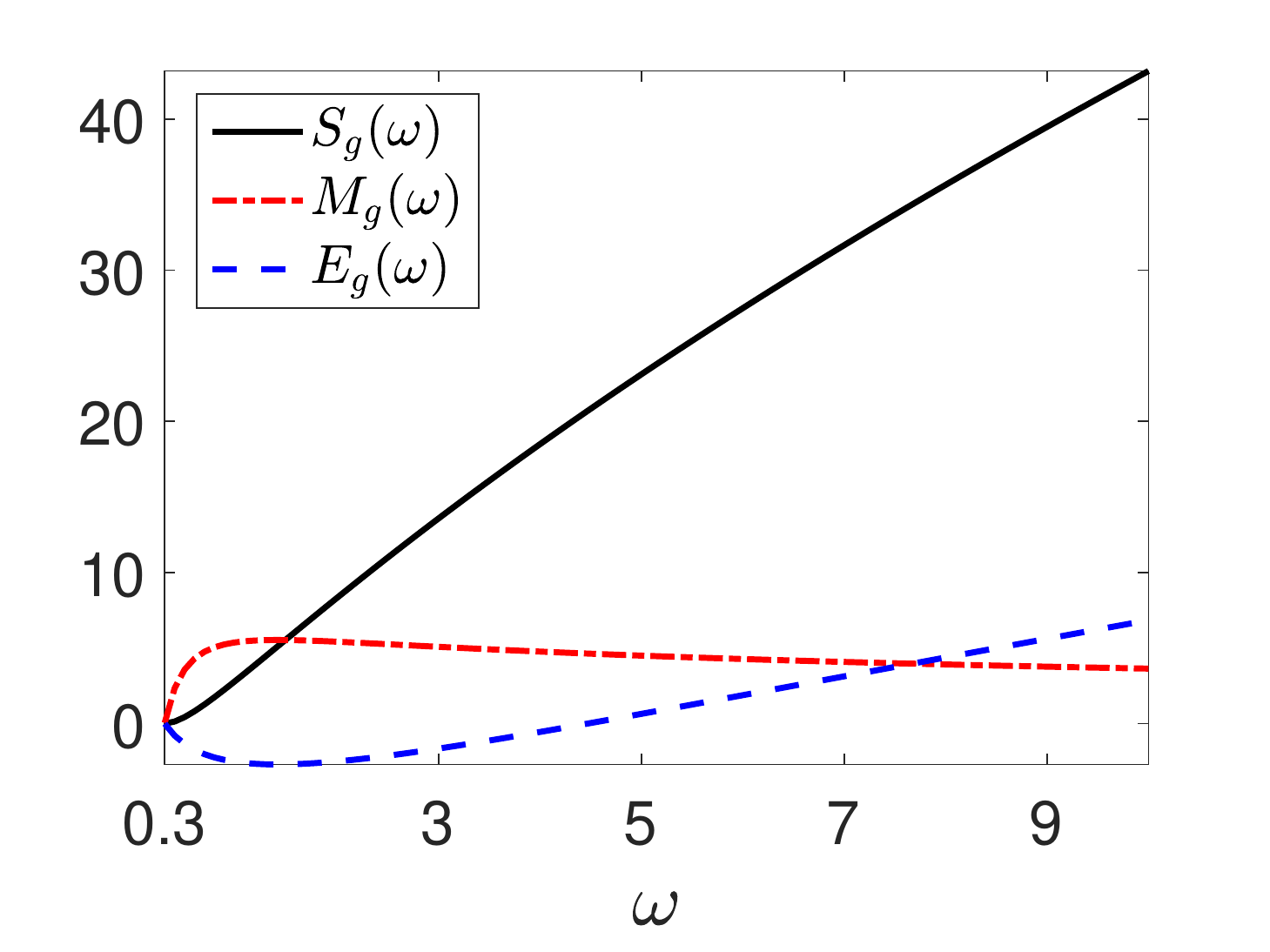}}
			\caption{density of the least action ground states in radial variable (left) and change of $ S_g $, $ M_g $ and $ E_g $ (right) for Case I (top), Case II (middle) and Case III (bottom) in \cref{exmp:inverse_power}}
			\label{fig:inverse_power_ground_states}
		\end{figure}
	\end{exmp}
	
	\begin{exmp}\label{exmp:well}
		Least action ground states with finite well potential
		\begin{equation}
			V(\vx) = \left\{
			\begin{aligned}
				&-Z, &&\vx \in U \subset \R^d, \\
				&0,  &&\vx \notin U. 
			\end{aligned}
			\right. 
		\end{equation}
		We consider the following three cases in 1D, 2D and 3D:  
		\begin{description}
			\item [Case I.] $ d=1 $, $ U = (-2, 2) $, $ Z = 2 $, $ \alpha = 1 $, $ \omega_0 \approx 1.6685 $; 
			\item [Case II. ] $ d=2 $, $ U = \{\vx \in \R^2:|\vx| < 2\} $, $ Z = 2 $, $ \alpha = 1 $, $ \omega_0 \approx 1.2433 $; 
			\item [Case III. ] $ d=3 $, $ U = \{\vx \in \R^3:|\vx| < 2\} $, $ Z = 2 $, $ \alpha = 1 $, $ \omega_0 \approx 0.7544 $.
		\end{description}
		 
		 Similar to \cref{exmp:inverse_power}, the symmetry of the potential function allows us to reduce all the cases to a 1D problem. The least action ground states with different $ \omega $ and the change of $ S_g $, $ M_g $ and $ E_g $ are shown in \cref{fig:well_ground_states}. In computation, for all the cases, we choose $ \Omega= \{\vx \in \R^d : |\vx| < 16 \} $, $ \tau = 1 $, $ \vep = 10^{-9} $, $ h=2^{-8} $ and use linear finite element method for spatial discretization. 
		 
		 Since the singularity of the well potential is in some sense weaker than the delta potential or the inverse power potential, we cannot observe any singularity of the least action ground states around the jump of the potential function as shown in the left column of \mbox{\cref{fig:well_ground_states}}, which suggests that the solution is in $ H^2 $. Also, we would like to mention that we cannot find two different nonnegative least action ground states associated with the same $ \omega $ in our numerical experiments. 
		 
		 For the change of $ S_g(\omega) $, $ M_g(\omega) $ and $ E_g(\omega) $ in the right column of \cref{fig:well_ground_states}, similar observation can be made as in the previous two examples. One difference is that in Case II (which is the $ L^2 $-critical case), the least action ground state energy $ E_g $ also tends to a constant equal to $ -2 M(\phi_1^0) $ due to the special structure of the potential as well as \cref{lim_omega_infty}. However, we have to admit that even though we can observe similar behavior of $ M_g(\omega) $ in this example, it is not clear whether the criterion \ref{cri1}-\ref{cri2} can be used here to deduce the stability or instability of the standing wave solutions. 
		\begin{figure}[htbp]
			\centering
			\subfloat{\includegraphics[width=0.475\textwidth]{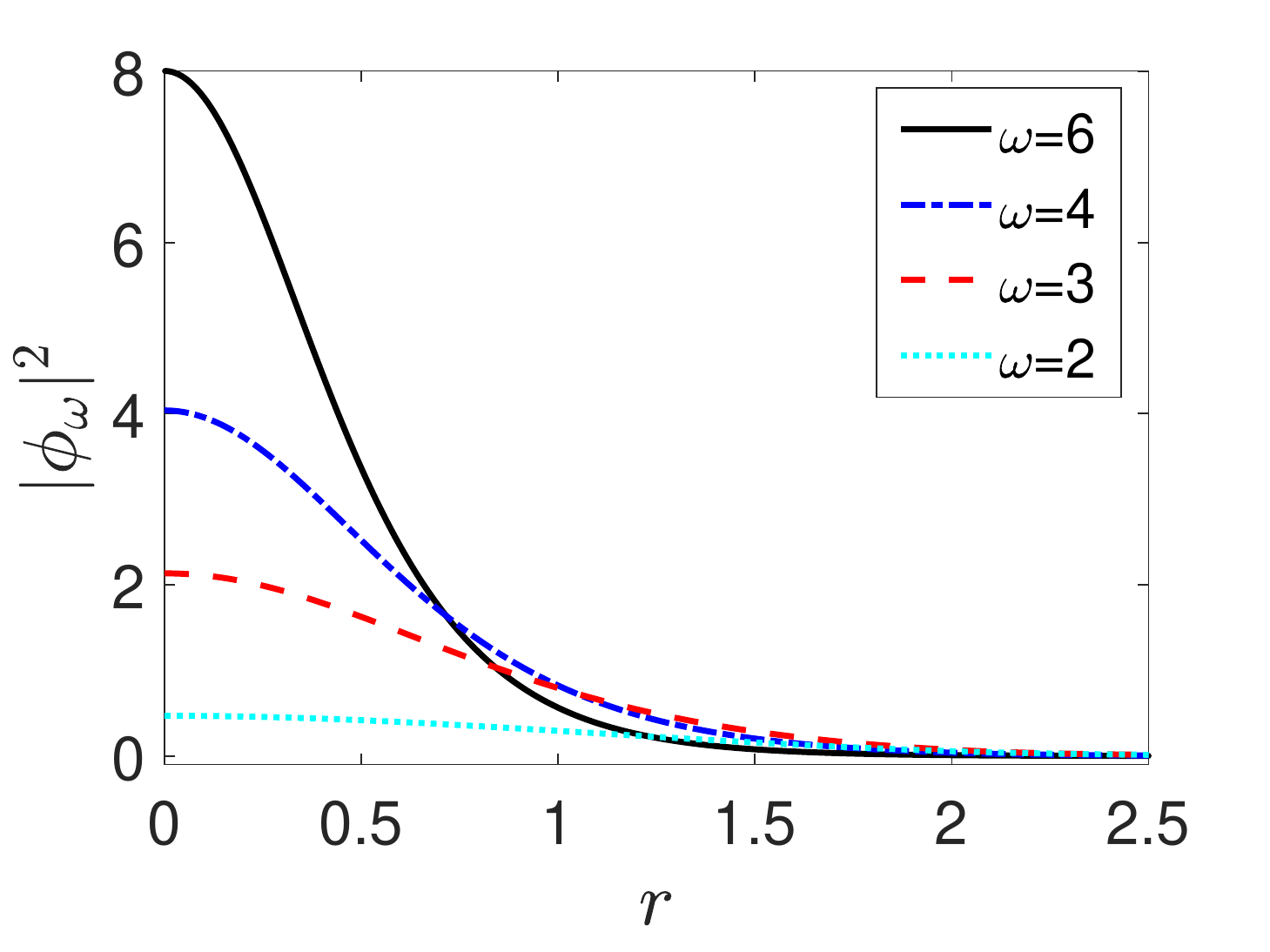}}
			\subfloat{\includegraphics[width=0.475\textwidth]{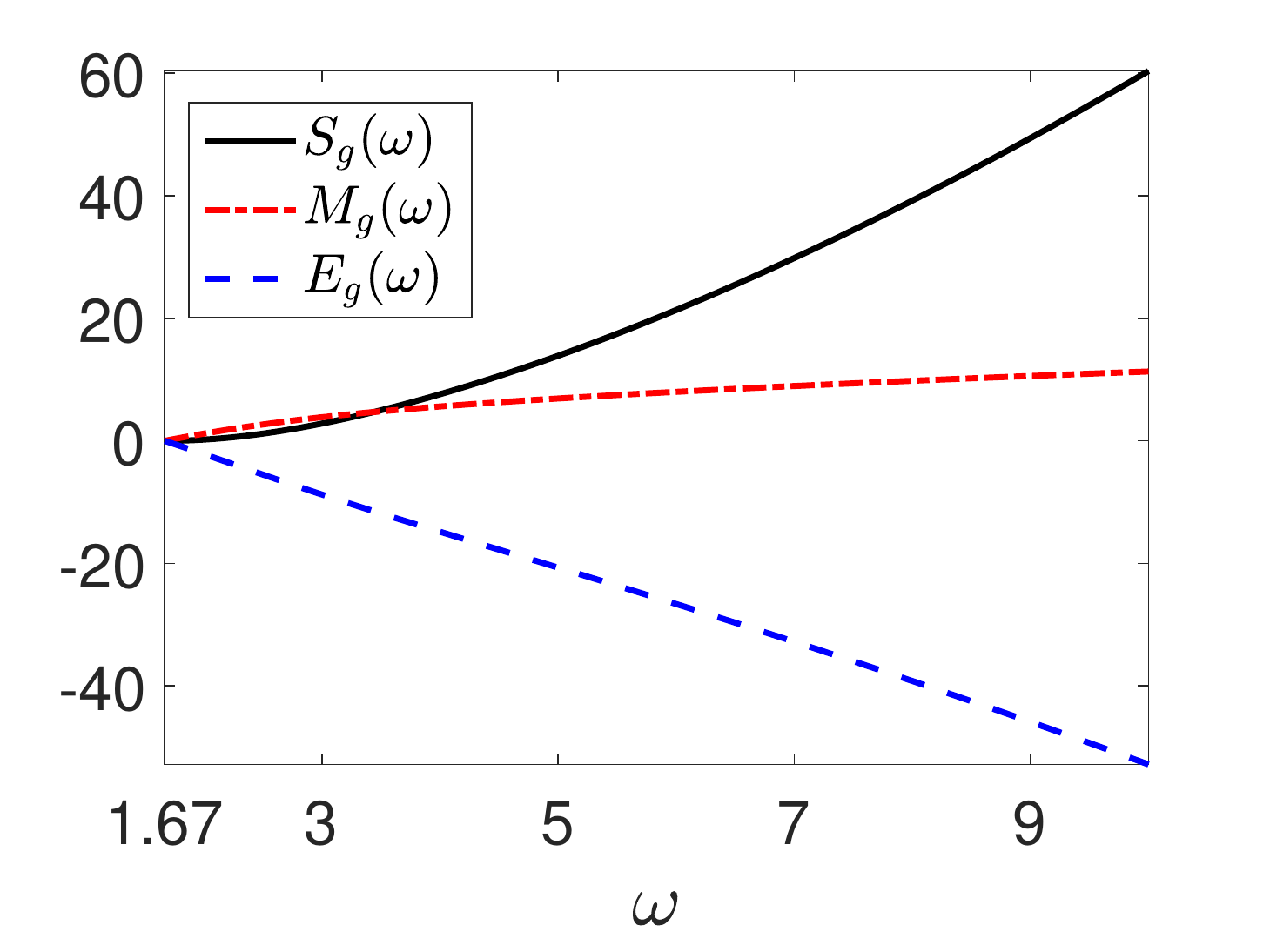}}\\
			\subfloat{\includegraphics[width=0.475\textwidth]{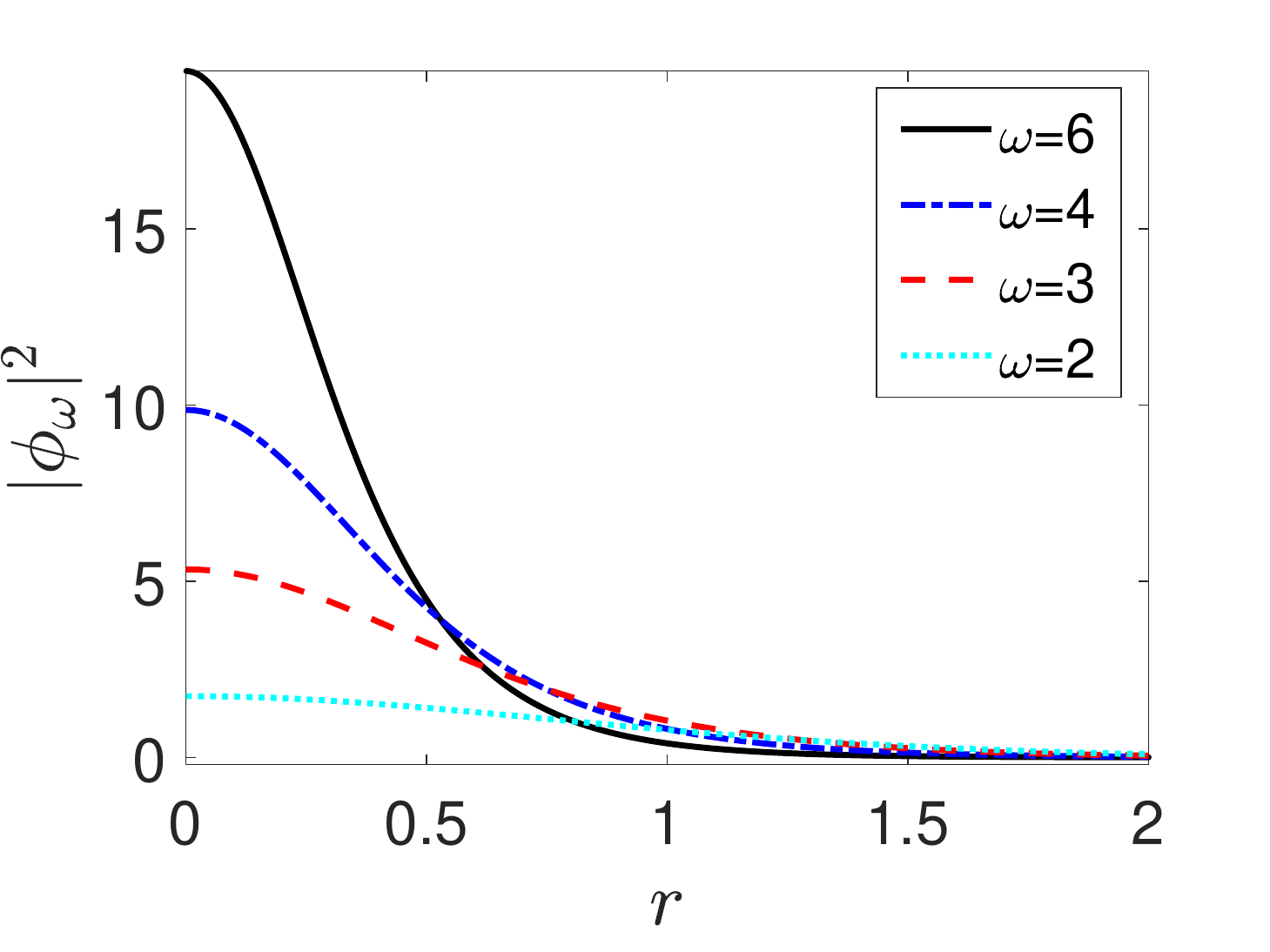}}
			\subfloat{\includegraphics[width=0.475\textwidth]{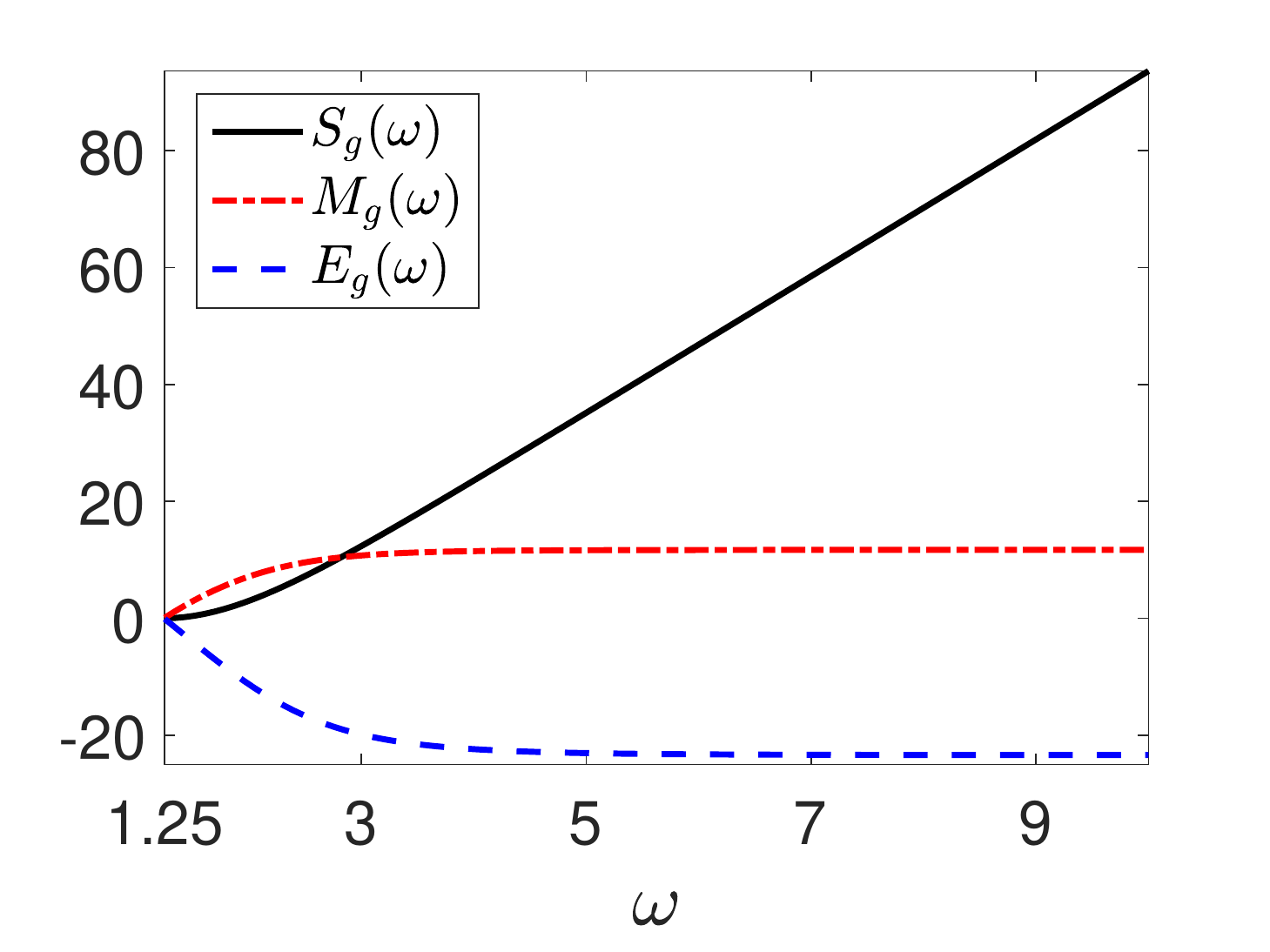}}\\
			\subfloat{\includegraphics[width=0.475\textwidth]{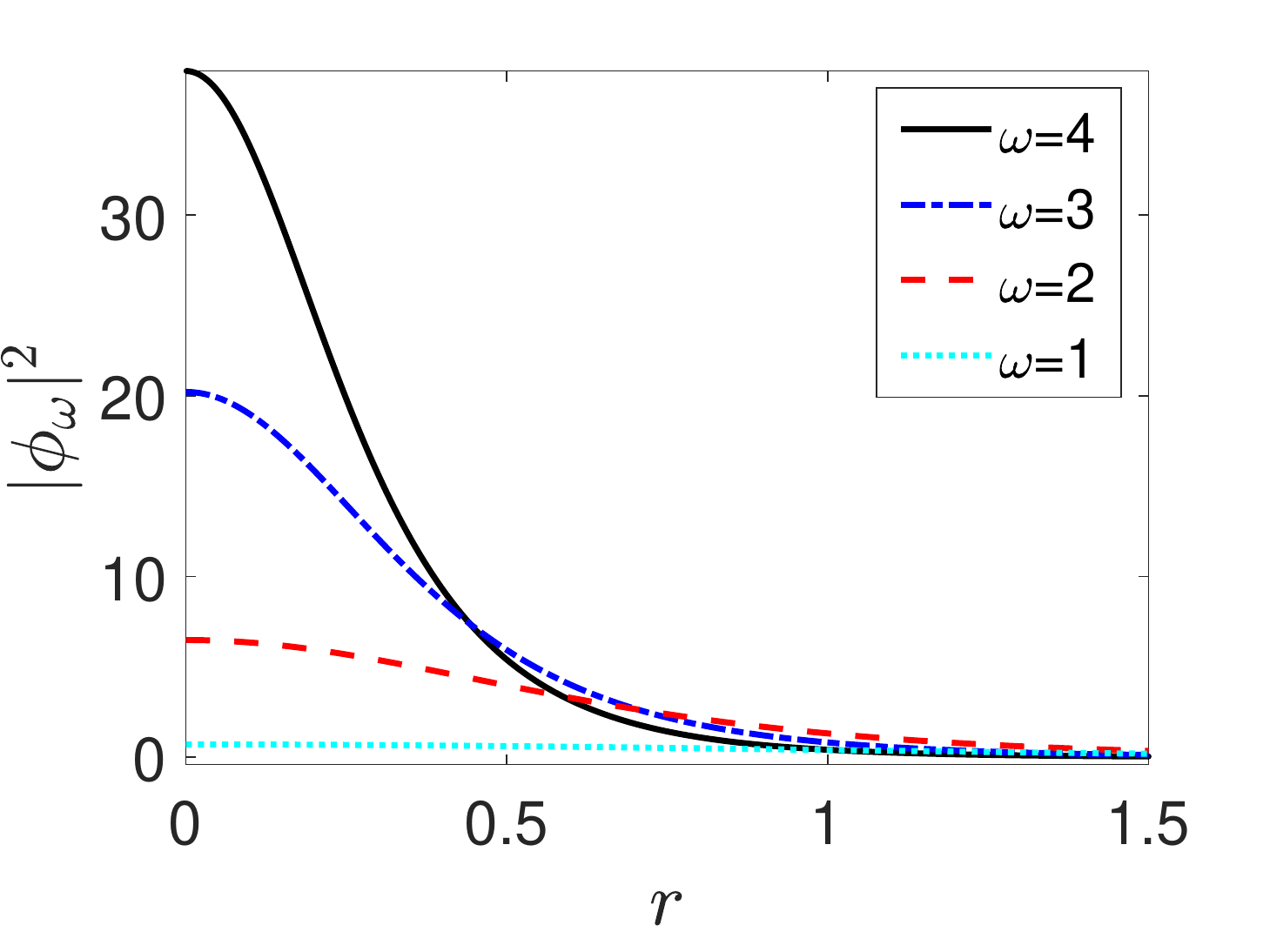}}
			\subfloat{\includegraphics[width=0.475\textwidth]{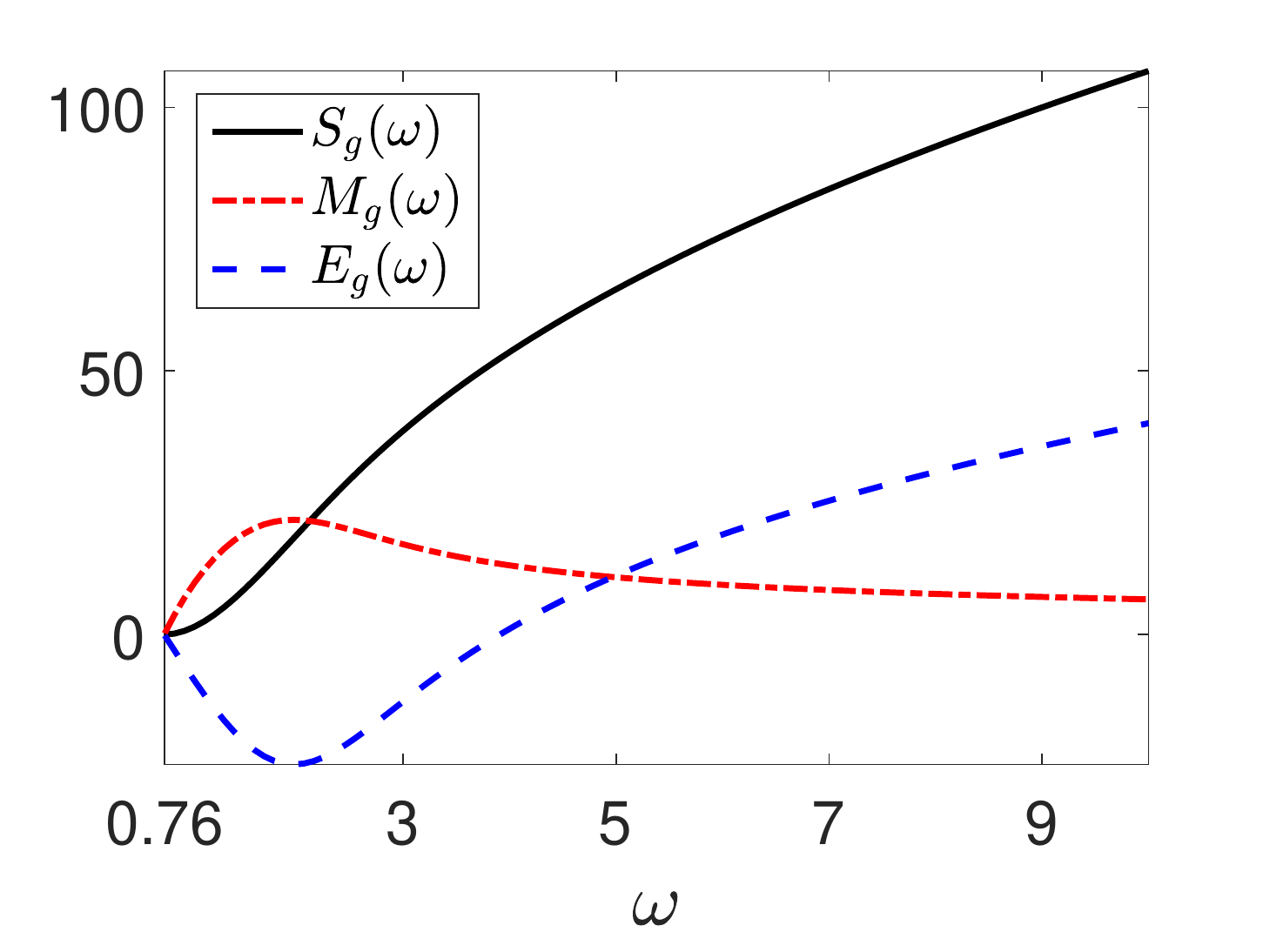}}
			\caption{density of the least action ground states in radial variable (left) and change of $ S_g $, $ M_g $ and $ E_g $ (right) for Case I (top), Case II (middle) and Case III (bottom) in \cref{exmp:well}}
			\label{fig:well_ground_states}
		\end{figure}
	\end{exmp}
	
	\begin{exmp}\label{exmp:smooth}
		In this example, we are devoted to showing some least action ground states with smooth potentials in 1D and 2D. We will use sine pseudospectral method or finite difference method for spatial discretization. The 1D example is solved with a double well potential
		\begin{equation}\label{smooth_1D}
			V(x) = -e^{-(x-2)^2} - e^{-(x+2)^2}, \quad x \in \R. 
		\end{equation}
		For this potential, $ \omega_0 \approx 0.4277 $. With a double well potential, the symmetry breaking bifurcation will happen when the ground state mass is larger than some critical value, and the corresponding least action ground state will concentrate at one well \cite{sym_breaking}. Due to this phenomenon, we have to choose initial data properly (shifted gaussian for example). In computation, we choose 
		\begin{equation*}
			V \text{ in }\cref{smooth_1D}, \quad \alpha = \frac{1}{2}, \quad \tau = 4, \quad \Omega = [-16, 16], \quad \vep = 10^{-14}, \quad h = \frac{1}{2^{10}}
		\end{equation*}
		and use finite difference method for spatial discretization. We show the least action ground states with different $ \omega $ and the change of ground state action, mass and energy with respect to $ \omega $ in \cref{fig:smooth_1D}. The symmetry breaking bifurcation occurs when $ \omega \approx 0.52 $. We also observe in the numerical experiment that near bifurcation, the GFDN-BF scheme converges much slower. The algorithm takes almost 10 times more steps to converge for $ \omega = 0.52 $ than $ \omega = 0.53 $ with initial data $ \phi_0 $ being a shifted Gaussian away from origin by 2 units. This phenomenon will be further investigated in our future work. 
		\begin{figure}[htbp]
			\centering
			\subfloat{\includegraphics[width=0.475\textwidth]{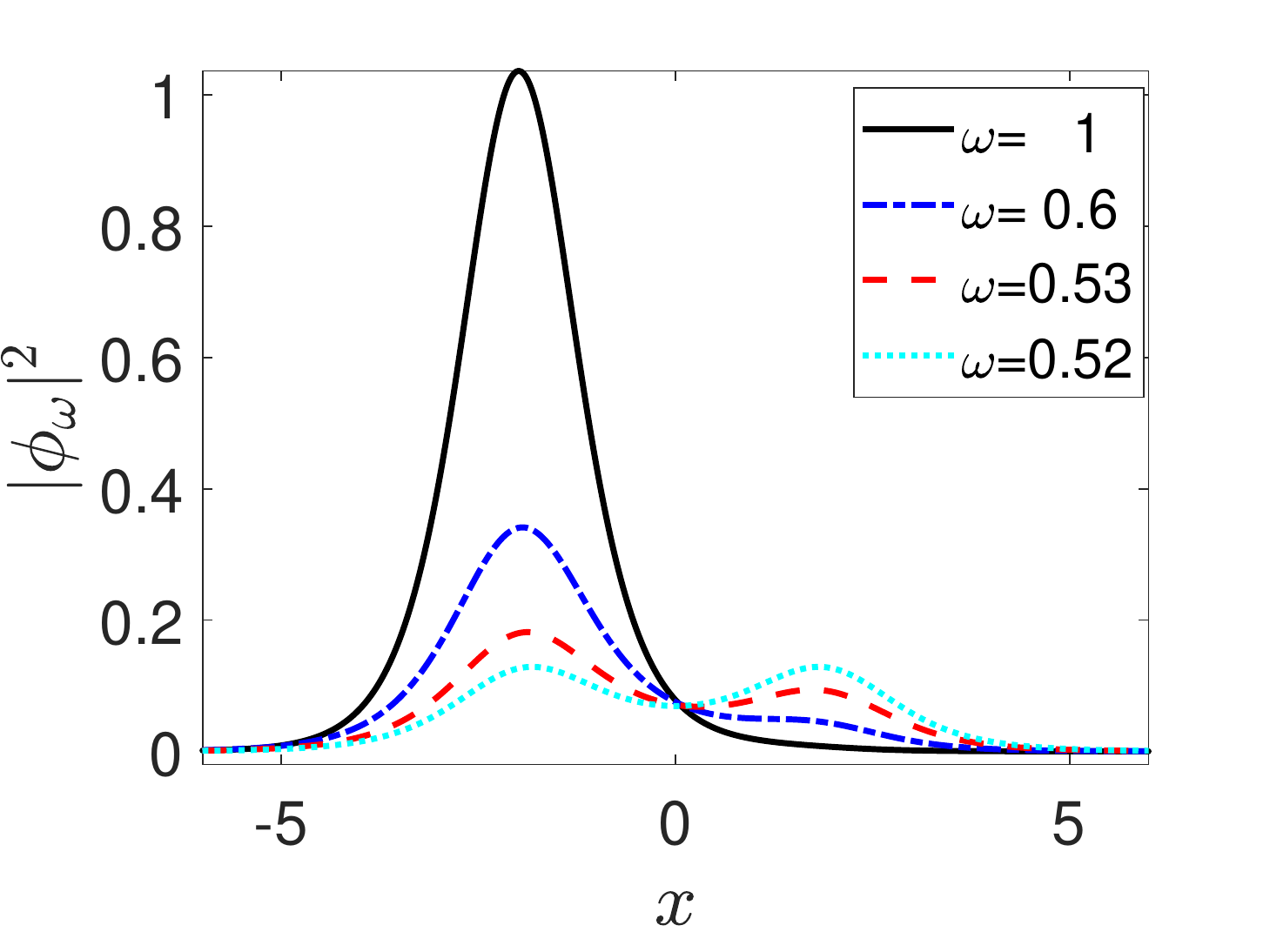}}
			\subfloat{\includegraphics[width=0.475\textwidth]{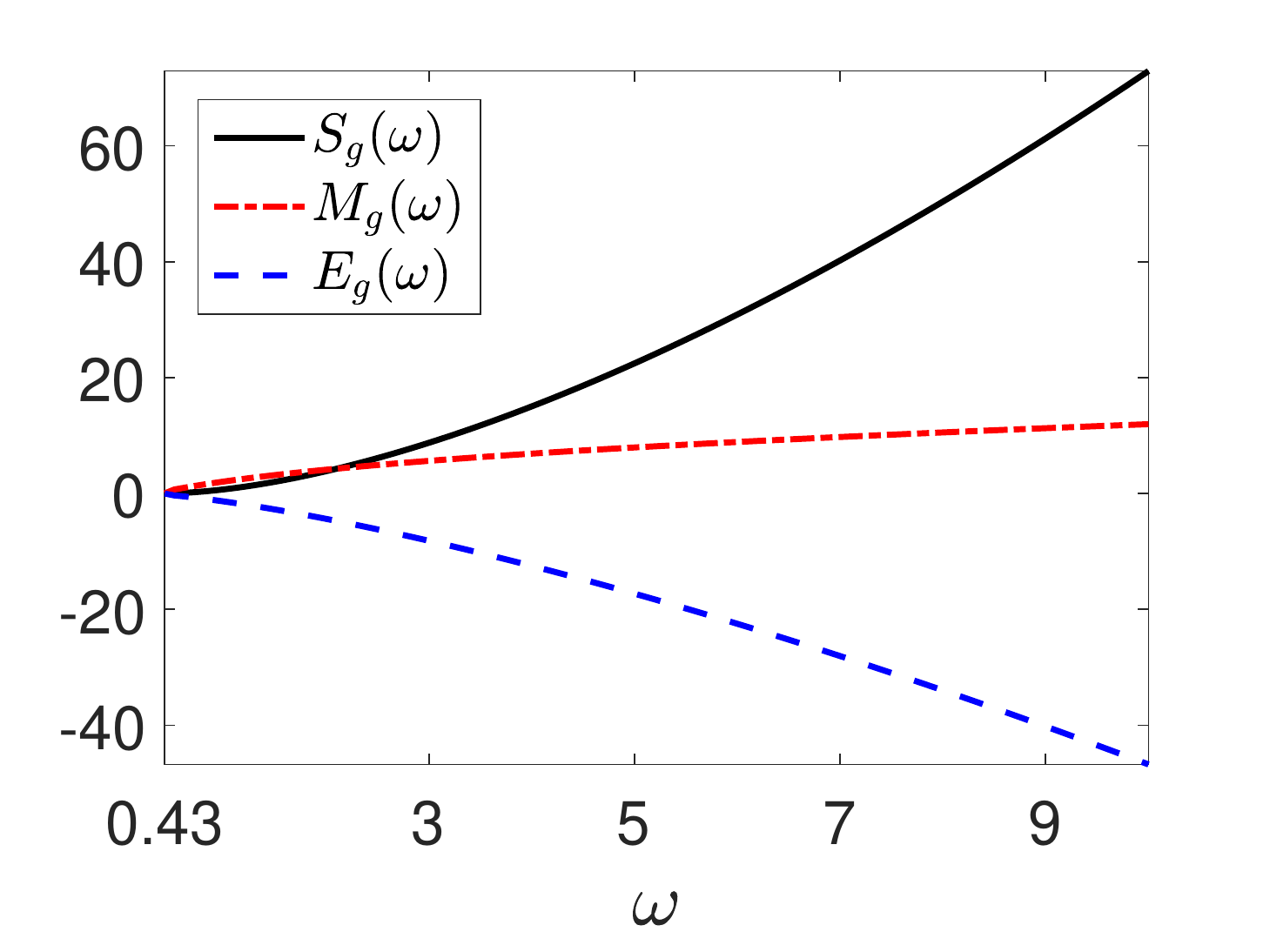}}
			\caption{density of the least action ground states with different $ \omega $ (left) and change of $ S_g $, $ M_g $ and $ E_g $ (right)}
			\label{fig:smooth_1D}
		\end{figure}
		
		For the 2D example, we consider the following potential
		\begin{equation}
			V(\vx) = -2e^{-|\vx|^2}[1-\cos(\pi x)] [1-\cos(\pi y)], \quad \vx = (x, y) \in \R^2. \label{V_1_smooth}
		\end{equation}
		This potential will also be considered in \cref{sec:action_vs_energy} as a nontrivial example to show the relationship between the least energy ground state and the least action ground state. 
		For this potential, $ \omega_0 \approx 0.4652 $ and the least energy ground state exists for any $ m > 0 $ if $ 0<\alpha<1 $  \cite{ikoma2020}. 
		Then we show the least action ground states with $ \omega =1, 2, 3 $, which are also least energy ground states with $ m \approx 4.78, 26.88, 46.80 $ considered in \cref{sec:action_vs_energy}. In computation, we choose
		\begin{equation*}
			V \text{ in } \cref{V_1_smooth} , \quad \alpha = \frac{1}{2}, \quad \Omega = [-16, 16]^2, \quad \tau = 1, \quad h = \frac{1}{2^5}. 
		\end{equation*}
		We use sine pseudospectral method for spatial discretization. Density of the least action ground states with $ \omega = 1, 2, 3 $ and change of ground state action, mass and energy with respect to $ \omega $ are shown in \cref{fig:smooth_ground_states}. Similarly, for this potential, we can observe the symmetry breaking bifurcation. 
		\begin{figure}[htbp]
			\centering
			\subfloat{\includegraphics[width=0.325\textwidth]{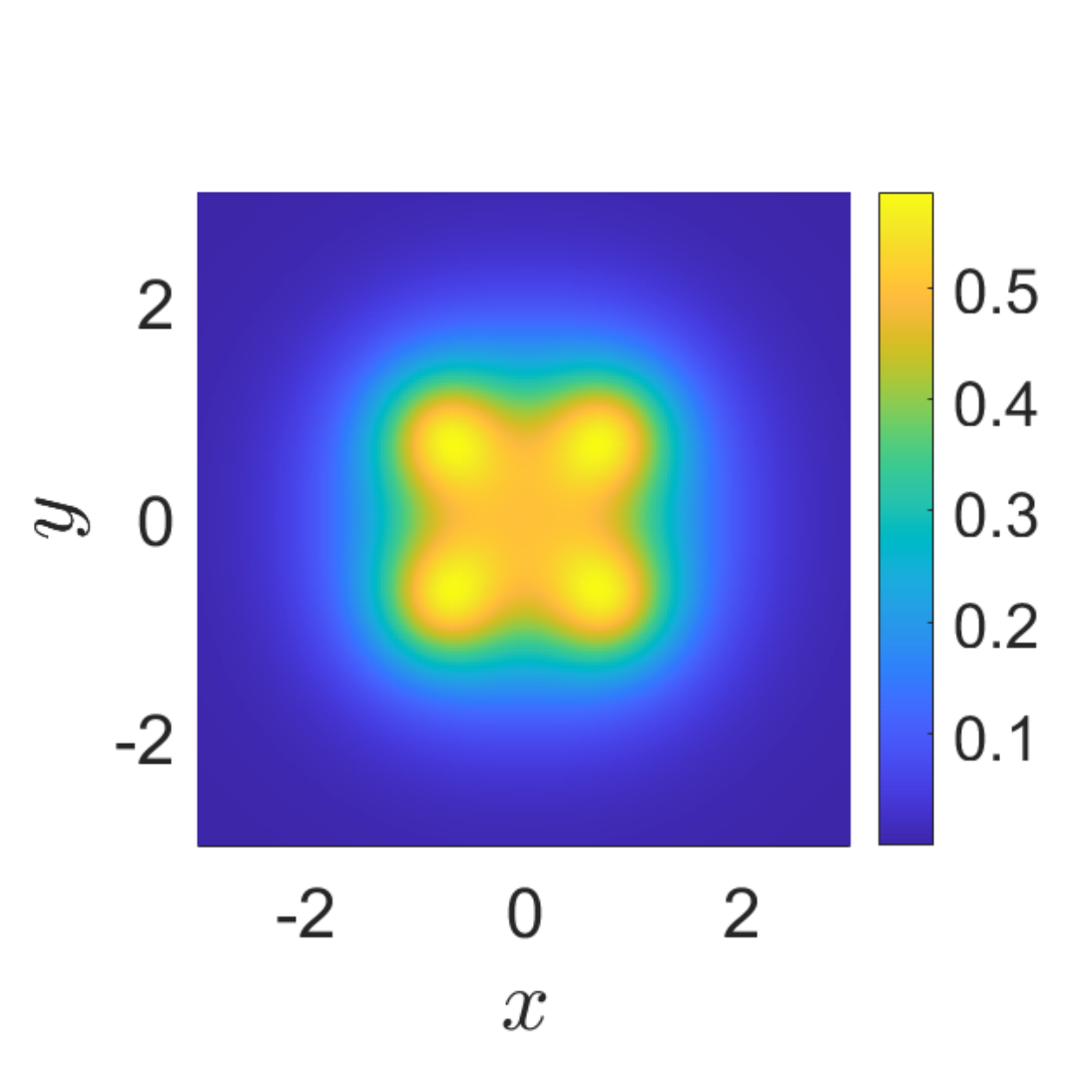}}
			\subfloat{\includegraphics[width=0.325\textwidth]{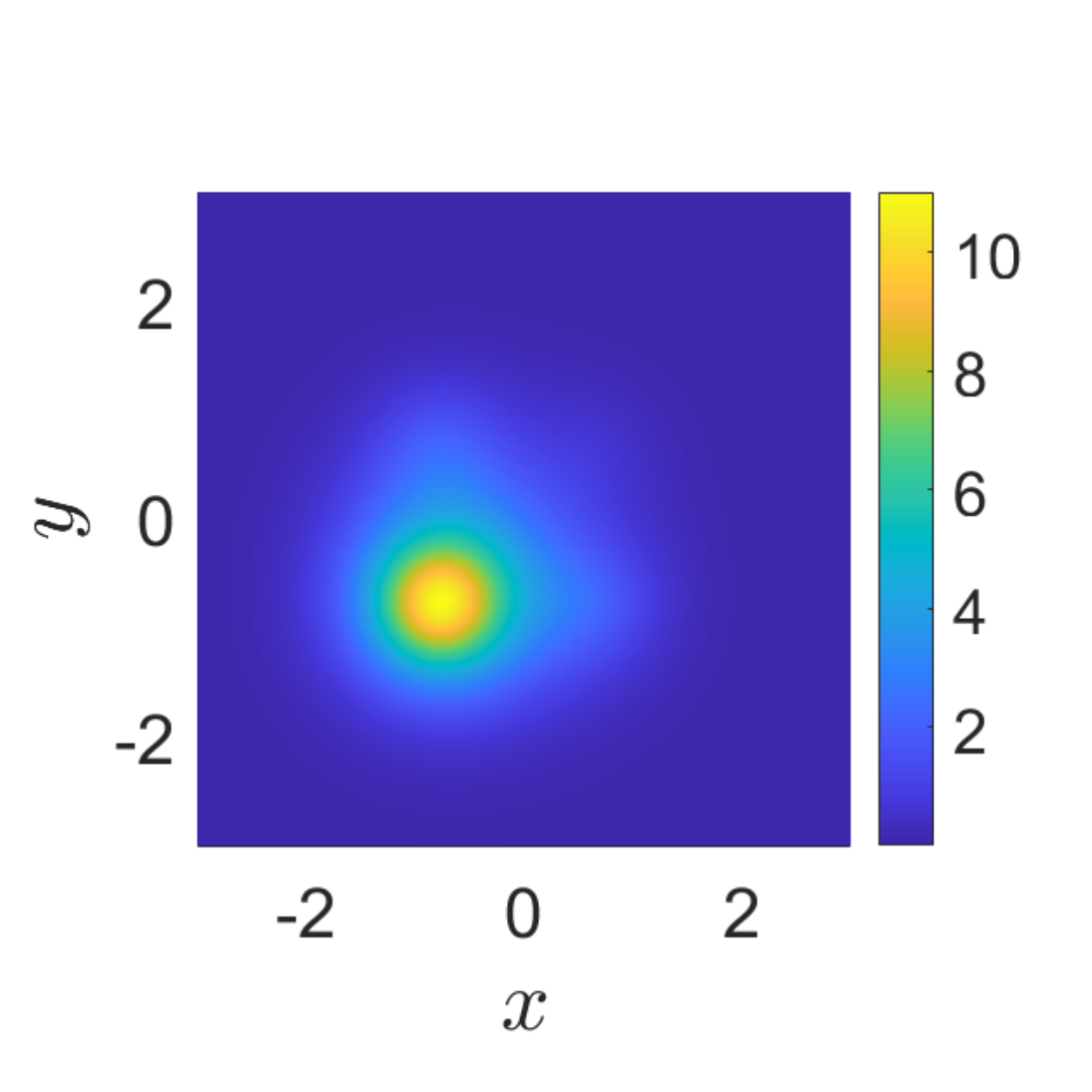}}
			\subfloat{\includegraphics[width=0.325\textwidth]{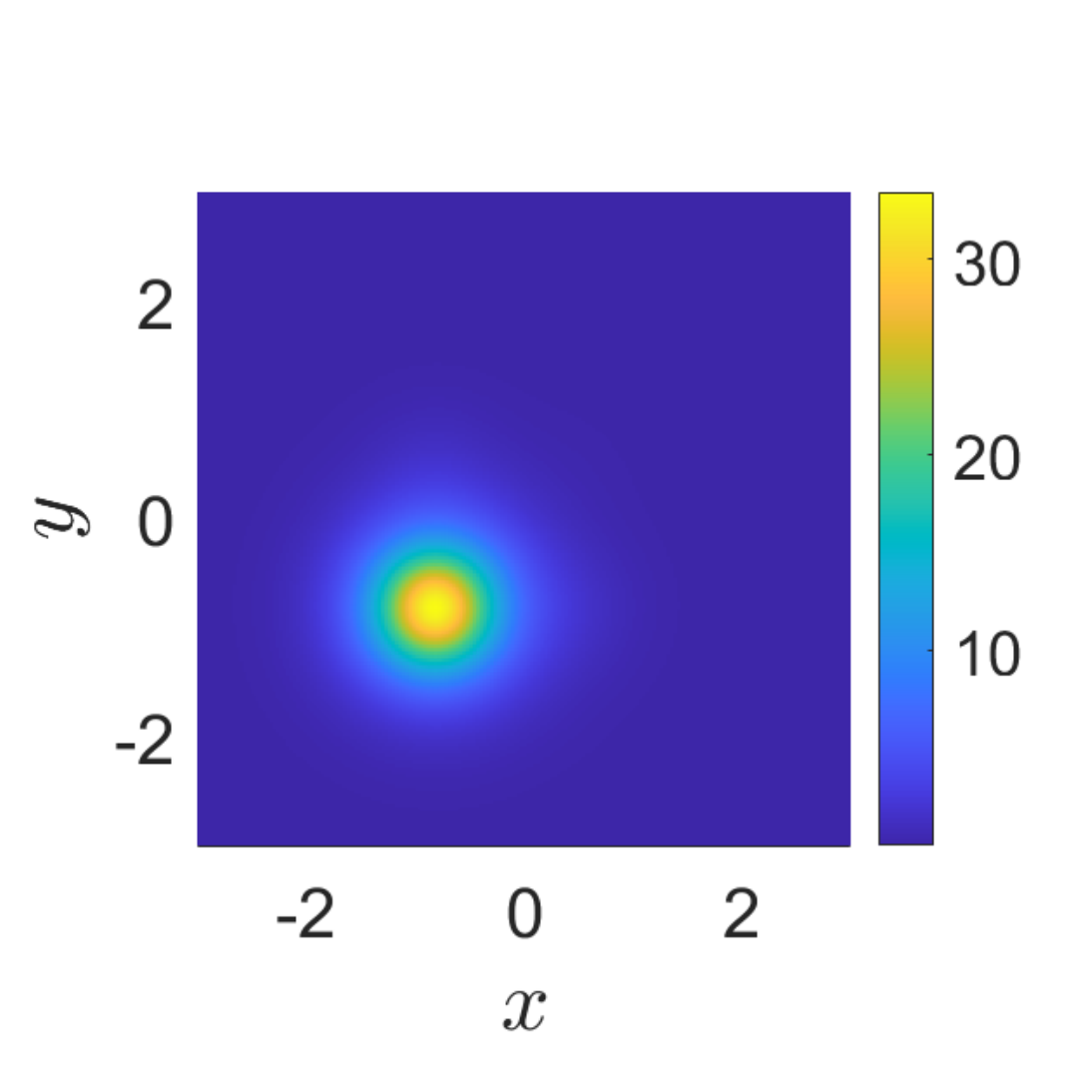}}\\\vspace{-0.35cm}
			\subfloat{\includegraphics[width=0.475\textwidth]{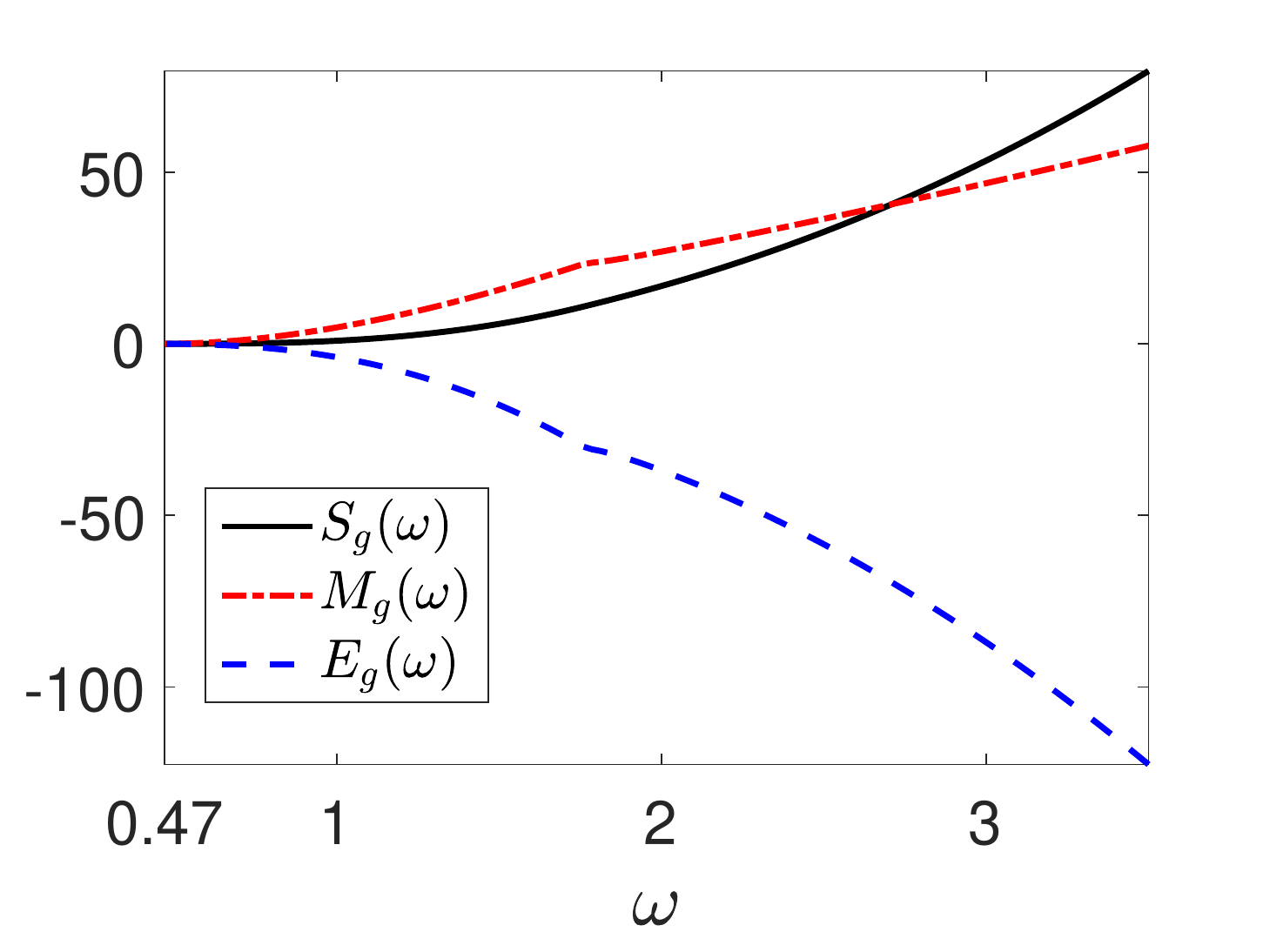}}\vspace{-0.35cm}
			\caption{density of the least action ground states with $ \omega = 1 $ (top left), $ \omega = 2 $ (top middle), $ \omega = 3 $ (top right) and change of $ S_g $, $ M_g $ and $ E_g $ (bottom)}
			\label{fig:smooth_ground_states}
		\end{figure}
	\end{exmp}
	
	\subsection{Verification of existing results and application to some conjectures}\label{verification}
	In this subsection, we first use our algorithm to verify the relationship between the least energy ground state and the least action ground state which is proved very recently in \cite{dovetta2021}. We also make a conjecture about the orbital stability. 
	
	\subsubsection{Relationship between the least energy ground state and the least action ground state}\label{sec:action_vs_energy}
	In this example, we shall verify the relationship between the least action ground state and the least energy ground state obtained in \cite{dovetta2021}. 
	By Theorem 1.3 in \cite{dovetta2021}, if $ \phi_m^E $ is a least energy ground state with prescribed mass $ m $ and associated Lagrangian multiplier $ \mu^g = \mu^g(\phi^E_m) $, then
	\begin{enumerate}
		\item[i)] $ \phi^E_m $ will be a least action ground state of \cref{SNLS} with $ \omega = - \mu^g $; 
		\item[ii)] Any least action ground states of \cref{SNLS} with $ \omega = - \mu^g $ have mass $ m $ and are also least energy ground states. 
	\end{enumerate}
	Even though the results in \cite{dovetta2021} are proved for the non-potential case on general domains, they are still valid if the variational characterization \cref{minimization} by means of Nehari manifold is valid, which, in particular, implies that their results are valid in our setting. 
	
	To verify this result, we choose the smooth potential \cref{V_1_smooth} in \cref{exmp:smooth}. Recall that for this potential, we have $ \omega_0 \approx 0.4652 $ and the least energy ground state exists for all $ m > 0 $. 
	Then we shall compare the least energy ground states and the least action ground states computed by our algorithm. We consider the case $ \alpha = 1/2 $, which is $ L^2 $-subcritical in 2D. In detail, we choose a sequence of $ m $ from $ 0.1 $ to $ 100 $ and compute the least energy ground states $ \phi^E_m $ as well as the Lagrangian multiplier $ \mu^g(m) = \mu^g(\phi^E_m) $ by the algorithm given in \cite{bao2006,liu2021}. Then we compute the least action ground state $ \phi_\omega $ with $ \omega = -\mu^g $ and the corresponding mass $ M_g(\omega) = M(\phi_\omega) $. We found that $ -\mu^g(m) $ as a function of $ m $ and $ M_g(\omega) $ as a function of $ \omega $ are all monotonically increasing and $ \omega = -\mu^g(m) $ is the same as the inverse function of $ m = M_g(\omega) $ (see \cref{fig:action_energy_diff}). This suggests that the least energy ground states are also the least action ground states. 
	
	Moreover, the numerical result shows that $ m \rightarrow -\mu^g(m) $ seems to be a bijection between $ (0, \infty) $ and $ (\omega_0, \infty) $. If this indeed holds, then we can conclude from i) and ii) above that the least action ground state $ \phi_\omega $ for any $ \omega > \omega_0 $ will also be a least energy ground state with $ m = \| \phi_\omega \|^2_{L^2} $. 
	\begin{figure}
		\centering\vspace{-0.5cm}
		\subfloat{\includegraphics[width=0.5\textwidth]{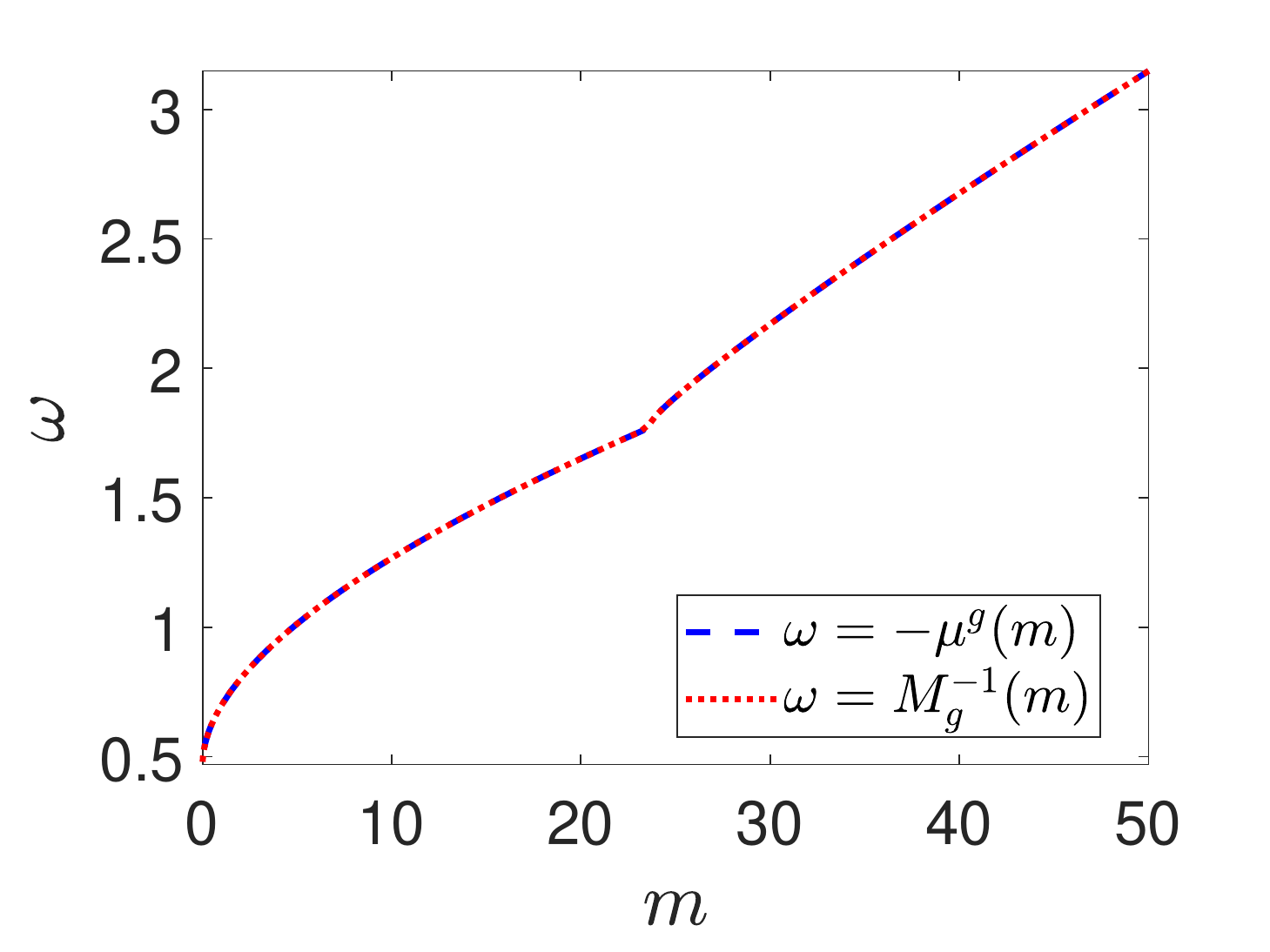}}\vspace{-0.35cm}
		\caption{the plot of $ m \rightarrow -\mu^g(m) $ and the inverse function of $ \omega \rightarrow M_g(\omega) $}
		\label{fig:action_energy_diff}\vspace{-0.5cm}
	\end{figure}
	
		
	
	\subsubsection{Orbital stability and instability}\label{sec:stability}
	In this example, we shall examine the orbital stability and instability of standing waves when $ V $ is the attractive inverse power potential \cref{inver_power_potential_in_example} considered in \cref{exmp:inverse_power}. Thanks to the result in \cite{fukaya2021}, we can use the criterion \ref{cri1}-\ref{cri2} due to Grillakis, Shatah and Strauss. To our knowledge, no complete answer is known and some partial results can be found in \cite{fukaya2021,fukuizumi2003,li2020}. We shall use our numerical results to make a conjecture. 
	\begin{definition}\label{def:stability}
		We say that $ e^{i \omega t} \phi(\vx) $ is (orbitally) stable if for any $ \vep>0 $, there exists $ \delta > 0 $ such that for any $ \phi_0 \in H^1(\R^d) $ satisfying
		\begin{equation*}
			\inf_{\theta \in \R} \| \phi_0 - e^{i \theta} \phi \|_{H^1(\R^d)} < \delta, 
		\end{equation*}
		the solution $ \psi $ of \cref{NLSE} with initial data $ \psi(\cdot, 0) = \phi_0 $ satisfies
		\begin{equation*}
			\inf_{\theta \in \R} \| \psi(\cdot, t) - e^{i \theta} \phi \|_{H^1(\R^d)} < \vep 
		\end{equation*}
		for any $ t \geq 0 $. Otherwise, $ e^{i \omega t} \phi(\vx) $ is said to be unstable. 
	\end{definition}
	
	In \cref{inverse_power_stability}, we plot $ M_g(\omega) $ for the 1D, 2D and 3D cases respectively with three different $ \alpha $. We choose $ \gamma = 1 $ and $ \sigma = d/2 $ (similar results can be observed for other choices of $ \gamma $ and $ \sigma $). In each figure, we choose three $ \alpha $ to show the results in $ L^2 $-subcritical, $ L^2 $-critical and $ L^2 $-supercritical regimes. From our numerical results, we have the following conjecture, which is roughly the same as Theorem 1 in \cite{fukuizumi2008} for the attractive delta potential. 
	\begin{conj}\label{conjecture}
		Let $ V $ be the attractive inverse power potential in \cref{inver_power_potential_in_example} and $ \phi_\omega $ be the unique positive least action ground state. Then
		\begin{enumerate}[label = \roman*)]
			\item if $ 0<\alpha\leq 2/d $, $ e^{i\omega t}\phi_\omega(\vx) $ is stable for any $ \omega > \omega_0 $;
			\item if $ 2/d < \alpha < 2/(d-2)_+ $, there exists a unique $ \omega_\text{c}>\omega_0 $ such that $ e^{i\omega t}\phi_\omega(\vx) $ is stable when $ \omega_0<\omega<\omega_\text{c} $ and unstable when $ \omega>\omega_\text{c} $. 
		\end{enumerate}
	\end{conj}

	\begin{figure}[H]
		\centering
		\vspace{-0.35cm}
		\subfloat{\includegraphics[width=0.475\textwidth]{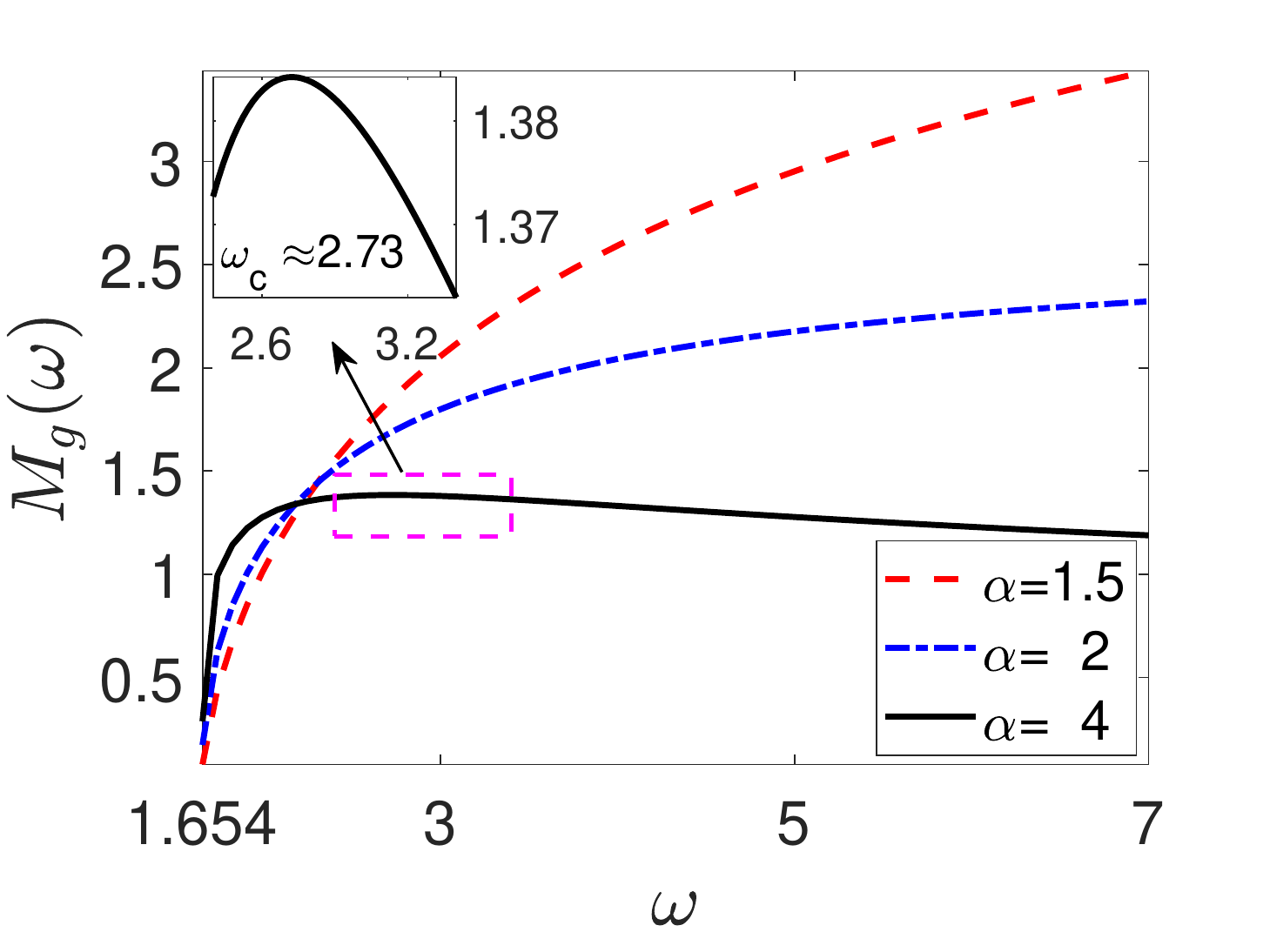}}
		\subfloat{\includegraphics[width=0.475\textwidth]{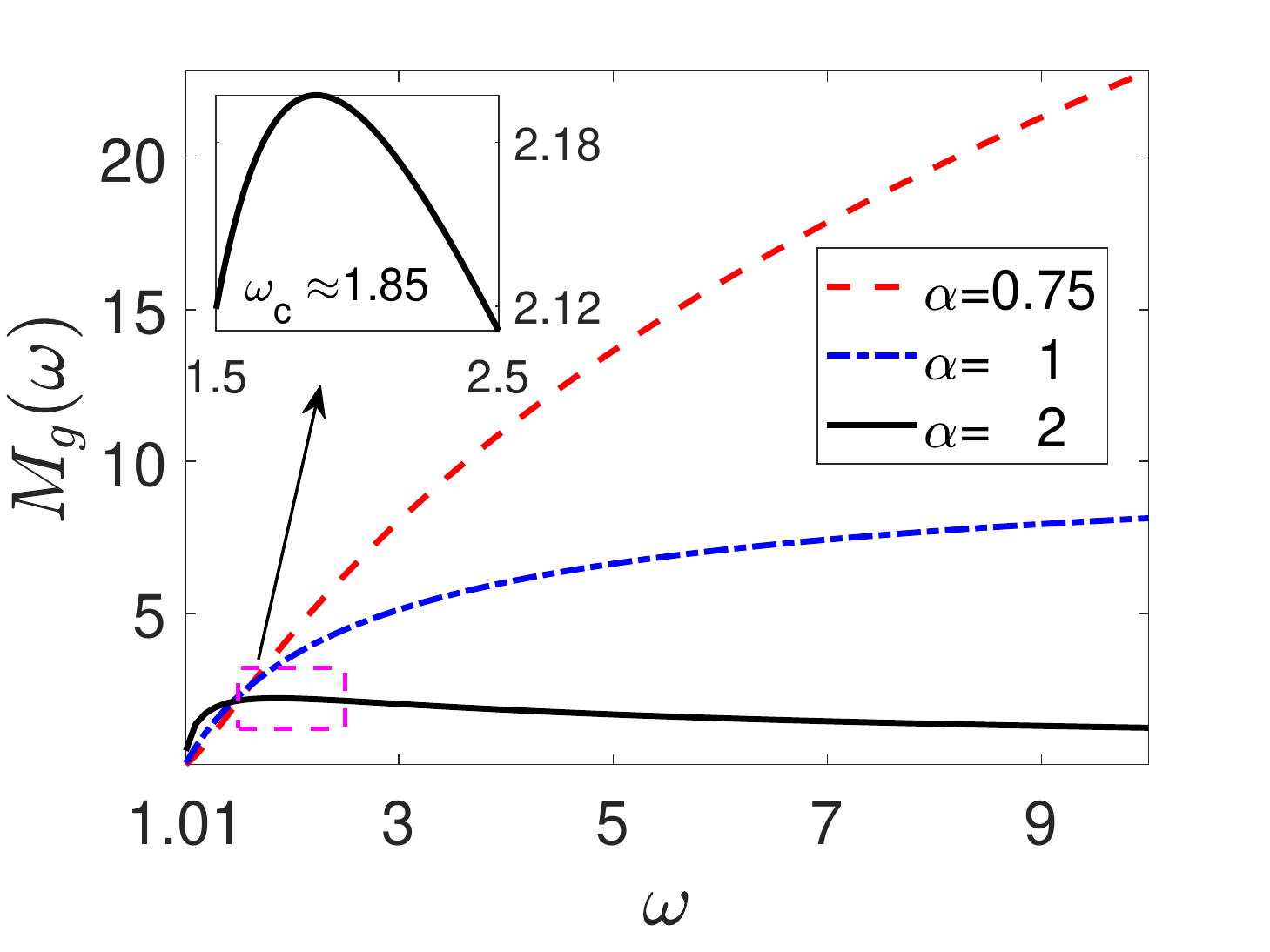}}\\
		\vspace{-0.35cm}
		\subfloat{\includegraphics[width=0.475\textwidth]{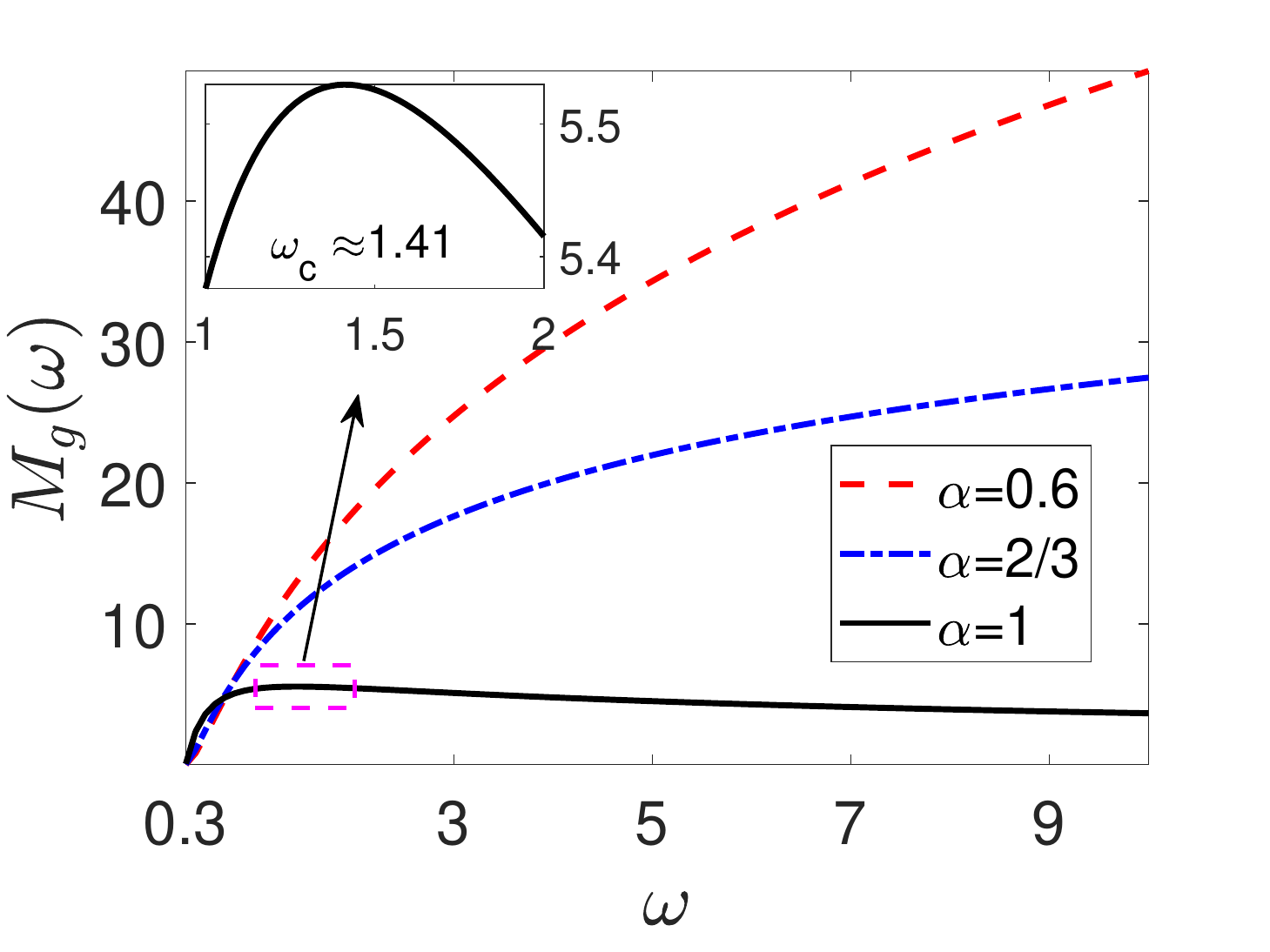}}
		\vspace{-0.35cm}
		\caption{$ M_g(\omega) $ as a function of $ \omega $ with $ V $ being the inverse power potentials in \cref{inver_power_potential_in_example} for $ d=1 $ (top left), $ d=2 $ (top right) and $ d=3 $ (bottom) }
		\label{inverse_power_stability}
	\end{figure}
	
	\section{Conclusion}
	In this paper, we study the normalized gradient flow method for computing the least action ground state of nonlinear Sch\"odinger equation with potentials. We present a continuous normalized gradient flow (CNGF) and prove that it preserves $ I_\omega $ and diminishes the action, which gives a mathematical justification of the gradient flow with discrete normalization (GFDN). Then we further discretize the GFDN in time with the backward-forward Euler method and obtain the GFDN-BF scheme. For this scheme, only a linear elliptic equation with constant coefficient (remaining unchanged) need to be solved at each step and thus it is very efficient coupled with finite difference (FD), finite element (FE) or sine pseudospectral (SP) discretization. Moreover, we show that the GFDN-BF scheme is positivity preserving and unconditionally action diminishing. We also compare the GFDN-BF with some other schemes modified from corresponding schemes used when computing the least energy ground states and find that the GFDN-BF scheme is the best. Finally, we show extensive numerical results of least action ground states with different potentials in 1D, 2D or 3D, which also lead to a few conjectures for some interesting problems in mathematics and physics. \\
	
	\noindent{\bf Acknowledgments} The author is grateful to Prof. Weizhu Bao for fruitful discussions and instructions.\\
	

\end{document}